\documentclass[twoside, 11pt]{article}
\usepackage{geometry}
\usepackage{tikz}
\usepackage{tikz-cd}
\usetikzlibrary{babel,arrows.meta,positioning,calc}

\usepackage{times}
\usepackage[utf8]{inputenc}
\usepackage[vietnamese,english]{babel}
\usepackage[
    backend=biber,
    style=alphabetic,
    sorting=anyt,
    sortlocale=vi-VN,
    maxbibnames=99,     
    maxalphanames=4,    
    minalphanames=3     
]{biblatex}
\usepackage[displaymath, tightpage]{preview}
\usepackage[all, cmtip]{xy}

\usepackage{amsmath, bm}
\usepackage{amsxtra}
\usepackage{amscd}
\usepackage{amsthm}
\usepackage{amsfonts}
\usepackage{amssymb}
\usepackage{bbm}
\usepackage{mathtools}
\usepackage{enumitem}
\usepackage{tikz}
\usepackage{tikz-cd}
\usetikzlibrary{decorations.pathmorphing}

\usepackage[bitstream-charter]{mathdesign}
\usepackage[T1]{fontenc}
\usepackage[new]{old-arrows}
\usepackage[pagewise]{lineno}

\usepackage{fancyhdr}
\usepackage{sectsty}
\sectionfont{\fontsize{13}{15}\selectfont}

\usepackage[margin=1cm]{caption}
\usepackage[hidelinks]{hyperref}
\usepackage{cleveref}
\usepackage{aliascnt}

\numberwithin{equation}{section}

\newtheoremstyle{theorem}
{3mm}
{1mm}
{\normalfont\itshape}
{}
{\normalfont\scshape}
{.}
{.5em}
{\thmname{#1}\thmnumber{ #2}\thmnote{ (#3)}}
\theoremstyle{theorem}

\newtheorem{theorem}{Theorem}
\numberwithin{theorem}{section}

\newaliascnt{lemma}{theorem}
\newtheorem{lemma}[lemma]{Lemma}
\aliascntresetthe{lemma}
\crefname{lemma}{Lemma}{Lemmas}
\Crefname{lemma}{Lemma}{Lemmas}

\newaliascnt{proposition}{theorem}
\newtheorem{proposition}[proposition]{Proposition}
\aliascntresetthe{proposition}
\crefname{proposition}{Proposition}{Propositions}
\Crefname{proposition}{Proposition}{Propositions}

\newaliascnt{corollary}{theorem}
\newtheorem{corollary}[corollary]{Corollary}
\aliascntresetthe{corollary}
\crefname{corollary}{Corollary}{Corollaries}
\Crefname{corollary}{Corollary}{Corollaries}

\newaliascnt{claim}{theorem}

\aliascntresetthe{claim}
\crefname{claim}{Claim}{Claims}
\Crefname{claim}{Claim}{Claims}

\newaliascnt{assumption}{theorem}

\aliascntresetthe{assumption}
\crefname{assumption}{Assumption}{Assumptions}
\Crefname{assumption}{Assumption}{Assumptions}

\newaliascnt{conjecture}{theorem}
\newtheorem{conjecture}[conjecture]{Conjecture}
\aliascntresetthe{conjecture}
\crefname{conjecture}{Conjecture}{Conjectures}
\Crefname{conjecture}{Conjecture}{Conjectures}

\newaliascnt{question}{theorem}

\aliascntresetthe{question}
\crefname{question}{Question}{Questions}
\Crefname{question}{Question}{Questions}

\newtheoremstyle{definition}
{3mm}
{1mm}
{\normalfont\normalfont}
{}
{\normalfont\scshape}
{.}
{.5em}
{\thmname{#1}\thmnumber{ #2}\thmnote{ (#3)}}
\theoremstyle{definition}

\newaliascnt{definition}{theorem}
\newtheorem{definition}[definition]{Definition}
\aliascntresetthe{definition}
\crefname{definition}{Definition}{Definitions}
\Crefname{definition}{Definition}{Definitions}

\newaliascnt{notation}{theorem}

\aliascntresetthe{notation}
\crefname{notation}{Notation}{Notations}
\Crefname{notation}{Notation}{Notations}

\newaliascnt{example}{theorem}
\newtheorem{example}[example]{Example}
\aliascntresetthe{example}
\crefname{example}{Example}{Examples}
\Crefname{example}{Example}{Examples}

\newaliascnt{remark}{theorem}
\newtheorem{remark}[remark]{Remark}
\aliascntresetthe{remark}
\crefname{remark}{Remark}{Remarks}
\Crefname{remark}{Remark}{Remarks}

\DeclareMathOperator{\ad}{ad}

\DeclareMathOperator{\Tr}{Tr}

\newcommand{\pa}{\partial}

\newcommand{\C}{\mathbb C}
\DeclareMathOperator{\eucl}{eucl}
\newcommand{\D}{\mathcal D}

\newcommand\dqd{\foreignlanguage{vietnamese}{Đinh Quý Dương}}

\usepackage{comment}
\DeclareUnicodeCharacter{039F}{O}

\title{Lagrangian correspondences for moduli spaces of Higgs bundles and holomorphic connections}
\author{Panagiotis Dimakis, \dqd, Shengjing Xu}
\date{}

\begin{document}
\maketitle

\begin{abstract}
On a compact connected Riemann surface $C$ of genus at least $2$,
we construct Lagrangian correspondences between moduli spaces of rank-$n$ Higgs bundles (respectively, holomorphic connections) and the Hilbert schemes of points on $T^\ast C$ (respectively, the twisted cotangent bundles of $C$).
Central to these constructions are Higgs bundles (respectively, holomorphic connections) which are transversal to line subbundles of the underlying bundles:
these naturally induce divisors on $C$ together with auxiliary parameters, namely lifts to divisors on spectral curves for Higgs bundles and residue parameters of apparent singularities for holomorphic connections.
We discuss the evidence showing that the Dolbeault geometric Langlands correspondence is generically realized by these Lagrangian correspondences; we expect that the de Rham geometric Langlands correspondence can be realized by their quantization, following Drinfeld's construction of Hecke eigensheaves. 
We also discuss the relations of our constructions to various topics, including reductions of Kapustin-Witten equations, the conformal limit, separation of variables and degenerate fields in conformal field theories. 
\end{abstract}

\tableofcontents

\section{Introduction}
Let $C$ be a compact connected Riemann surface of genus $g \geq 2$. 
The main purpose of this paper is to construct families of Lagrangian correspondences between the Hitchin moduli spaces of Higgs bundles 
(respectively, the de Rham moduli spaces of holomorphic connections) on $C$ 
and the Hilbert scheme of the complex surface $T^\ast C$
(respectively, the Hilbert scheme of the twisted cotangent bundle of $C$). 
Central to these constructions are families of Lagrangians $\mathbb{L}_H(D)$ and $\mathbb{L}_{dR}(D)$ in the respective moduli spaces induced by fixed effective divisors $D$ on $C$.
We achieve these results by studying Higgs bundles and holomorphic connections that are transversal to the line subbundles of their underlying vector bundles.
Such objects are introduced in \Cref{sect-triples}, and the main results of the paper are stated in Sections \ref{sect-intro-lag-property} and \ref{sect-intro-lag-cor}.

In \Cref{sect-intro-discussion}, we discuss the relations of our results to related topics, in particular the Dolbeault and de Rham geometric Langlands correspondence (GLC). We conjecture that the Dolbeault GLC is generically realized by the Fourier transforms induced by the Lagrangian correspondences for Hitchin moduli spaces, and the de Rham GLC is generically realized by the corresponding quantized Fourier transforms in the spirit of Drinfeld's formulation \cite{Drinfeld}. 
It is also likely that parts of the quantum GLC could be realized by quantizing the Lagrangian correspondences for de Rham moduli spaces.

\subsection{Higgs bundles and flat connections transversal to line subbundles}
\label{sect-triples}
When studying moduli problems related to bundles $E$ of rank $n$ on $C$, possibly with additional structure such as Higgs fields or holomorphic connections, keeping track of the extra data of line subbundles $L \hookrightarrow E$ has often enabled explicit parametrization of the problems \cite{Tyurin1967, yoshi97, EnriquezRubtsov2007}. 
This strategy has been most successful when $n = 2$ \cite{LN83, Hit87a};
in particular, in the first example of the geometric Langlands correspondence (GLC), Drinfeld constructed Hecke eigensheaves associated to flat bundles on $C$ by considering the auxiliary moduli spaces of pairs of rank-$2$ bundles and line subbundles \cite{Drinfeld}.
For $n > 2$, this strategy so far has been less effective and often comes with constraints on $E$ and $L$:
for example, Tyurin's parametrization of the moduli space of rank-$n$ stable bundles $E$ works for $\deg(E) = ng$ and $L = \mathcal{O}_C$ \cite{Tyurin1967}.

In this paper, our central objects are Higgs bundles $(E, \phi)$ and holomorphic connections $(E, \nabla)$ of rank-$n$ on $C$, equipped with the data of a line subbundle $L \hookrightarrow E$. 
Specifically, we consider triples of the form 
\begin{equation}\label{intro-triples}
(L \overset{i}{\hookrightarrow} E, \phi) \qquad  \text{ and } \qquad 
(L \overset{i}{\hookrightarrow} E, \nabla),
\end{equation}
satisfying some generic constraints.
These triples were considered in the context of a reduction of the Kapustin-Witten equations called extended Bogomolny equations (EBE) on $C \times \mathbb{R}_+$: they are in set-theoretic correspondence with solutions to EBE subjected to particular boundary and asymptotic conditions, as shown by He-Mazzeo \cite{HM20}. 

Following \cite{HM20}, given triples of the form \((L \overset{i}{\hookrightarrow} E, \phi)\), we consider the induced $\mathcal{O}_C$-linear morphisms
\begin{align*}
    s_i(\phi): L^n &\longrightarrow \det(E) K_C^{n\choose 2} \\ 
    f_1 \otimes f_2 \otimes f_3 \otimes \dots \otimes f_n 
    &\longmapsto  \det \left( i(f_1), \phi\circ i(f_2), \phi^2 \circ i(f_3), \dots, \phi^{n-1} \circ i(f_n) \right),
\end{align*}
where $K_C$ is the canonical line bundle of $C$. 
For $s_i(\phi)$ not identically zero, which is the case for $(E, \phi)$ stable, denote by $D_i(\phi)$ the corresponding divisors.
Similarly, given triples of the form \((L \overset{i}{\hookrightarrow} E, \nabla)\) 
we define the $\mathcal{O}_C$-linear morphisms 
\begin{align*}
    s_i(\nabla): L^n &\longrightarrow \det(E) K_C^{n\choose 2}. 
\end{align*}
The construction of $s_i(\nabla)$ is analogous to that of $s_i(\phi)$;
we refer to Section 4.1 for the detailed definition.
For $s_i(\nabla)$ not identically zero, which is the case for $(E, \nabla)$ irreducible, denote by $D_i(\nabla)$ the corresponding divisors.

We note that such sections $s_i(\nabla)$ and divisors $D_i(\nabla)$ have been considered in the context of the Riemann-Hilbert problem, where $D_i(\nabla)$ are the so-called \textit{apparent singularities}: to realize a generic monodromy representation in terms of meromorphic differential equations on $C$, one needs to include such singularities in the differential equations \cite{yoshi97, iwasaki92}. 
In this paper, we will call points $p_k < D_i(\nabla)$ \textit{apparent singularities} associated to the triples $(L \overset{i}{\hookrightarrow} E, \nabla)$. Such triples have been discussed by Gaiotto-Witten in the context of the Kapustin-Witten equations in \cite{GW11}; in that context, such triples were called \textit{opers with apparent singularities}. 
In fact, the line subbundle $L \overset{i}{\hookrightarrow} E$ equips $(E, \nabla)$ with an oper structure 
\footnote{The case of general complex connected reductive group $G$ is proved by Arinkin: any irreducible $G$-connection admits a \textit{generic oper structure} \cite{Ar16}.}
away from $D_i(\nabla)$.
In Section 4.2 of this paper, we will make this analogy precise. Given a reduced divisor $D$ on $C$, we will show that triples $(L \overset{i}{\hookrightarrow} E, \nabla)$ with $D_{i}(\nabla) = D$ are equivalent to meromorphic opers on $C$ which are tamely ramified on $D$ and have trivial monodromy around points $p < D$.

\subsection{Lagrangian subvarieties in the moduli of Higgs bundles and flat connections}
\label{sect-intro-lag-property}
Given a fixed line bundle $\Lambda$ of degree $\ell$ on $C$, we consider the moduli space $\mathcal{M}_{H}(n, \Lambda)$ of stable rank-$n$ Higgs bundles $(E, \phi)$ where the bundles $E$ have fixed determinant $\det(E) = \Lambda$ and the Higgs fields are traceless.  
Fix an effective divisor $D$ on $C$ with $\deg(D) \equiv \ell \pmod{n}$. For each such divisor we define a subvariety $\mathbb{L}_{H}(D) \subset \mathcal{M}_{H}(n, \Lambda)$ by setting 
\begin{align*}
\mathbb{L}_{H}(D) = 
\left\{ (E, \phi) \in \mathcal{M}_{H}(n, \Lambda) 
\mid \exists i: L \hookrightarrow E, \text{ } s_i(\phi) \neq 0,  \text{ } D_i(\phi) = D\right\}.
\end{align*}

Similarly, we consider the de Rham moduli space $\mathcal{M}_{dR}(n, \mathcal{O}_C)$ of irreducible $\mathrm{SL}_n$ holomorphic connections $(E, \nabla)$ where $\det(E) \cong \mathcal{O}_C$ and $\nabla$ induces the trivial connection on $\mathcal{O}_C$. 
Fix an effective divisor $D$ on $C$ with $\deg(D) \equiv 0 \pmod{n}$.
\footnote{In fact, loosening the congruence condition on $\deg(D)$ is equivalent to considering projective connections whose monodromy representations do not lift to $\mathrm{SL}_n$.}
For each such divisor $D$ we define a subvariety $\mathbb{L}_{dR}(D) \subset \mathcal{M}_{dR}(n, \mathcal{O}_C)$ by setting 
\begin{align*}
&\mathbb{L}_{dR}(D) = \left\{ (E, \nabla) \in \mathcal{M}_{dR}(n, \mathcal{O}_C)  
\mid \exists i: L \hookrightarrow E, \text{ } s_i(\nabla) \neq 0,  \text{ } D_i(\nabla) = D \right\}.
\end{align*}

Our first main results explain how these constructions define holomorphic Lagrangian subvarieties in the respective moduli spaces.

\begin{theorem}\label{intro-main-thm-1} 
(a) For any effective divisor $D$ of degree $\deg(D) \equiv \ell$ mod $n$, $\mathbb{L}_H(D)$ is non-empty and has an open dense subset which is Lagrangian in $\mathcal{M}_H(n, \Lambda)$. 
\\
(b) For any reduced effective divisor $D$ of degree divisible by $n$, $\mathbb{L}_{dR}(D)$ is non-empty and has an open dense subset which is co-isotropic in $\mathcal{M}_{dR}(n, \mathcal{O}_C)$.
For a generic reduced effective divisor $D$ of degree divisible by $n$, an open dense subset of $\mathbb{L}_{dR}(D)$ is Lagrangian in $\mathcal{M}_{dR}(n, \mathcal{O}_C)$.
\end{theorem}

Analogous constructions hold for the moduli spaces $\mathcal{M}_H(n, \ell)$ of Higgs bundles $(E, \phi)$ of rank-$n$ and with $\deg(E) = \ell$, and the moduli spaces $\mathcal{M}_{dR}(n)$ of irreducible connections $(E, \nabla)$ of rank-$n$ on $C$. In these cases, one can construct Lagrangians $\mathbb{L}(L, D) \subset \mathcal{M}_H(n, \ell)$ and $\mathbb{L}_{dR}(L, D) \subset \mathcal{M}_{dR}(n)$ by fixing both the line subbundle $L$ and the induced divisors $D_i(\phi)$ and $D_i(\nabla)$.

Special cases of \Cref{intro-main-thm-1} can be extracted from previous works in the literature. 
In the special case $n=2$ and $\deg(D) = 3g-3 = \frac{1}{2} \dim \mathcal{M}_H(2, \Lambda)$, part (a) of \Cref{intro-main-thm-1} follows from \cite{DT23} while part (b) follows from \cite{iwasaki92, pinchbeck, D24-lambda}. 
Analogues of part (a) of \Cref{intro-main-thm-1} for $\mathcal{M}_H(n, \ell)$ with $\deg(D) = n^2(g-1) + 1 = \frac{1}{2} \dim \mathcal{M}_H(n, \ell)$ and $L = \mathcal{O}_C$ follow from the main results 
\footnote{It follows from the congruence condition on $\deg(D)$ that \cite{GNR01} actually treated the case where $\ell$ is congruent to $n^2(g-1) + 1$ modulo $n$.}
in \cite{GNR01}. 
These are instances of the application of \textit{Separation of Variables},
a powerful method in the theory of integrable systems, 
to the moduli spaces of Higgs bundles and holomorphic connections:
fixing the divisors $D_i(\phi)$ and $D_i(\nabla)$ in these cases amounts to fixing half of the Darboux coordinates or \textit{separated variables}.
On the other hand, for $\deg(D)$ sufficiently small,  
$\mathbb{L}_H(D)$ and $\mathbb{L}_{dR}(D)$ are precisely the intersections with the fibers over $\lambda = 0$ and $\lambda = 1$, respectively, of upward flows to Hodge bundles induced by the $\mathbb{C}^\ast$ action on the Hodge moduli space of $\lambda$-connections (cf. \Cref{sect-CL}). 
These subvarieties were shown to be Lagrangians in the respective moduli spaces by Hausel-Hitchin \cite{HH22} and Simpson \cite{Simpson08}. 

These instances suggested that subvarieties defined by fixing the divisors $D_i(\phi)$ and $D_i(\nabla)$ might be Lagrangians for a larger range of divisors, perhaps for any divisors of any degree congruent to $\ell = \deg(E)$ modulo $n$.
It was in \cite{Dim24} that this possibility was first checked, 
where the first author proved part (a) of \Cref{intro-main-thm-1} for the case $n = 2$;
the technical challenge there from the perspective of extended Bogomolny equations was to show that $\mathbb{L}_H(D)$ is non-empty for any $D$. 
In hindsight, part (a) even for general $n$ follows readily from the spectral correspondence and Hecke transforms of Higgs bundles (more specifically, Hecke transforms of Higgs bundles that belong to a Hitchin section), and it is the de Rham analogue in part (b) of \Cref{intro-main-thm-1} that is more technically involved. It is however the extended Bogomolny equations perspective that highlights the importance of triples of the form \eqref{intro-triples} with transversality conditions and ultimately led the authors to the statement of \Cref{intro-main-thm-1}.  

In the special case where the sections $s_i(\phi)$ and $s_i(\nabla)$ are constant sections of the trivial line bundle $\mathcal{O}_C$, i.e. $D_i(\phi) = D_i(\nabla)=\varnothing$, the Lagrangians $\mathbb{L}_H(\varnothing)$ and $\mathbb{L}_{dR}(\varnothing)$ are respectively the Hitchin sections and the spaces of holomorphic opers (i.e. opers \textit{without} apparent singularities). Such Lagrangians are also the lowest strata defined by the $\mathbb{C}^\ast$ action and the induced Bia\l ynicki-Birula stratification on the Hodge moduli space. The Lagrangians $\mathbb{L}_H(D)$ and $\mathbb{L}_{dR}(D)$ for general $D$ can be regarded as ``excitations'' of these lowest strata. While $\mathbb{L}_H(D)$ can essentially be obtained by applying Hecke transforms to the Hitchin section $\mathbb{L}_H(\varnothing)$, there seems to be no obvious geometric mechanism to ``excite'' the space $\mathbb{L}_{dR}(\varnothing)$ of holomorphic opers into  $\mathbb{L}_{dR}(D)$. 

\begin{figure}[h!]
    \centering
    \includegraphics[width=0.87\linewidth]{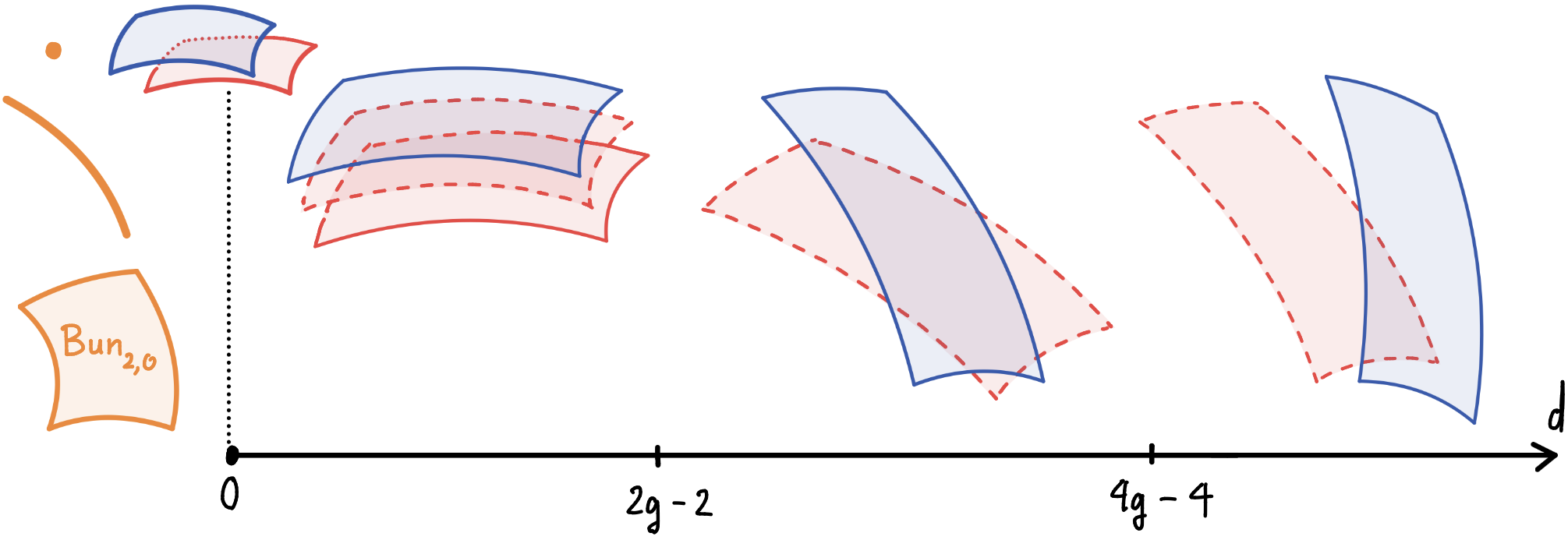}
    \caption{Illustration of Lagrangians $\textcolor{red}{\mathbb{L}_H(D)} \subset \mathcal{M}_H(2, \mathcal{O}_C)$ and $\textcolor{blue}{\mathbb{L}_{dR}(D)} \subset \mathcal{M}_{dR}(2, \mathcal{O}_C)$ as $d = \deg(D)$ varies. Shapes with dashed boundaries indicate subvarieties that are not closed in the moduli spaces. 
    Projecting to the components of the space $\mathcal{M}_H(2, \mathcal{O}_C)^{\mathbb{C}^\ast}$ of Hodge bundles corresponds to projecting to the \textcolor{orange}{orange shapes} on the left vertical.    
    As $d$ increases, the images of these projections increase dimension; in particular, the projections from $\mathbb{L}_{dR}(D)$ are dominant for generic $D$ of degree $d \geq 4g-4$. See \Cref{sect-Lag-rk2} for details.}
    \label{intro-fig-lag}
\end{figure}

\vspace{-5pt}
We remark that for a general complex reductive group $G$, one can similarly construct Lagrangians as in \Cref{intro-main-thm-1}. For example, one can again consider the Hitchin section in the moduli space $\mathcal{M}_H(G)$ of $G$-Higgs bundles: the Hecke transform of Higgs bundles in the Hitchin section should again, upon appropriate normalization, define a Lagrangian in $\mathcal{M}_H(G)$.

\subsection{Lagrangian correspondences to (twisted) Hilbert schemes}
\label{sect-intro-lag-cor}
Our second set of main results provides constructions of Lagrangian correspondences between the moduli spaces of Higgs bundles (respectively, holomorphic connections) and Hilbert schemes of the surface $T^\ast C$ (respectively, the Hilbert scheme of the twisted cotangent bundle of $C$).

In the case of Higgs bundles, the construction of Lagrangian correspondences follows from the following: 
given a triple $(L \overset{i}{\hookrightarrow} E, \phi)$ where the spectral curve $\widetilde{C} \overset{\pi}{\rightarrow} C$ associated to $(E, \phi)$ is smooth,
the spectral line bundle $\mathcal{L}_{(E, \phi)}$ can be canonically written as
\begin{subequations}\label{intro-Dtilde}
\begin{equation}
    \mathcal{L}_{(E, \phi)} \cong \pi^\ast(L) \otimes \mathcal{O}_{\widetilde{C}}(\widetilde{D}_i(\phi)),  
\end{equation}
for some effective divisor $\widetilde{D}_i(\phi)$ on $\widetilde{C}$.
Indeed, if $\widetilde{C}$ is smooth then the composition
\begin{equation}
    \pi^\ast(L) \longrightarrow \pi^\ast(E) \longrightarrow \mathcal{L}_{(E, \phi)}
\end{equation}
\end{subequations}
is nonzero; in this case, we consider the corresponding zero divisor $\widetilde{D}_i(\phi)$ on $\widetilde{C}$ and check that the norm $\pi_\ast(\widetilde{D}_i(\phi))$ coincides with $D_i(\phi)$ (cf. \Cref{prop-Nm(D)}).
Such divisors $\widetilde{D}_i(\phi)$ have been considered by Hitchin in his seminal paper introducing the moduli spaces of Higgs bundles \cite{Hit87a}, recently discussed by Hausel-Hitchin in their study of  \textit{wobbly Higgs bundles} \cite{HH22}, and are called Baker-Akhiezer divisors in \cite{DT23} due to their relation to Baker-Akhiezer functions in the theory of integrable systems. 
They define points $\widetilde{D}$ in the Hilbert scheme $\mathrm{Hilb}^d(T^\ast C)$ via the embedding $\widetilde{C} \hookrightarrow T^\ast C$.
Given $d \equiv \ell$ mod $n$, consider the subscheme  
\begin{align*}
\mathbb{L}_{H}(d) = \left\{
    \left( \widetilde{D}, (E, \phi) \right) \in \mathrm{Hilb}^d(T^*C) \times \mathcal{M}_{H}(n, \Lambda) \,\middle|\,
    \begin{array}{l}
    \text{$(E,\phi)$ has smooth spectral curve;}\\
    \text{there exist $L \stackrel{i}{\hookrightarrow} E$ with $\deg(\widetilde{D}_{i}(\phi)) = d$}\\
    \text{and such that $\widetilde{D}_{i}(\phi)=\widetilde{D}$}
    \end{array}
\right\}
\end{align*}
of $\mathrm{Hilb}^d(T^*C) \times \mathcal{M}_{H}(n, \Lambda)$.

\begin{theorem}\label{intro-main-result-lag-cor-Higgs}
$\mathbb{L}_H(d)$ defines a Lagrangian correspondence 
\[
\xymatrix{
& \mathbb{L}_H(d) 
  \ar[dl]_{\mathrm{pr}_1} \ar[dr]^{\mathrm{pr}_2} & \\
\operatorname{Hilb}^d(T^*C) && \mathcal{M}_{H}(n, \Lambda)
}
\]
i.e. $\mathbb{L}_H(d)$ is a Lagrangian subscheme of $\mathrm{Hilb}^d(T^\ast C) \times \mathcal{M}_H(n, \Lambda)$. 
\end{theorem}

To construct Lagrangian correspondences for the moduli spaces of holomorphic connections, we will need the notion of residue parameters of apparent singularities.
For $n = 2$, these parameters were discussed in  \cite{iwasaki92, pinchbeck, D24-lambda} in the cases of Riemann surfaces of genus $g \geq 2$ (possibly with marked points), 
and, for general $n$, by Dubrovin-Mazzocco in a local sense  \cite{DM07}.  
We refer to Section 4.3 for the detailed definition and discussion of residue parameters of the apparent singularities induced by triples $(L \overset{i}{\hookrightarrow} E, \nabla)$ on $C$.
The key point is that, given a triple $(L \overset{i}{\hookrightarrow} E, \nabla)$ with a reduced divisor $D_i(\nabla) = \sum_{k=1}^d p_k$, the apparent singularities $p_k$ together with their residue parameters $\nu_k$ define points $\widetilde{D}_i(\nabla) \in \mathrm{Hilb}^d(\mathcal{S})$, where $\mathcal{S}$ is the total space of an affine bundle modelled on $T^\ast C$.
One shall regard points on $\mathcal{S}$ as deformation of points on $T^\ast C$; this is compatible with the usual point of view on holomorphic connections as deformations of Higgs bundles.
Given a natural number $d$ divisible by $n$, consider the subscheme
\begin{align*}
\mathbb{L}_{dR}(d) = \left\{
    \left( \widetilde{D}, (E, \nabla) \right) \in \mathrm{Hilb}^d(\mathcal{S}) \times \mathcal{M}_{dR}(n, \mathcal{O}_C)  \,\middle|\,
    \begin{array}{l}
        \text{there exist $L \overset{i}{\hookrightarrow} E$ such that 
        $\widetilde{D}_i(\nabla) = \widetilde{D}$,} \\
        \deg(D_i(\nabla)) = d,\ \text{and $D_i(\nabla)$ is reduced}
    \end{array}
\right\}
\end{align*}
of $\mathrm{Hilb}^d(\mathcal{S}) \times \mathcal{M}_{dR}(n, \mathcal{O}_C) $.

\begin{theorem}\label{intro-main-result-lag-cor-dR}
$\mathbb{L}_{dR}(d)$ defines a Lagrangian correspondence  
\[
\xymatrix{
& \mathbb{L}_{dR}(d) 
  \ar[dl]_{\mathrm{pr}_1} \ar[dr]^{\mathrm{pr}_2} & \\
\operatorname{Hilb}^d(\mathcal{S}) && \mathcal{M}_{dR}(n, \mathcal{O}_C) 
} 
\]
i.e. $\mathbb{L}_{dR}(d)$ is a Lagrangian subscheme of $\mathrm{Hilb}^d(\mathcal{S}) \times \mathcal{M}_{dR}(n, \mathcal{O}_C) $.
\end{theorem}

We note that by fixing the line subbundle $L$, we can similarly construct Lagrangian correspondences between $\mathcal{M}_H(n, \ell)$ (respectively, $\mathcal{M}_{dR}(n)$) and $\mathrm{Hilb}^d(T^\ast C)$ (respectively, $\mathrm{Hilb}^d(\mathcal{S})$).

\subsection{Discussion}
\label{sect-intro-discussion}
We now discuss the relation between our main results and other topics. We do not pursue these connections in this paper and leave these explorations for future works.  

\subsubsection{Relation to reductions of Kapustin-Witten equations}
In this subsection we discuss how elements of $\mathbb L_H(D)$ and $\mathbb L_{dR}(D)$ naturally arise from a particular set of solutions to the dimensional reduction of the family of Kapustin-Witten (KW) equations. 

Given a $4$-manifold $M$ equipped with a Riemannian metric $g_M$ and a complex rank $n$ bundle $E$ over $M$ equipped with a Hermitian metric $H$, let $A$ denote a unitary connection on $E$ and $\Phi$ and $\ad(E)$-valued $1$-form. Then the family of KW equations takes the form 

\begin{equation}
e^{i\theta}F_{\mathcal A} = \overline{\star e^{i\theta}F_{\mathcal A}}
\end{equation}

where $\mathcal A = A+i\Phi$ is a complex connection, $F_{\mathcal A}$ is its curvature, $\star$ denotes the Hodge star operator and $\theta\in [0,\pi/2]$. The usual KW equations correspond to $\theta = \pi/4$. In \cite{Kapustin-Witten06, GW11} it was observed that when $M \cong \mathbb R_{x_1}\times C\times \mathbb R_{+,y}$, with product metric $g_M = \eucl_{x_1}\oplus g_C\oplus \eucl_{y}$ then solutions invariant in the $x_1$-direction which further satisfy the condition $A_1-\tan(\beta)\Phi_1 = 0$ with $\tan(\theta) = \tan(\pi/4-3\beta/2)$ satisfy a different set of equations which have a Hermitian Yang-Mills structure, namely the family of the extended Bogomolny equations. With respect to the complex structure on $C$ given by the conformal class of $g_C$, they have the following form:

\begin{equation}\label{EBE}
\begin{split}
[\D_i,\D_j] = 0 \text{ for } i,j\in \{1,2,3\}&\\
\frac{i}{2}\Lambda([\D_1,\D_1^{\dagger_H}] + [\D_2,\D_2^{\dagger_H}]) + [\D_3,\D_3^{\dagger_H}] &= 0
\end{split}
\end{equation}
where $\forall y>0$, $\sec(\beta)\D_1$ restricted to the bundle $E|_{C\times\{y\}}$ is a Dolbeault operator, $\csc(\beta)\D_2$ is a holomorphic connection for the same bundle, $\D_3 = \pa_y +A_y - i\Phi_1$ is a parallel transport operator in the $y$ direction, $H$ is a given background Hermitian metric on $E$ and $\Lambda$ denotes the contraction with the K\"ahler form. 

\begin{remark}
Notice that the existence of $\D_3$ implies that $E$ is a pullback of a vector bundle $V$ on $C$ with fibers over different values of $y$ identified by the parallel transport operator. Thus without loss of generality, we can also require the background Hermitian metric $H$ to be independent of $y$. 
\end{remark}

Therefore, a solution to the EB equations amounts to finding $\D_1,\D_2,\D_3$ satisfying \eqref{EBE}. In order to find solutions, one has to impose boundary and asymptotic conditions as $y\to 0$ and $y\to \infty$ respectively. As $y\to \infty$ the physically motivated requirement is that the solution converges to a $y$-independent irreducible $\lambda$-connection $(\sec(\beta)\D_1|_{y = \infty}, \tan(\beta)\D_2|_{y=\infty})$ for $\lambda = \tan(\beta)$ and as $y\to 0$ it becomes asymptotic to a singular solution which has been studied in detail in \cite{Witten2011,GW11,MW2014,Mikhaylov2012,MW2017,Dim22} in various degrees of generality. In the simplest possible case when $\beta=0$ and the rank of the bundle is $n=2$, this condition requires that in an open half ball $B_p$ around a point $p\in C\times\{0\}$ 
\begin{equation}
\D_3 = \pa_y+ \frac{1}{y}\begin{pmatrix}1 & 0\\ 0 & -1  \end{pmatrix} + \mathcal O(y^{-1+\epsilon})
\end{equation}
for some small $\epsilon$ in some local holomorphic frame $s_1,s_2$ for $E$. For a local holomorphic section $s = a_1s_1 + a_2s_2$ of $E|_{y_0}$ consider the solution to the parallel transport equation $\D_3s(y) = 0$ with initial condition $s(y_0) = s$. Then 
\begin{equation}
s(y) =  (a_1y^{-1} + \mathcal O(y^{-1+\epsilon}))s_1 + (a_2y + \mathcal O(y^{1+\epsilon}))s_2 \text{ as } y\to 0. 
\end{equation}
The span of the section $s_2$ is distinguished because its parallel transport vanishes as $y\to 0$. Gluing these local ``small'' sections leads to the invariant description of the \textit{vanishing line bundle}

\begin{equation}
L:= \{\Gamma(E) : \D_3s= 0, ~\lim\limits_{y\to 0} |y^{-1+a}s| = 0\}
\end{equation}

for any $0 <a<1$. If $p$ is not a divisor point then in the same local holomorphic frame 
\begin{equation}
\D_2 = \frac{1}{y}\begin{pmatrix} 0&1\\0&0\end{pmatrix}\,dz + \mathcal O(y^{-1+\epsilon})
\end{equation}
and we can define the line bundle $L_1 = \D_2(L)\otimes K^{-1}$ which is isomorphic to $L$. The two line bundles are independent away from the divisor $D$. In a neighborhood of a reduced point $p$ on the divisor, $\D_2$ is of the form 
\begin{equation}
\D_2 = \frac{1}{y}\begin{pmatrix} 0&z\\0&0\end{pmatrix}\,dz + \mathcal O(y^{-1+\epsilon})
\end{equation}
and therefore the sections defining $L$ and $L_1$ become linearly dependent at $p$. Globally this gives a map 
\begin{equation}
1\wedge \D_2: L^2\to K 
\end{equation}
whose zeroes are precisely the divisor $D$. When $\beta = 0$ this is precisely the condition that the zeroes of $1\wedge \Phi :L^2\to K$ give rise to the divisor $D$ while for $\beta\neq 0$ the equivalent boundary condition leads to the condition that the zeroes of the divisor are given by the zeroes of the map $1\wedge \D_2:L^2\to K$. This picture can be generalized to arbitrary rank and for divisors of arbitrary degree and this led to the conjecture that the submanifolds $\mathbb L_H(D)$ and $\mathbb L_{dR}(D)$ are holomorphic Lagrangian within $M_H$ and $M_{dR}$ respectively. Another physically meaningful question to ask is whether it is possible to describe the moduli space of the equations to \eqref{EBE} when the divisor is allowed to move. We partially answer this question by showing that triples with varying divisor give a Lagrangian correspondence in $\mathrm{Hilb}^d(T^\ast C) \times \mathcal{M}_H(n, \Lambda)$ for $\beta=0$ and in $\mathrm{Hilb}^d(\mathcal{S}) \times \mathcal{M}_{dR}(n, \mathcal{O}_C)$ when $\beta\neq 0$.

\subsubsection{Relation to the Dolbeault geometric Langlands correspondence}
The de Rham geometric Langlands correspondence (de Rham GLC) is a categorical equivalence between $\mathcal{D}$-modules on the moduli stack $\mathrm{Bun}_G$ of $G$-bundles on $C$ and coherent sheaves 
\footnote{Subjected to \textit{nilpotent singular support} conditions.}
on the moduli stack of flat bundles on $C$ 
with gauge group being the Langlands dual $\check{G}$ \cite{BD92, GLC-proof}.
Generically its classical limit is an equivalence of derived categories 
\begin{equation}\label{Dolbeault-GLC}
D^b\left( \mathcal{M}_H(G) \right) \cong D^b\left( \mathcal{M}_H(\check{G}) \right).    
\end{equation}
Here $\mathcal{M}_H(G)$ is the moduli space of principal $G$-Higgs bundles on $C$.
We call this the Dolbeault geometric Langlands correspondence due to the symmetric presence of the (Dolbeault) moduli of Higgs bundles on both sides of the correspondence. 
In Donagi-Pantev's first formulation of the correspondence, they showed that for any reductive group $G$, $\mathcal{M}_H(G)$ and $\mathcal{M}_H(\check{G})$ share the same Hitchin base with generic fibers being torsors over dual abelian varieties: in this case \eqref{Dolbeault-GLC} reduces to the twisted relative Fourier-Mukai transforms associated to the dual families of abelian varieties \cite{DonagiPantev2012}. 
It is expected that the correspondence \eqref{Dolbeault-GLC} will have to be modified when one includes singular Hitchin fibers into consideration; 
exciting progress on this issue was made in \cite{ArinkinFedorov2016, Arinkin2012, Li2018} and in the recent works of Pădurariu-Toda \cite{PT25}.

We would like to suggest that the Lagrangian correspondence in \Cref{intro-main-result-lag-cor-Higgs} generically realizes the Dolbeault GLC.
One first step to establish a relation between our Lagrangian correspondence and the GLC is the following observation. Given the data of a Higgs bundle $(E, \phi) \in \mathcal{M}_H(n, \Lambda)$ 
and a line subbundle $L \hookrightarrow E$, twisting these data by an $n$th-root of $\mathcal{O}_C$ does not change the induced point $\widetilde{D}$ in  $\mathrm{Hilb}^d(T^\ast C)$. At least for $d$ sufficiently large, for $\widetilde{D} \in \mathrm{im}(\mathrm{pr}_1)$, $\mathrm{pr}_{2} \circ \mathrm{pr}_1^{-1}(\widetilde{D})$ consists of precisely $n^{2g}$ Higgs bundles that differ from each other by such a twist. 
This suggests that the passage from the moduli of Higgs bundles to the Hilbert schemes in our construction involves dualizing the gauge group  $\mathrm{SL}_n$ into its Langlands dual $\mathrm{PSL}_n$: choosing one among $n^{2g}$ such data corresponds to choosing a lift from a $\mathrm{PSL}_n$-object to an $\mathrm{SL}_n$-one
\footnote{From the character variety perspective, this corresponds to choosing a lift to $\mathrm{SL}_n$ of a monodromy representation $\pi_1 \rightarrow \mathrm{PSL}_n$.}.

In the rank-2 case, the fact that the Lagrangian correspondence $\mathbb{L}_H(d)$ generically realizes the Dolbeault GLC is very natural from the point of view of Drinfeld's construction of Hecke eigensheaves \cite{Drinfeld}.
To see this, let us summarize an adaptation 
\footnote{Drinfeld constructed Hecke eigensheaves on moduli stacks for $G = \mathrm{GL}_2(F)$ over the function field $F$ of a curved defined over finite fields.}
of Drinfeld's construction for $G = \mathrm{SL}_2$.
For simplicity, we will not distinguish between moduli stacks and moduli spaces in this discussion.
Starting with a $\mathrm{PSL}_2$-flat connection $\sigma$ on $C$, we want to construct a $\mathcal{D}$-module $\mathcal{F}_\sigma$ on the moduli space $\mathrm{Bun}_{\mathcal{O}_C} \equiv \mathrm{Bun}_{\mathrm{SL}_2}$ of $\mathrm{SL}_2$-bundles on $C$ which is a Hecke eigensheaf with eigenvalue $\sigma$.
The key ingredient in this construction is the auxiliary moduli space $\mathrm{Bun}_{\mathcal{O}_C, l}$ of pairs $(E, L)$ where $E$ is an $\mathrm{SL}_2$-bundle and $L$ a line subbundle of sufficiently negative  $\deg(L) = l$.
The forgetful map which remembers only the line subbundle equips $\mathrm{Bun}_{\mathcal{O}_C, l}$ with a projective bundle structure over $\mathrm{Pic}^l(C)$. 
Its dual projective bundle $\mathrm{Bun}^\vee_{\mathcal{O}_C, l}$ has fibers $\mathbb{P}H^0(C, KL^{-2})$ over $L \in \mathrm{Pic}^l(C)$ and is hence a $2^{2g}:1$ étale covering of $\mathrm{Hilb}^d(C)$, where $d = 2g - 2 - 2l$.
Following Drinfeld, we consider the induced flat bundle $\sigma^{[d]}$ on $\mathrm{Hilb}^d(C)$, its pull-back to $\mathrm{Bun}^\vee_{\mathcal{O}_C, l}$ and the corresponding Radon transform on $\mathrm{Bun}_{\mathcal{O}_C, l}$. 
The technical part of Drinfeld's construction was to show that the Radon transform descends to exactly the Hecke eigensheaf $\mathcal{F}_\sigma$ on $\mathrm{Bun}_{\mathcal{O}_C}$ 
(or, more precisely, as we are working with moduli spaces, the flat bundle induced by $\mathcal{F}_\sigma$ on the locus of \textit{very stable bundles} in $\mathrm{Bun}_{\mathcal{O}_C}$).

\begin{equation}\label{diagram-Drinfeld}
\begin{tikzcd}
&\mathbb{L}_H(d) \cong p^\ast T^\ast \mathrm{Bun}_{\mathcal{O}_C} \arrow{d} \arrow{r} 
&T^\ast \mathrm{Bun}_{\mathcal{O}_C, l} \arrow[dashed]{rrr}{\mathrm{SoV}} \arrow{d} 
& & &\mathrm{Hilb}^{d}(T^\ast C) \arrow{d} \\
&T^\ast \mathrm{Bun}_{\mathcal{O}_C} \arrow{d}
&\mathrm{Bun}_{\mathcal{O}_C, l} \arrow[dashed]{dl}{p} \arrow{dr}[swap]{q}
& &\mathrm{Bun}^\vee_{\mathcal{O}_C, l} \arrow{r} \arrow{dl}{q^\vee} &\mathrm{Hilb}^{d}(C)\\
& \mathrm{Bun}_{\mathcal{O}_C} & &\mathrm{Pic}^{l}(C)
\end{tikzcd}
\end{equation}

Now, in the range $2g-2 < d \leq 3g-3 = \frac{1}{2}\dim \mathcal{M}_H(2, \mathcal{O}_C)$, the second author with Teschner constructed \cite{DT23} a rational local symplectomorphism 
\begin{equation}\label{intro-SOV-map}
    \mathrm{SoV}: T^\ast \mathrm{Bun}_{\mathcal{O}_C, l} \dashrightarrow \mathrm{Hilb}^d(T^\ast C)
\end{equation}
which is generically finite of degree $2^{2g}:1$.
The map $\mathrm{SoV}$ is constructed using the same divisors $\widetilde{D}_i(\phi)$ (cf. \eqref{intro-Dtilde}) that are essential to the construction of the Lagrangian correspondence $\mathbb{L}_H(d)$ in \Cref{intro-main-result-lag-cor-Higgs}.
In fact, $\mathbb{L}_H(d)$ is canonically birational to the pull-back of $T^\ast \mathrm{Bun}_{\mathcal{O}_C}$ along the forgetful map 
$\mathrm{Bun}_{\mathcal{O}_C, l} \overset{p}{\dashrightarrow} \mathrm{Bun}_{\mathcal{O}_C}$, 
and the composition 
\[ \mathbb{L}_H(d) \overset{\simeq}{\dashrightarrow} p^\ast T^\ast \mathrm{Bun}_{\mathcal{O}_C} \dashrightarrow \mathrm{Hilb}^d(T^\ast C), \]
coincides with the restriction of the forgetful map $\mathrm{pr}_1$ in \Cref{intro-main-result-lag-cor-Higgs}.
Certain results in this direction were also obtained in an unpublished work of Enriquez-Roubtsov \cite{volodya-unpublished}.

One can regard the rational map $\mathrm{SoV}$ as the Dolbeault limit of the $\mathcal{D}$-module Radon transform between the dual fibrations $\mathrm{Bun}_{\mathcal{O}_C, l}$ and $\mathrm{Bun}^\vee_{\mathcal{O}_C, l}$ relative over $\mathrm{Pic}^l(C)$ (composed with the descent to $\mathrm{Hilb}^d(C)$).
In addition, it was suggested by Arinkin \cite{Arinkin2012} and proved in some cases \cite{AD} that the composition 
\( p^\ast T^\ast \mathrm{Bun}_{\mathcal{O}_C} \dashrightarrow \mathrm{Hilb}^d(T^\ast C) \)
is compatible with the Fourier-Mukai transform that generically realizes the Dolbeault GLC. 
Namely, given a generic adjoint Higgs bundle $(E, \phi) \in \mathcal{M}_H(\mathrm{PSL}_2)$, the pull-back of the induced Higgs sheaf $\mathcal{L}_{(E,\phi)}^{[d]}$ on $\mathrm{Hilb}^d(T^\ast C)$ along $p^\ast T^\ast \mathrm{Bun}_{\mathcal{O}_C} \dashrightarrow \mathrm{Hilb}^d(T^\ast C)$ shall coincide with the pull-back of the restriction to $T^\ast \mathrm{Bun}_{\mathcal{O}_C}$ of the Fourier-Mukai transform $\Phi_\Theta((E, \phi))$ along $p^\ast T^\ast \mathrm{Bun}_{\mathcal{O}_C} \dashrightarrow T^\ast \mathrm{Bun}_{\mathcal{O}_C} \subset \mathcal{M}_H(n, \Lambda)$.
The fact that this symmetry intertwines the actions of Hecke and Wilson operators, a necessary symmetry of the Dolbeault GLC \cite{Kapustin-Witten06, DonagiPantev2012, HauselICM}, shall follow by construction: the Lagrangians $\mathbb{L}_H(D)$ constructed in \Cref{intro-main-thm-1} are nothing but Hecke transforms of the Hitchin sections, and Hecke symmetries commute.

Note that Drinfeld's machinery has a natural Dolbeault version which, instead of a $\mathrm{PSL}_2$-flat bundle, takes as an input a Higgs bundle $(E, \phi) \in \mathcal{M}_H(\mathrm{PSL}_2)$ and outputs a singular Higgs sheaf on $\mathrm{Bun}_{\mathcal{O}_C, l}$ as a Radon transform.
One can expect that this singular Higgs sheaf is isomorphic to two other constructions: 
\begin{itemize}
    \item the direct image to $\mathrm{Bun}_{\mathcal{O}_C, l}$ of the Higgs sheaf $\mathrm{SoV}^\ast \left( \mathcal{L}_{(E,\phi)}^{[d]} \right)$ on $T^\ast\mathrm{Bun}_{\mathcal{O}_C, l}$;
    \item the pull-back along $\mathrm{Bun}_{\mathcal{O}_C, l} \dashrightarrow \mathrm{Bun}_{\mathcal{O}_C}$ of the singular Higgs bundle $(\mathcal{E}, \Phi)$ defined as the direct image 
to $\mathrm{Bun}_{\mathcal{O}_C}$ of the Fourier-Mukai transform $\Phi_\Theta((E, \phi))$. 
\end{itemize}
Conjecturally, $(\mathcal{E}, \Phi)$ develops singularities along the \textit{wobbly locus} which parametrizes stable bundles admitting nonzero nilpotent Higgs fields.
These objects feature in Donagi-Pantev's approach to the GLC which uses Simpson's and Mochizuki's versions of non-abelian Hodge correspondence to interpolate between Higgs bundles and flat bundles on $C$ and $\mathrm{Bun}_G$ respectively \cite{Simpson1991, Simpson1992, Simpson1997, Mochizuki2006, Mochizuki2009, DP09, DonagiPantev2012}. 
Namely, for $\sigma$ the non-abelian Hodge correspondence of $(E, \phi)$, Donagi-Pantev conjectured that, up to some birational transformations
\footnote{Mochizuki's non-abelian Hodge theorem applies to parabolic Higgs bundles $(\mathcal{E}, \Phi)$ with $c_1((\mathcal{E}, \Phi)) = 0 = c_2((\mathcal{E}, \Phi))$ and singular locus which has at worst codimension-$1$ normal crossing singularities. To apply these results, one might need to identify and resolve codimension-$1$ singularities in the wobbly locus that are more complicated than normal crossing.}, 
$(\mathcal{E}, \Phi)$ is the non-abelian Hodge correspondence of the flat bundle which is the restriction to the complement of the wobbly locus of the Hecke eigensheaf $\mathcal{F}_\sigma$. 
This realization of GLC for $G = \mathrm{SL}_2$ and its compatibility with Drinfeld's construction has been checked in the case $g = 2$ \cite{DPS24}, and provides a correspondence between the (classical) Dolbeault GLC and de Rham GLC.

These considerations suggests the following:
\vspace{-24pt}
\begin{quote}
\begin{conjecture}
    The Lagrangian correspondence $\mathbb{L}_H(d)$ in \Cref{intro-main-result-lag-cor-Higgs} generically realizes the Dolbeault GLC in the sense of \cite{Drinfeld, DP09, DonagiPantev2012, Arinkin2012}. 
\end{conjecture}
\end{quote}

We note that Laumon proposed a framework to generalize Drinfeld's construction from rank-2 to higher ranks, essentially by considering chains of auxiliary moduli spaces of bundles of rank $k$ with subbundles of rank $k-1$, thereby extending the diagram \eqref{diagram-Drinfeld} to the left until one reaches the moduli of rank-$n$ bundles \cite{Laumon95}.  It would be interesting to understand how our Lagrangian correspondence fits into Laumon's framework.

Let us remark that Theorems \ref{intro-main-result-lag-cor-Higgs} and \ref{intro-main-result-lag-cor-dR} 
extend and strengthen various Lagrangian correspondence results
which readily follow from \cite{iwasaki92, pinchbeck, DT23, D24-lambda} 
not only in allowing arbitrary ranks
but also in the following important aspect.
These previous works were done following Drinfeld's first construction of the GLC which uses the auxiliary moduli space $\mathrm{Bun}_{\mathcal{O}_C, l}$ with a caveat: 
the fibers of the forgetful map $\mathrm{Bun}_{\mathcal{O}_C, l} \dashrightarrow \mathrm{Bun}_{\mathcal{O}_C}$ have positive dimension in Drinfeld's work but are generically finite in the range $2g-2 < d \leq 3g-3$ that applies to these works.
On the other hand, Theorems \ref{intro-main-result-lag-cor-Higgs} and \ref{intro-main-result-lag-cor-dR} apply to any value of $d$, hence extending these results to Drinfeld's original setting where the fibers of the forgetful map have positive dimension.

\subsubsection{Quantization of Lagrangian correspondences}
Under suitable conditions, a Lagrangian correspondence, or more precisely its corresponding Fourier transform, can be canonically quantized \cite{Arinkin-Block-Pantev}. 
In the case of the Lagrangian correspondence $\mathbb{L}_H(d)$, the canonical quantization of the Fourier transform 
\begin{equation}\label{intro-Fourier-lag}
\mathbf{L_H(d)}: \mathrm{QCoh}\left( \mathrm{Hilb}^d(T^\ast C) \right)
\longrightarrow \mathrm{QCoh}\left(\mathcal{M}_H(n, \Lambda)\right)    
\end{equation}
given by $\mathrm{pr}_{2, \ast} \circ \mathrm{pr}_1^\ast (\mathcal{F})$
shall give rise to a functor  
\begin{equation}\label{intro-quantum-Fourier}
\mathbf{q-L_H(d)}: \mathrm{\mathcal{D}-mod}\left( \mathrm{Hilb}^d(C) \right)
\longrightarrow \mathrm{\mathcal{D}-mod}\left(\mathrm{Bun}_{n, \Lambda}\right),
\end{equation}
where $\mathrm{Bun}_{n, \Lambda}$ is the moduli space of rank-$n$ semistable bundles with fixed determinant $\Lambda$.
One might immediately recognize how such a functor could realize the de Rham GLC in the spirit of Drinfeld's construction. Namely, given a $\mathrm{PSL}_n$-flat bundle $\sigma$ on $C$, consider $\sigma^{[d]}$ and regard it as a $\mathcal{D}$-module on $\mathrm{Hilb}^d(C)$. 
We then expect that the $\mathcal{D}$-module $\mathbf{q-L_H(d)}(\sigma^{[d]})$ is a flat bundle on the complement of the wobbly locus in $\mathrm{Bun}_{n, \Lambda}$ and is isomorphic to the restriction of the Hecke eigensheaf $\mathcal{F}_\sigma$.

\vspace{-20pt}
\begin{quote}
\begin{conjecture} \label{intro-conj-quantized-Fourier}
The Fourier transform $\mathbf{L_H(d)}$ can be quantized to a functor $\mathbf{q-L_H(d)}$ as in \eqref{intro-quantum-Fourier}. 
For sufficiently large $d$, $\mathbf{q-L_H(d)}$ realizes the de Rham GLC in the sense of \cite{Drinfeld}.
\end{conjecture}
\end{quote}

Similarly, one can consider the Fourier transform 
\begin{equation} 
\mathbf{L_{dR}(d)}: \mathrm{QCoh}\left( \mathrm{Hilb}^d(\mathcal{S}) \right) \longrightarrow \mathrm{QCoh}\left( \mathrm{Hilb}^d(\mathcal{M}_{dR}(n, \mathcal{O}_C) \right) \end{equation}
induced by the Lagrangian correspondence $\mathbb{L}_{dR}(d)$ between the de Rham moduli space and Hilbert scheme of the total space $\mathcal{S}$ of the twisted cotangent bundle of $C$, and look to quantize this Fourier transform. 
The corresponding quantization shall give rise to a functor 
\begin{equation}
    \mathbf{q-L_{dR}(d)}: \mathrm{\mathcal{D}-mod}_{k'} \left( \mathrm{Hilb}^d(C) \right) \longrightarrow 
    \mathrm{\mathcal{D}-mod}_k\left(\mathrm{Bun}_{n, \Lambda}\right).
\end{equation}
between the categories of twisted $\mathcal{D}$-modules on $\mathrm{Hilb}^d(C)$ and twisted $\mathcal{D}$-modules on $\mathrm{Bun}_{n, \Lambda}$. As $\mathbf{L_{dR}(d)}$ can be considered as the deformation of $\mathbf{L_{H}(d)}$ by turning on the twistor parameter $\lambda \in \mathbb{C}$, which parametrizes the base of the Hodge moduli space of $\lambda$-connections on $C$, it is reasonable to expect that $\mathbf{q-L_{dR}(d)}$ realizes particular cases of the quantum GLC. 
It is also possible that, upon appropriate normalization, the procedures of deformation and quantization, which in the following diagram are respectively represented by dashed arrows and ``wiggling'' arrows, are commutative.

\vspace{-8pt}
\[
\begin{tikzcd}[row sep=2em, column sep=-0.9em]
    & \mathrm{QCoh}\left( \mathrm{Hilb}^d(\mathcal{S}) \right) \arrow[rr, "\mathbf{L_{dR}(d)}"] \arrow[dd, squiggly] & & 
    \mathrm{QCoh}\left( \mathrm{Hilb}^d(\mathcal{M}_{dR}(n, \mathcal{O}_C)) \right) \arrow[dd, squiggly] \\
    \mathrm{QCoh}\left(\mathrm{Hilb}^d(T^\ast C)\right) \arrow[ur, dashed] \arrow[rr, crossing over, "\mathbf{L_H(d)}" near end] \arrow[dd, squiggly] & & 
    \mathrm{QCoh}\left(\mathcal{M}_H(n, \Lambda)\right) \arrow[ur, dashed] \\
    & \mathrm{\mathcal{D}-mod}_{k'} \left( \mathrm{Hilb}^d(C) \right) \arrow[rr, "\mathbf{q-L_{dR}(d)}" near end] & & \mathrm{\mathcal{D}-mod}_k\left(\mathrm{Bun}_{n, \Lambda}\right) \\
    \mathrm{\mathcal{D}-mod}\left( \mathrm{Hilb}^d(C) \right) \arrow[ur, dashed] \arrow[rr, "\mathbf{q-L_H(d)}"] & & \mathrm{\mathcal{D}-mod}\left(\mathrm{Bun}_{n, \Lambda}\right) \arrow[ur, dashed] \arrow[from=uu, squiggly, crossing over]
\end{tikzcd}
\]

\subsubsection{Relation to branes and Whittaker sheaves}
The seminal work of Kapustin-Witten formulates mirror symmetry and the GLC as a reduction of electro-magnetic duality in four dimensional supersymmetric gauge theory \cite{Kapustin-Witten06}.
Of particular importance in their approach are branes in the hyperk\"{a}hler Hitchin systems: with respect to each of the complex structures $(I, J, K)$, a subvariety is called of type $A$ or $B$ if it is a Lagrangian or a holomorphic subvariety respectively 
\footnote{The $A$-branes and $B$-branes in addition are equipped with flat bundles and holomorphic sheaves respectively}. 
In this language, the Lagrangians $\mathbb{L}_H(D)$ and $\mathbb{L}_{dR}(D)$ are branes of type $(BAA)$ and $(ABA)$ respectively.

The mirrors of $(BAA)$ and $(ABA)$ branes, according to Kapustin-Witten, should be respectively $(BBB)$ and $(ABA)$ branes in the moduli spaces associated to the dual group $\check{G}$. Conjecturally, these mirrors are generically given by the Fourier-Mukai transforms \eqref{Dolbeault-GLC}. For example, consider the Lagrangians $\mathbb{L}_H(L, D) \subset \mathcal{M}_H(n, \ell)$, i.e. the self-dual case $G = \mathrm{GL}_n$.
As $\mathbb{L}_H(L, D)$ intersects a generic Hitchin fiber at $d^n$ distinct points, it follows from the self-duality of Jacobians of spectral curves that, over an open dense subset of $\mathcal{M}_H(n, \ell)$, the mirror $\Phi_\Theta(\mathbb{L}_H(L, D))$ given by Fourier-Mukai transform is a vector bundle of rank $d^n$. 
When $d$ is sufficiently small and $D$ is reduced such that $\mathbb{L}_H(L, D)$ coincides with the upward flow $W^0_{(E_0, \phi_0)}$ associated to a \textit{very stable Higgs bundle} $(E_0, \phi_0) \in \mathcal{M}_H(n, \ell)^{\mathbb{C}^\ast}$, 
it can be shown that the mirror extends to a vector bundle over $\mathcal{M}_H(n, \ell)$ \cite{HH22}. 
In the case where $(E_0, \phi_0)$ is a \textit{wobbly Higgs bundle} or where $d$ is sufficiently large, the structure of the mirror is less obvious: one obstruction to such an investigation is the fact that $\mathbb{L}_H(L, D)$ is not closed and the intersection of $\overline{\mathbb{L}_H(L, D)}$ with the nilpotent cone in this case is of positive dimension.

Finally, we note the relationship between our constructions and the Whittaker sheaf $\mathrm{Whit}_{G} \in \mathcal{D}-\mathrm{mod}(\mathrm{Bun}_{G})$. In the de Rham GLC, the Whittaker sheaf is dual to the structure sheaf $\mathcal{O}_{\mathcal{M}_{dR}(\check{G})} \in \mathrm{QCoh}(\mathcal{M}_{dR}(\check{G}))$ and admits the Hitchin section $\mathbb{L}(\varnothing) \subset \mathcal{M}_H(G)$ as its microlocal support \cite{BenZvi26}.

Our constructions of $\mathbb{L}_H(D)$, which are obtained via Hecke transforms of the Hitchin section $\mathbb{L}_H(\varnothing)$, can be regarded as the Dolbeault counterparts to applying Hecke functors to $\mathrm{Whit}_{G}$. 
We expect that the Lagrangians $\mathbb{L}_H(D)$ are the microlocal support of appropriate restriction of the images of $\mathrm{Whit}_{G}$ along Hecke functors.
In particular, the fact that the forgetful map $\mathbb{L}_H(d) \rightarrow \mathcal{M}_H(G)$ is dominant for $d \geq \frac{1}{2} \dim \mathcal{M}_H(G)$ can be regarded as the geometric counterpart of the categorical statement that $\mathrm{Whit}_{G}$, together with its Hecke transforms, generate the category $\mathcal{D}-\mathrm{mod}(\mathrm{Bun}_G)$.

The relation between the Whittaker model and Drinfeld's construction of the GLC was discussed by Frenkel in \cite{Frenkel1995}. 
Assuming that \Cref{intro-conj-quantized-Fourier} is true, i.e. for generic $\sigma \in \mathcal{M}_{dR}(\check{G})$, the quantized Fourier transform $\mathbf{q-L_H(d)}$ sends $\sigma^{[d]} \mapsto \mathcal{F}_\sigma$, we observe that the construction goes through using different, distinguished bases on the two sides of the GLC:
\begin{itemize}
    \item On the spectral side: skyscraper sheaves $\mathcal{O}_\sigma \in \mathrm{QCoh}(\mathcal{M}_{dR}(\check{G}))$ -- the Langlands dual of these skyscraper sheaves are Hecke eigensheaves $\mathcal{F}_\sigma$ whose singular supports are in the nilpotent cone and microlocal supports are generically line bundles on Hitchin fibers; 
    \item On the automorphic side: the Whittaker sheaf $\mathrm{Whit}_G$, which is dual to $\mathcal{O}_{\mathcal{M}_{dR}(\check{G})}$ and has the Hitchin section $\mathbb{L}_H(\varnothing)$ as microlocal support, and its Hecke transforms, which ought to have $\mathbb{L}_H(D)$ as microlocal supports upon appropriate restriction.
\end{itemize}
Such an intertwining of bases is the signature that makes  \textit{Separation of Variables} techniques powerful in the theories of systems of PDE in general and of integrable systems in particular.
Our Lagrangian correspondences in Theorems \ref{intro-main-result-lag-cor-Higgs} and \ref{intro-main-result-lag-cor-dR}, which we conjecture to realize the (classical) Dolbeault GLC, appear to be the \textit{classical changes of variables} compatible with the spectral decomposition $\mathcal{F}_\sigma \mapsto \sigma^{[d]} \mapsto \sigma$ in the de Rham GLC.

\subsubsection{Relation to conformal field theories}
In this subsection, we discuss the relations between our work and conformal field theories on Riemann surfaces $\Sigma$ with marked points.

We first note that, following Beilinson-Drinfeld's seminal work \cite{BD92}, Hecke eigensheaves $\mathcal{F}_\sigma$ with eigenvalues $\sigma$ being $\check{G}$-opers (\textit{without} apparent singularities) can be constructed by taking the critical limit of conformal blocks of Wess-Zumino-Witten conformal field theories with Kac-Moody $\hat{\mathfrak{g}}$ symmetries. A concrete example of this limit is the case when $\Sigma$ is the Riemann sphere with marked points. In this case, the quantum Hitchin Hamiltonian eigenvalue equations that define $\mathcal{F}_\sigma$ can be recovered as the critical limit of the Knizhnik-Zamolodchikov (KZ) equations \cite{ReshetikhinVarchenko95, Teschner2011}. 
As discussed in \Cref{sect-intro-lag-property}, in our setting, the space of $\mathrm{SL}_n$-opers is precisely the Lagrangian $\mathbb{L}_{dR}(\varnothing) \subset \mathcal{M}_{dR}(n, \mathcal{O}_C)$, while the Lagrangians $\mathbb{L}_{dR}(D)$ in \Cref{intro-main-thm-1} can be considered as ``excitations'' of the space of opers at effective divisors $D$ on $C$.
It would be a natural and interesting investigation to look for an analogue of Beilinson-Drinfeld's conformal field theories machinery that construct Hecke eigensheaves $\mathcal{F}_\sigma$ with $\sigma \in \mathbb{L}(D)$. 
Ideas for this direction have been discussed in \cite{Frenkel2007, Ar16, GLC-proof}: one might account for these ``excitations'' by inserting certain finite dimensional $\mathfrak{g}$-representations at the divisors $D$ of apparent singularities.  

The second point we would like to emphasize here is that opers with apparent singularities naturally feature at the classical limits of certain Belavin-Polyakov-Zamolodchikov (BPZ) equations with degenerate fields.
This relation was discussed in \cite{Reshetikhin92, Teschner2011, Teschner17, Litvinov2014} in the case where $\Sigma$ is the Riemann sphere $\{z \in \mathbb{C} \}\cup \{\infty\}$ with marked points defining a reduced divisor $\mathbf{z} = z_1 + \dots + z_N$.
We now summarize this relation following Teschner (Section 3 of \cite{Teschner17}).

Let us consider highest weight representations $V_\alpha$ of the Virasoro algebra with central charge $c$, with the conformal weights $\Delta_\alpha$ of the highest weight vectors and the central charge $c$ often parametrized as
\[ \Delta_\alpha = \alpha(Q - \alpha), \qquad \qquad c = 1 + 6Q^2, \qquad \qquad \text{where } Q = b + b^{-1}.  \]
We assign representations $V_{\alpha_r + 1/2b}$ to the marked points $z_1, \dots, z_d$, $V_\alpha$ to marked points $z_{d+1}, \dots, z_N$, and degenerate representations $V_{-1/2b}$ to points in the reduced divisor $D = p_1 + \dots + p_d$. We also assign the degenerate representation $V_{-b/2}$ to $z = y$, regarding such field as a ``free variable''. 
The corresponding chiral partition function $Z(D, \mathbf{z}; y)$ then satisfies BPZ equations of the form 
\begin{subequations}\label{intro-BPZ-Liouville}
\begin{align}\label{intro-BPZ-Liouville-1}
\left[ \frac{1}{b^2} \partial_y^2 + T(y) \right] Z(D, \mathbf{z}; y) = 0,& \\
T(y) \coloneqq \sum_{r=1}^n \left( \frac{\Delta_r}{(y - z_r)^2} + \frac{1}{y - z_r} \partial_{z_r} \right) - \sum_{k=1}^d \left( \frac{3b^{-2} + 2}{4(y - p_k)^2} - \frac{1}{y - p_k} \partial_{p_k} \right)& \nonumber
\end{align}
and
\begin{align}\label{intro-BPZ-Liouville-2}
\left( b^2 \partial_{p_k}^2 + \check{\mathsf{T}}_k(p_k) - \frac{3b^2 + 2}{4(p_k - y)^2} + \frac{1}{p_k - y} \partial_y \right) Z(D, \mathbf{z}; y) = 0,& \\
\check{\mathsf{T}}_k(p_k) \coloneqq \lim_{y \to p_k} \left( \mathsf{T}(y) + \frac{3b^{-2} + 2}{4(y - p_k)^2} - \frac{1}{y - p_k} \partial_{p_k} \right)&, \nonumber
\end{align}
\end{subequations}
together with BPZ equations for the primary fields inserted at the marked points $z_1, \dots, z_N$.

One then can recover differential equations that define opers with apparent singularities at $p_1, \dots, p_d$ by carefully taking the $b \rightarrow 0$ limit of \eqref{intro-BPZ-Liouville}.
Indeed, observe that at the limit $b \rightarrow 0$ with $\delta_r = b^2 \Delta_r$ kept fixed, the ansatz 
\footnote{To be precise, in using this ansatz we have chosen to break the $\mathfrak{sl}_2$ symmetry on the sphere; namely, we have fixed $z_{N-2} = 1$, $z_{N-1} = 0$ and $z_N = \infty$.}
\[ Z(D, \mathbf{z}; y) =  e^{-\frac{1}{b^2}S(D, \mathbf{z})} \chi(y \mid D, \mathbf{z})  \left( 1 + \mathrm{O}(b^2) \right) \]
solves \eqref{intro-BPZ-Liouville-1} provided that $\chi$ solves the oper equation
\[ [\partial_y^2 + 
t(y)] \chi(y) = 0, \qquad t(y) = \sum_{r=1}^N \left( \frac{\delta_r}{(y-z_r)^2} - \frac{H_r}{y-z_r}\right) - 
\sum_{k=1}^d \left( \frac{3}{4(y-p_k)^2} - \frac{\nu_k}{y - p_k}\right), \]
with residue parameters $\nu_k = - \partial_{p_k} S(D, \mathbf{z})$ and $H_r = \partial_{z_r}S(D, \mathbf{z})$. In particular, it follows from equations \eqref{intro-BPZ-Liouville-2} that the residue parameters $\nu_k$ defined as such satisfy the constraints that ensure trivial $\mathrm{PSL}_2$-monodromy 
\[ \nu_k^2 + \check{t}_k(p_k) = 0, \qquad \check{t}_k(p_k) := \lim_{y \to p_k} \left( t(y) + \frac{3}{4(y - p_k)^2} - \frac{\nu_k}{y - p_k} \right), \] 
hence defining a reduced divisor $D = p_1 + \dots p_d$ of apparent singularities. The residues $H_r$ correspond to the Hamiltonians generating isomonodromic flows as one varies the positions of the marked points $z_1, \dots, z_N$.

We therefore have recovered $\mathrm{PSL}_2$-opers with apparent singularities at the classical limit of Virasoro conformal blocks and corresponding chiral partition functions. These objects have a natural place in Liouville conformal field theories, which are special cases of $\mathfrak{sl}_n$ Toda conformal field theories with $n = 2$. For general $n \geq 3$, one expects an analogous relation to higher order differential BPZ equations induced by the enlarged \(W_n\)-symmetry. For example, in the case of \(\mathfrak{sl}_3\) Toda field theories, the relevant 
\footnote{Note that the parameter $b$ in Toda theory differs from the one in Liouville theory because of the convention that simple roots are not orthonormal but rather satisfy $\langle e_i,e_i\rangle = 2$.}
degenerate field is \( V_{-2b^{-1}\omega_1}(z)\),
where $\omega_1$ is the positive simple root for $\mathfrak{sl}_3$.
In this case, the corresponding null-vector equations 
induce the third-order differential equation analogue of \eqref{intro-BPZ-Liouville-2} 
(cf. equations (5.13) -- (5.15) in \cite{CercleHuang2021}).
A similar formal semiclassical analysis suggests that at the limit $b \rightarrow 0$, these BPZ equations degenerate to the constraints on residue parameters which ensure trivial $\mathrm{PSL}_3$-monodromy around $p_1, \dots, p_d$, and hence defining these points as apparent singularities of a $\mathrm{PSL}_3$-oper differential equation that appears at the limit. 

We expect similar relations in the case of Riemann surfaces $\Sigma$ with marked points of higher genera. Namely, we expect to recover differential equations defining $\mathrm{PSL}_n$ opers with apparent singularities on $\Sigma$ and, in particular, the constraints on residue parameters of apparent singularities (cf. Section 4.3) at the $b \rightarrow 0$ limit of BPZ equations for appropriate degenerate fields in $\mathfrak{sl}_n$ Toda theories. From this perspective, degenerate fields in conformal field theories provide a quantum origin for opers with apparent singularities. 

Finally, we note that one can define integral transforms that intertwine correlation functions of Liouville theories and $H_3^+$ WZW theories. Such integral transforms can be rigorously defined in the probabilistic approach to CFT \cite{GKR25}. In the context of the analytic Langlands correspondence (ALC) \cite{Teschner2018, EFK2021, EFK2022, EFK2023, AT25}, at the critical limits, these correlation functions degenerate to solutions of differential equations defining \textit{real $\mathrm{PSL}_2$ opers} (without apparent singularities) and sections of the density bundle $|K^{1/2}_{\mathrm{Bun}_{\mathrm{SL}_2}}|$ that solve both the holomorphic and anti-holomorphic quantum Hitchin Hamiltonian eigenvalue equations.
For this reason, Gaiotto-Teschner \cite{GT24} 
argued that such integral transforms should be viewed as a realization of the \textit{quantum analytic Langlands correspondence}. 
Variants of such integral transforms at the critical level limits naturally feature in the GLC \cite{Frenkel1995} and ALC; 
in particular, this technique can be used to bypass certain analytic difficulties in the ALC regarding square-integrability near the wobbly locus \cite{AT25, AADT}. In turn, the Lagrangian correspondence functor \eqref{intro-Fourier-lag} and the local symplectomorphism \eqref{intro-SOV-map} that realize the Dolbeault GLC appear at the classical limits of these critical level integral transforms. 

The emerging picture is that these integral transforms and their limits are variants of the \textit{classical and quantum Separation of Variables} technique applied to the integrable and hyperk\"{a}hler Hitchin systems, and they naturally feature in different levels of quantisation, deformation and imposing quantisation conditions on the GLC.

\paragraph{Acknowledgements} The authors are grateful to Tony Pantev for many valuable discussions and suggestions that help shape and improve this paper. We also thank Dima Arinkin, David Ben-zvi, Indranil Biswas, Avik Chakravarty, Daebeom Choi, Ron Donagi, Volodya Roubtsov, J\"org Teschner and Richard Wentworth for valuable discussions and correspondences. 
Đ.Q.D is supported by the German Research Foundation (DFG) via the Walter Benjamin fellowship. 
S.X is partially supported by the BSF/NSF grant DMS-2200914.

\section{Preliminaries}

\subsection{Moduli of Higgs bundles}
From this point forward we fix $C$ to be a smooth projective curve over $\mathbb{C}$ of genus $g\ge 2$. For $G=\mathrm{GL}_n$, let $\mathcal H^{ss}(n,\ell)$ be the space of semistable Higgs bundles of rank $n$ and degree $\ell$ on $C$, parametrizing pairs $(E,\phi)$ where $E$ is a holomorphic vector bundle of rank $n$ and degree $d$ on $C$ whose holomorphic structure we denote by $\bar\pa_E$, and $\phi\in H^0\!\left(C,\operatorname{End}(E)\otimes K_C\right)$ is a Higgs field satisfying $\bar\pa_E \phi = 0$. The gauge group $\mathrm{GL}(E)$ acts on such pairs
\begin{equation*}
g\cdot (\bar\pa_E,\phi) := (g^{-1}\circ \bar\pa_E\circ g, g^{-1}\circ \phi\circ g). 
\end{equation*}
We define the moduli space of Higgs bundles as $\mathcal{M}_H(n,\ell):=\mathcal H^{ss}(n,\ell)/\mathrm{GL}(E)$. Each point of $\mathcal{M}_H(n,\ell)$ corresponds to the unique polystable representative in its $S$--equivalence class, and $\mathcal{M}_H(n,\ell)$ is a quasi-projective variety of complex dimension $\dim_{\mathbb{C}}\mathcal{M}_H(n,\ell)=2n^2(g-1)+2$. The structure of the $\mathrm{SL}_n$ moduli space is analogous, the only difference being that we impose the conditions $\det(E) = \mathcal O$ and $\Tr(\phi) = 0$.

\paragraph{Holomorphic symplectic form.}

The stable locus of $\mathcal{M}_H(n,\ell)$ can be equipped with a canonical holomorphic symplectic form $\Omega_H$. A tangent vector at $(E,\phi)\in \mathcal H^{ss}(n,\ell)$ is given by a pair
\[
(\alpha,\beta)\in \Omega^{0,1}\!\left(C,\operatorname{End}(E)\right)\oplus \Omega^{1,0}\!\left(C,\operatorname{End}(E)\right)\]
satisfying the linearized Higgs field equation 
\[
\bar{\partial}_E\beta+[\alpha,\phi]=0.
\]
For two such tangent vectors $(\alpha_1,\beta_1)$ and $(\alpha_2,\beta_2)$,
\[
\Omega_H\bigl((\alpha_1,\beta_1),(\alpha_2,\beta_2)\bigr)
:=\int_C \operatorname{tr}\!\left(\alpha_1\wedge \beta_2-\alpha_2\wedge \beta_1\right)
\]
defines a closed, nondegenerate $(2,0)$--form on $\mathcal H^{ss}(n,d)$ which descends to a symplectic form on the stable locus.

\paragraph{Hitchin fibration.}

For $G=\mathrm{GL}_n$, the \emph{Hitchin fibration} is the map
\begin{equation}\label{e: hitchin fibration}
h:\mathcal{M}_H(n,\ell)\longrightarrow B:=\bigoplus_{i=1}^n H^0\!\left(C,K_C^i\right),
\end{equation}
sending $(E,\phi)$ to the coefficients of the characteristic polynomial of $\phi$; we call $B$ the \emph{Hitchin base}. Let $\pi:\operatorname{Tot}(K_C)\to C$ be the projection and let $\lambda\in H^0(\operatorname{Tot}(K_C),\pi^\ast K_C)$ be the tautological section. Writing $h(E,\phi)=b=(b_1,\dots,b_n)$, the characteristic polynomial defines a section
\[
\det\!\left(\lambda\cdot \mathrm{Id}-\pi^\ast\phi\right)
=
\lambda^n+\pi^\ast b_1\,\lambda^{n-1}+\cdots+\pi^\ast b_n
\in
H^0\!\left(\operatorname{Tot}(K_C),\pi^\ast K_C^{\otimes n}\right),
\]
whose vanishing locus is the \emph{spectral curve} $C_b\subset \operatorname{Tot}(K_C)$. The map $h$ endows $\mathcal{M}_H(n,\ell)$ with the structure of a completely integrable system. For a generic $b\in B$, $C_b$ is smooth and the fiber $h^{-1}(b)$ is isomorphic to $\operatorname{Pic}^{d}(C_b)$ with $d=\ell+n(n-1)(g-1)$. Let $B_{\mathrm{reg}}\subset B$ be the locus of $b$ for which $C_b$ is smooth, and set $\mathcal{M}_{H,\mathrm{reg}}(n,\ell):=h^{-1}(B_{\mathrm{reg}})$.

\paragraph{Spectral correspondence.}
Over $B_{\mathrm{reg}}$, a point of $\mathcal{M}_{H,\mathrm{reg}}(n,\ell)$ corresponds to spectral data $(\widetilde{C},\mathcal{L})$, where $\widetilde{C}=C_b$ is a smooth spectral curve with covering map $\pi:\widetilde{C}\to C$, and $\mathcal{L}\in \operatorname{Pic}^{d}(\widetilde{C})$ is a line bundle such that
\[
(E,\phi)\ \cong\ \bigl(\pi_\ast \mathcal{L},\ \pi_\ast(\lambda:\mathcal{L}\to \mathcal{L}\otimes \pi^\ast K_C)\bigr).
\]
Conversely, given $(E,\phi)$ with smooth spectral curve $\widetilde{C}$, the spectral line bundle $\mathcal{L}=\mathcal{L}_{(E,\phi)}$ can be recovered as
\[
\mathcal{L}_{(E,\phi)}
\coloneqq
\operatorname{coker}\!\left(
\pi^\ast(\phi)-\lambda:
\pi^\ast\!\left(E\otimes K_C^{-1}\right)\longrightarrow \pi^\ast E
\right),
\]
where $\lambda$ acts by multiplication on $\pi^\ast E$ along $\widetilde{C}$.

\subsection{Moduli of holomorphic connections}
\label{s:moduli-hol-conn}

For $G=\mathrm{GL}_n$, we consider holomorphic connections $(E,\nabla)$ on $C$, where $E$ is a holomorphic vector bundle of rank $n$ and
\(\nabla:E\to E\otimes K_C\)
is a holomorphic connection, i.e.\ a $\mathbb{C}$--linear map satisfying $\nabla(fs)=f\nabla(s)+(\partial f)\otimes s$. A theorem of Atiyah--Weil asserts that $E$ admits a holomorphic connection if and only if every indecomposable summand of $E$ has degree $0$. This in particular forces $\deg(E) = 0$. 
Let $\mathcal H_{dR}^{ss}(n)$ denote the space of semistable pairs $(E,\nabla)$ which further satisfy the compatibility condition $[\bar\pa_E,\nabla]=0$. Again, the gauge group acts on this space and we define the de Rham moduli space $\mathcal{M}_{dR}(n):= \mathcal H_{dR}^{ss}(n)/\mathrm{GL}(E)$. It is a normal quasi-projective variety of complex dimension $\dim_{\mathbb{C}}\mathcal{M}_{dR}(n)=2n^2(g-1)+2$. Equivalently, it may be viewed as the moduli space of semisimple flat $\mathrm{GL}_n$--connections on $C$ by sending $(\bar\pa_E,\nabla) \to \bar\pa_E+\nabla$.

\paragraph{Irreducible locus and tangent space.}
We write $\mathcal{M}_{dR}^{\mathrm{irr}}(n)\subset \mathcal{M}_{dR}(n)$ for the open locus of \emph{irreducible} connections, i.e.\ those connections for which $E$ has no proper nonzero $\nabla$--invariant holomorphic subbundles (equivalently, the monodromy representation of $\pi_1(C)$ is irreducible). On $\mathcal{M}_{dR}^{\mathrm{irr}}(n)$ the moduli space is smooth, and the tangent space at $(E,\nabla)$ is canonically identified with
\[
T_{(E,\nabla)}\mathcal{M}_{dR}^{\mathrm{irr}}(n)\ \cong\
\mathbb{H}^1\!\left(C,\ \mathcal{E}nd(E)\xrightarrow{\ \operatorname{ad}(\nabla)\ }\mathcal{E}nd(E)\otimes K_C\right),
\]
where $\operatorname{ad}(\nabla)(u)=[\nabla,u]$.

\paragraph{Atiyah--Bott holomorphic symplectic form.}
On $\mathcal{M}_{dR}^{\mathrm{irr}}(n)$ there is a canonical holomorphic symplectic form $\Omega_{\mathrm{AB}}$. It is induced by the trace pairing $\operatorname{tr}:\mathcal{E}nd(E)\otimes\mathcal{E}nd(E)\to\mathcal{O}_C$ via cup product and Serre duality, giving a bilinear pairing
\[
\Omega_{\mathrm{AB}}:\ 
\mathbb{H}^1(\mathcal{C}^\bullet_{(E,\nabla)})\otimes \mathbb{H}^1(\mathcal{C}^\bullet_{(E,\nabla)})
\longrightarrow
\mathbb{H}^2(C,K_C)\cong \mathbb{C},
\]
where $\mathcal{C}^\bullet_{(E,\nabla)}$ denotes the above two-term deformation complex. Equivalently, given a semisimple complex flat connection $\nabla=d+A$ on a fixed smooth bundle and two tangent vectors $\delta_1 A$ and $\delta_2 A$, we define the symplectic form
\[
\Omega_{\mathrm{AB}}(\delta_1 A,\delta_2 A)=\int_C \operatorname{Tr}\!\left(\delta_1 A\wedge \delta_2 A\right).
\]
This form descends to a non-degenerate symplectic form on $\mathcal H_{dR}^{s}(n)/\mathrm{GL}(E)$, namely on the stable locus. 

\medskip
The $\mathrm{SL}_n$ case is treated similarly, after imposing $\det(E)\cong \mathcal{O}_C$ and $\Tr(\nabla)=0$ (equivalently, passing to $\mathrm{SL}_n$--connections with fixed determinant).

\subsection{The Hodge moduli space}

We now introduce the Hodge moduli space, which interpolates between the moduli space of Higgs bundles and the de Rham moduli space.

For the group \(\mathrm{GL}_n\), the Hodge moduli space \(\mathcal{M}_{\mathrm{Hod}}(\mathrm{GL}_n)\) parametrizes triples
\[
(\lambda,E,\nabla_\lambda), \qquad \lambda\in \mathbb C,
\]
up to gauge equivalence, where \(E\) is a rank-\(n\) holomorphic vector bundle on \(C\), and \(\nabla_\lambda\) is a \(\lambda\)-connection on \(E\), namely a \(\mathbb C\)-linear map
\[
\nabla_\lambda:E\to E\otimes K_C, \text{ satisfying } \nabla_\lambda(fs)=\lambda\,\partial f\otimes s+f\,\nabla_\lambda(s)
\]
satisfying the compatibility condition $[\bar\pa_E,\nabla_{\lambda}] = 0$. 
When \(\lambda=0\), the operator \(\nabla_\lambda\) is simply a Higgs field, while when \(\lambda=1\), it is an ordinary holomorphic connection. More generally, for every \(\lambda\neq 0\), rescaling by \(\lambda^{-1}\) identifies the fiber over \(\lambda\) with the de Rham moduli space. In this way, \(\mathcal{M}_{\mathrm{Hod}}(n)\) fits into a natural family
\(
\mathcal{M}_{\mathrm{Hod}}\longrightarrow \mathbb{C}, 
\)
whose fiber over \(0\) is \(\mathcal{M}_H\), and whose fibers over \(\lambda\neq 0\) are naturally identified with \(\mathcal{M}_{dR}\).

For the group \(\mathrm{SL}_n\), one defines \(\mathcal{M}_{\mathrm{Hod}}(\mathrm{SL}_n)\) similarly, by imposing the natural extra conditions.

There is a natural \(\mathbb C^\ast\)-action on \(\mathcal{M}_{\mathrm{Hod}}\), given by scaling the parameter and the \(\lambda\)-connection:
\[
t\cdot [(\lambda,E,\nabla_\lambda)]
=
[(t\lambda,E,t\nabla_\lambda)],
\qquad t\in \mathbb C^\ast.
\]
On the central fiber \(\lambda=0\), this reduces to the usual scaling action on the Higgs fields.

By the nonabelian Hodge correspondence, the smooth locus of the Higgs/de Rham moduli space carries a hyperk\"ahler structure. The corresponding twistor family is obtained by gluing the Hodge moduli spaces along the nonzero fibers, and the parameter \(\lambda\) may be viewed as the coordinate along the twistor line. Thus \(\mathcal{M}_{\mathrm{Hod}}\) should be regarded as the holomorphic family underlying the twistor space of the nonabelian Hodge moduli space.

\subsection{Hilbert schemes of points on a surface}
\label{s:hilb-surface}

Let $S$ be a quasi-projective complex surface. For an integer $d\ge 1$, the Hilbert scheme of $d$ points on $S$ is the moduli space
\(
\operatorname{Hilb}^d(S)
\)
parametrizing $0$--dimensional subschemes $Z\subset S$ of length $d$, equivalently ideal sheaves $I_Z\subset \mathcal{O}_S$ with
\(\dim_{\mathbb{C}} H^0(S,\mathcal{O}_S/I_Z)=d\).
If $S$ is smooth, then $\operatorname{Hilb}^d(S)$ is smooth of complex dimension $2d$. 

\paragraph{Hilbert--Chow morphism.}
There is a natural morphism (the Hilbert--Chow map)
\[
\rho:\operatorname{Hilb}^d(S)\longrightarrow \operatorname{Sym}^d(S),
\]
sending a subscheme $Z$ to the associated effective $0$--cycle $\sum m_i p_i$. It is an isomorphism over the open locus of reduced subschemes $\operatorname{Sym}^d(S)\setminus \Delta$ (i.e.\ $d$ distinct points) and provides a resolution of singularities of the symmetric product.

\paragraph{Holomorphic symplectic structure.}
Assume that $S$ carries a holomorphic symplectic form, i.e.\ a non-degenerate holomorphic $2$--form
\(\omega_S\in H^0(S,\Omega_S^2)\).
Then $\operatorname{Hilb}^d(S)$ admits a canonical holomorphic symplectic form $\omega_S^{[d]}$, uniquely characterized by the property that on the open locus of reduced subschemes
\[
\operatorname{Hilb}^d(S)^{\mathrm{red}}\cong (S^d\setminus \Delta)/\mathfrak{S}_d,
\]
the form $\omega_S^{[d]}$ agrees with the $\mathfrak{S}_d$--invariant $2$--form induced from $\omega_S$ on each factor of $S^d$. In particular, $\omega_S^{[d]}$ is closed and nondegenerate, hence holomorphic symplectic.

\paragraph{$S=T^\ast C$ and twisted cotangent bundles.}
In this paper the symplectic surfaces we deal with are either the cotangent bundle of the curve $C$ or a twisted version of it. 

\smallskip
The cotangent bundle $T^\ast C$ carries its tautological $1$--form $\theta$ and the canonical holomorphic symplectic form
\[
\omega_{T^\ast C}= \mathrm{d}\theta \in H^0(T^\ast C,\Omega_{T^\ast C}^2),
\]
so $\operatorname{Hilb}^d(T^\ast C)$ inherits a holomorphic symplectic form $\omega_{T^\ast C}^{[d]}$.

\smallskip
More generally, let $\mathcal{S}\to C$ be a \emph{twisted cotangent bundle}. By this we mean a holomorphic torsor under $T^\ast C$, namely an affine bundle locally isomorphic to $T^\ast C$, with transition functions given by translations by holomorphic $1$--forms. Since translations preserve $\omega_{T^\ast C}$, the local identifications glue to a global holomorphic symplectic form
\[
\omega_{\mathcal{S}}\in H^0(\mathcal{S},\Omega_{\mathcal{S}}^2),
\]
and hence $\operatorname{Hilb}^d(\mathcal{S})$ inherits a holomorphic symplectic form $\omega_{\mathcal{S}}^{[d]}$.

\subsection{Opers}
Opers and their generalizations will be the main objects of study throughout section \ref{LDR}. In this subsection, we recall the classical notion of $\mathrm{SL}_n$ (or $\mathrm{PSL}_n$) opers, following \cite{We15, BD05}. The analogous results apply for $\mathrm{GL}_n$ opers, opers over a formal disc, or a punctured disc. 

\begin{definition}
An $\mathrm{SL}_n$-oper on a smooth curve $C$ is a triple $(E,\nabla,\{E_i\})$, where $E$ is a rank-$n$ holomorphic vector bundle with $\det E \cong \mathcal O_C$, $\nabla$ is a holomorphic connection inducing the trivial connection on $\det E$, and
\[
0=E_0\subset E_1\subset \cdots \subset E_n=E
\]
is a filtration by holomorphic subbundles such that
\[
\nabla(E_i)\subset E_{i+1}\otimes K_C,
\qquad
\nabla: E_i/E_{i-1}\xrightarrow{\simeq}(E_{i+1}/E_i)\otimes K_C,
\quad 1\le i\le n-1.
\]
These conditions are referred to as the \emph{oper conditions}.

A $\mathrm{PSL}_n$-oper is the projectivization of an $\mathrm{SL}_n$-oper, which is denoted by $\mathbb P(E,\nabla,\{E_i\})$. 
\end{definition}

\begin{remark}
The oper conditions admit a concrete local form. 
Choose a coordinate $z$ and a local frame adapted to the filtration.
Then the isomorphism conditions on the successive quotients allow one to normalize the connection matrix into the shape
\[
\nabla=\partial_z+
\begin{pmatrix}
* & * & \cdots & * \\
1 & * & \cdots & * \\
0 & 1 & \ddots & \vdots \\
\vdots & \ddots & \ddots & * \\
0 & \cdots & 0 & 1 & *
\end{pmatrix}.
\]
Equivalently, if $B\subset \mathrm{GL}_n$ is the Borel subgroup of upper triangular matrices with Lie algebra $\mathfrak b$, then after gauge transformation by its unipotent radical $N$ one may write
\begin{equation}\label{eq:oper-local-DS}
\nabla=\partial_z+p_{-1}+v(z),
\qquad
p_{-1}:=\sum_{i=1}^{n-1}E_{i+1,i},
\qquad
v(z)\in\mathfrak b.
\end{equation}
The fixed subdiagonal term $p_{-1}$ is precisely the local expression of the oper condition: it records that the induced maps
\(
E_i/E_{i-1}\xrightarrow{\sim}(E_{i+1}/E_i)\otimes K_C
\)
are everywhere nonzero, hence isomorphisms.
\end{remark}

\begin{remark}
More generally, let $G$ be a complex simple Lie group with Lie algebra $\mathfrak g$, and fix a Borel subgroup $B\subset G$ with unipotent radical $N$. 
Choose Chevalley generators $\{f_i\}$ for the negative simple root spaces and set
\(
p_{-1}:=\sum_i f_i,
\)
the principal nilpotent element. 
Then a $G$-oper is defined as a triple $(\mathcal F,\nabla,\mathcal F_B)$, consisting of a principal $G$-bundle, a connection, and a $B$-reduction, such that locally in a coordinate $z$ one has
\[
\nabla=\partial_z+p_{-1}+v(z),
\qquad v(z)\in\mathfrak b,
\]
up to gauge transformation by $N$. 
In this paper, however, we only need the cases $G=\mathrm{GL}_n,\mathrm{SL}_n,\mathrm{PSL}_n$.
\end{remark}

For simplicity, one specializes to $G=\mathrm{SL}_n$ or $\mathrm{PSL}_n$, so that
$\mathfrak g=\mathfrak{sl}_n$. 
Let
\[
V_{\mathrm{can}}:=\ker(\operatorname{ad} p_1)\cap \mathfrak n
\]
be the space of $\operatorname{ad} p_1$-invariants in $\mathfrak n$, where
$(p_{-1},2\check\rho,p_1)$ is the principal $\mathfrak{sl}_2$-triple.

\begin{lemma}[\cite{BD05}]\label{lem:DS-gauge}
Each $N$-gauge equivalence class of local operators of the form~\eqref{eq:oper-local-DS}
contains a unique representative of the form
\[
\nabla=\partial_z+p_{-1}+v(z),
\qquad v(z)\in V_{\mathrm{can}}.
\]
This representative is called the \emph{Drinfeld--Sokolov gauge} of the oper.
\end{lemma}

We denote by $\mathrm{Op}_{\mathrm{GL}_n}(C)$, $\mathrm{Op}_{\mathrm{SL}_n}(C)$, and $\mathrm{Op}_{\mathrm{PSL}_n}(C)$ the corresponding moduli spaces of opers on $C$, that is, the space of gauge-equivalence classes of triples $(E,\nabla,\{E_i\})$. 
For a complex semisimple Lie group $G$, the moduli space of $G$-opers is denoted by $\mathrm{Op}_G(C)$. 
A standard fact that will be relevant to us is that $\mathrm{Op}_G(C)$ embeds as a holomorphic Lagrangian subvariety of the de Rham moduli space $\mathcal M_{dR}(G)$.

The next proposition recalls the description of $\mathrm{SL}_n$- and $\mathrm{PSL}_n$-opers in terms of scalar differential operators.
\begin{proposition}[\cite{We15} Section 4.3.2, \cite{BD05} c.2.1]\label{p:oper equ d.o.}
The local system underlying an $\mathrm{SL}_n$-oper is irreducible. 
Moreover, an \( \mathrm{SL}_n \)-oper \( (E, \nabla, \{ E_i \}) \) is equivalent to an \( n \)-th order differential operator with nowhere vanishing principle symbol, 
\begin{equation}\label{e:diff oper for sln oper}
\mathbf{D}: \mathcal{L} \;\longrightarrow\; \mathcal{L} \otimes K_C^{n},
\end{equation}
where \( \mathcal{L} = E_n / E_{n-1} \) is a line bundle satisfying \( \mathcal{L}^n \cong K_C^{-\binom{n}{2}} \). Similarly, a $\mathrm{PSL}_n$-oper is described by an equivalence classes of \( n \)-th order differential operators,
\[
\mathbf{D} : K_C^{-\frac{n-1}{2}} \;\longrightarrow\; K_C^{\frac{n+1}{2}},
\]  
up to the choice of a square root $K_C^{1/2}$.
\end{proposition}

\begin{proof}[Proof sketch]
We only outline the construction of the differential operator $\mathbf{D}$ from a given oper. Note that one can write the oper in Drinfeld--Sokolov gauge as
\[
\nabla=\partial_z+p_{-1}+v(z),
\qquad v(z)\in V_{\mathrm{can}},
\]
then the flatness equation $\nabla s=0$ is equivalent to a single $n$-th order scalar equation for the last component of $s$, and this scalar equation is precisely $\mathbf{D}y=0$. The converse construction is standard and can be found in \cite[\S 4.3.2]{We15}.
\end{proof}

The next proposition explains how the coefficients of the scalar operator encode a projective connection together with higher holomorphic differentials.
\begin{proposition}[\cite{We15} Theorem 4.3, \cite{DIZ91}]\label{prop: k differential}
Let \( \mathbf{D}: \mathcal{L} \to \mathcal{L} \otimes K_C^{n} \) be the differential operator in \eqref{e:diff oper for sln oper}, with \( \mathcal{L} = K_C^{-(n-1)/2} \) after fixing a theta characteristic. With respect to a local coordinate $z$ and a local trivialization of $\mathcal{L}$, one may write
\begin{equation}\label{e:local form for D}
\mathbf{D}
=
\left(
\partial_z^n + Q_2(z)\partial_z^{n-2}+\cdots+Q_n(z)
\right)(dz)^n.
\end{equation}
Then the following hold:
\begin{enumerate}
    \item The coefficient \(12 Q_2/n(n^2-1)\)
    defines a projective connection on $C$.

    \item For each $k\ge 3$, there exists a universal differential polynomial
    \(
    w_k=f_k(Q_2,\dots,Q_k)
    \)
    in the coefficients $Q_2,\dots,Q_k$ and their derivatives, such that
    \(
    w_k\in H^0(C,K_C^k).
    \)

    \item Conversely, a projective connection together with differentials
    \[
    w_k\in H^0(C,K_C^k),\qquad 3\le k\le n,
    \]
    determines uniquely such an operator $\mathbf{D}$.
\end{enumerate}
\end{proposition}

For convenience, recall that a projective connection on $C$ is given by a collection of local holomorphic functions $\{S(z)\}$ such that under a change of coordinate $w=w(z)$ one has
\[
S(w)\,(w')^2=S(z)-\{w,z\},
\]
where $\{w,z\}$
is the Schwarzian derivative. The first few expressions for the differentials $w_k$ are
\begin{align}\label{e:W alg relations}
\begin{split}
w_2 &= Q_2,\\
w_3 &= Q_3-\frac{n-2}{2}\,Q_2',\\
w_4 &= Q_4-\frac{n-3}{2}\,Q_3'
+\frac{(n-2)(n-3)}{10}\,Q_2''
-\frac{(n-2)(n-3)(5n+7)}{10n(n^2-1)}\,Q_2^2.
\end{split}
\end{align}

As an immediate consequence, one obtains the moduli space of opers is affine.
\begin{theorem}[\cite{BD05} §C.3.4x] 
There is an isomorphism
\[
\mathrm{Op}_{\mathrm{PSL}_n}(C) \cong  \mathrm{Op}_{\mathrm{PSL}_2}(C) \times \bigoplus_{j=3}^n H^0(C, K_C^j),
\]
where $\mathrm{Op}_{\mathrm{PSL}_2}(C)$ may be identified with the space $\mathrm{Proj}(C)$ of projective connections on $C$, which is an $H^0(C,K_C^2)$-torsor. 
In particular, $\mathrm{Op}_{\mathrm{PSL}_n}(C)$ is an affine space modeled on the Hitchin base $B$.
\end{theorem}

For simplicity, we fix a reduced effective divisor
\(
D=\sum_{i=1}^d p_i
\)
on \(C\). 
We next introduce meromorphic opers with regular singularities along \(D\). 
The same constructions can be extended to the case of a non-reduced divisor.
\begin{definition}[\cite{Fr05}, \cite{BD05} §C.4.2]\label{def:global-meromorphic-oper}
One defines
\(
\mathrm{Op}_{\mathrm{PSL}_n,D}^{\mathrm{RS}}(C)
\)
to be the space of meromorphic \(\mathrm{PSL}_n\)-opers on \(C\setminus D\) with regular singularities along \(D\). 
Equivalently, these are \(\mathrm{PSL}_n\)-opers on \(C\setminus D\) such that for each \(p_i\in D\), after choosing a local coordinate \(z\) centered at \(p_i\), the oper is locally gauge equivalent to a connection of the form
\begin{equation}\label{eq:global-RS-oper}
\nabla
=
\partial_z+\frac{1}{z}\bigl(p_{-1}+v(z)\bigr),
\qquad
v(z)\in V_{\mathrm{can}}(\mathcal{O}_{U_i}),
\end{equation}
for some sufficiently small neighborhood \(U_i\) of \(p_i\).

For
\(
\boldsymbol{\check\lambda}=(\check\lambda_1,\dots,\check\lambda_d)\in \mathfrak h^d,
\)
with each \(\check\lambda_i\) being a dominant integral coweight,
we denote by
\[
\mathrm{Op}_{\mathrm{PSL}_n,D}^{\mathrm{RS}}(C)_{\boldsymbol{\check\lambda}}
\subset
\mathrm{Op}_{\mathrm{PSL}_n,D}^{\mathrm{RS}}(C)
\]
the subspace of meromorphic opers, such that, for every \(i=1,\dots,d\),
\[
\mathrm{Res}_{p_i}(\nabla)=-\check\lambda_i-\check\rho \in \mathfrak h/W.
\]
\end{definition}

\begin{definition}[\cite{Fr05}]\label{def:global-apparent-oper}
For
\(
\boldsymbol{\check\lambda}=(\check\lambda_1,\dots,\check\lambda_d)\in \mathfrak h^d,
\)
one defines
\(
\mathrm{Op}_{\mathrm{PSL}_n,D}(C)_{\boldsymbol{\check\lambda}}
\)
to be the space of \(\mathrm{PSL}_n\)-opers on \(C\setminus D\) such that for each \(p_i\in D\), after choosing a local coordinate \(z\) centered at \(p_i\), the oper is locally gauge equivalent to
\begin{equation}\label{eq:global-apparent-oper}
\nabla
=
\partial_z+\sum_{j=1}^{n-1} z^{\langle \alpha_j,\check\lambda_i\rangle}f_j + u(z),
\qquad
u(z)\in \mathfrak b(\mathcal{O}_{U_i}),
\end{equation}
where \(\alpha_j\in \mathfrak h^*\) are the simple positive roots of \(\mathfrak{sl}_n\). Later on, we will refer to 
\(
\mathrm{Op}_{\mathrm{PSL}_n,D}(C)_{\boldsymbol{\check\lambda}}
\)
as \emph{opers with apparent singularities of type \(\boldsymbol{\check\lambda}\)}.
\end{definition}

The relation between the two moduli spaces is given by the following lemma.
\begin{lemma}[\cite{Fr05} Lemma 1.2]\label{lem:global-apparent-mono}
There are natural inclusions
\[
\mathrm{Op}_{\mathrm{PSL}_n,D}(C)_{\boldsymbol{\check\lambda}}
\subset
\mathrm{Op}_{\mathrm{PSL}_n,D}^{\mathrm{RS}}(C)_{\boldsymbol{\check\lambda}}
\subset
\mathrm{Op}_{\mathrm{PSL}_n}(C\setminus D).
\]
Moreover,
\(
\mathrm{Op}_{\mathrm{PSL}_n,D}(C)_{\boldsymbol{\check\lambda}}
\)
is precisely the subspace of
\(
\mathrm{Op}_{\mathrm{PSL}_n,D}^{\mathrm{RS}}(C)_{\boldsymbol{\check\lambda}}
\)
consisting of opers with trivial local monodromy around each point of \(D\).
\end{lemma}

Finally, the space of meromorphic opers admits the same type of affine description as in the regular case.
\begin{theorem}[\cite{BD05} §C.4.1–C.4.2]\label{thm: singular oper space}
There exists an isomorphism 
\begin{align}\label{e:affine of oper}
\mathrm{Op}_{\mathrm{PSL}_n,D}^{\mathrm{RS}}(C)
  \;\cong\;
  \mathrm{Op}_{\mathrm{PSL}_2,D}^{\mathrm{RS}}(C)\;\times\;
  \bigoplus_{j=3}^n H^0\!\left(C, K_C^j(jD)\right),
\end{align}
where $\mathrm{Op}_{\mathrm{PSL}_2,D}^{\mathrm{RS}}(C)$ is the space of projective connections on $C$
with regular singularities along $D$, hence an $H^0(C,K_C^2(2D))-$torsor. In particular, $\mathrm{Op}_{\mathrm{PSL}_n,D}^{\mathrm{RS}}(C)$ is an affine space modeled on $\bigoplus_{j=2}^n H^0\!\left(C, K_C^j(jD)\right)$.
\end{theorem}

As a meromorphic analogue of \Cref{p:oper equ d.o.} and \Cref{prop: k differential}, each element of \( \mathrm{Op}_{\mathrm{PSL}_n,D}^{\mathrm{RS}}(C) \) corresponds to a differential operator \( \mathbf{D} \). Locally, \(\mathbf{D}\) takes the form \eqref{e:local form for D}, where
where \(12 Q_2/n(n^2-1)\) defines a meromorphic projective connection on \( C \) with at most a second-order pole at \( D \), while for each \(k\ge 3\), the corresponding differential \(w_k\), obtained as in \eqref{e:W alg relations}, lies in \(H^0(C,K_C^k(kD))\).

\section{Lagrangian subvarieties and Lagrangian correspondences for Hitchin moduli spaces}

In this section, we formulate and prove our main results concerning the moduli spaces of stable Higgs bundles. As introduced in \Cref{sect-triples}, the main objects of our consideration are triples of the form $(L \overset{i}{\hookrightarrow} E, \phi)$, where $(E, \phi)$ is a Higgs bundle and $L$ is a line subbundle of $E$. 
After introducing the sections $s_i(\phi) \in \Gamma(C,K_C^{n \choose 2} L^{-n} \det(E))$ induced by such triples, we discuss how generic triples of such form are Hecke transforms of Higgs bundles in the Hitchin section. We then prove the Lagrangian correspondence results in \Cref{intro-main-result-lag-cor-Higgs} and as a corollary construct Lagrangians $\mathbb{L}_H(D)$ in the moduli spaces of Higgs bundles by fixing divisors $D$ on $C$.

\subsection{Higgs bundles with line subbundles}
Given a triple \((i:L\hookrightarrow E,\phi)\) with \((E,\phi)\in \mathcal M_H(n,\ell)\), for each \(j\ge 0\) consider the morphism
\[
\phi^j\circ i:L\xrightarrow{i}E\xrightarrow{\phi}E\otimes K_C\xrightarrow{\phi\otimes \mathbbm 1}\cdots \longrightarrow E\otimes K_C^j.
\]
Define an \(\mathcal O_C\)-linear morphism
\begin{subequations}\label{eqn-def-s-i-phi}
\begin{equation}
s_i(\phi)=i \wedge (\phi \circ i) \wedge \cdots \wedge (\phi^{n-1} \circ i):L^{n}\longrightarrow \det(E)K_C^{\binom{n}{2}},
\end{equation}
by
\begin{equation}
\sigma_1\otimes\cdots\otimes \sigma_n\longmapsto \det\bigl(i(\sigma_1),\phi\circ i(\sigma_2),\phi^2\circ i(\sigma_3),\dots,\phi^{n-1}\circ i(\sigma_n)\bigr),
\end{equation}
\end{subequations}
for local sections \(\sigma_k\in \Gamma(U,L)\). If \(s_i(\phi)\not\equiv 0\), we define
\[
D_i(\phi):=\operatorname{div}\bigl(s_i(\phi)\bigr).
\]
By the spectral correspondence, \((E,\phi)\) determines spectral data \((\widetilde C,\mathcal L)\). In the rest of this section we always assume that \(\widetilde C\) is smooth, and write \(\pi:\widetilde C\to C\). Consider the composition
\begin{equation}\label{eqn-composition-tilde-D}
f:\pi^*L\xrightarrow{\pi^*i}\pi^*E\longrightarrow \mathcal L.
\end{equation}
If \(f\not\equiv 0\), we define an effective divisor \(\widetilde D_i(\phi)\) on \(\widetilde C\) (possibly empty) by the short exact sequence
\begin{equation}\label{e: recons spectral line bundle}
0\to \pi^*L\xrightarrow{\,f\,}\mathcal L\to \mathcal L|_{\widetilde D_i(\phi)}\to 0,
\qquad\text{equivalently}\qquad
\mathcal L\cong \pi^*L\bigl(\widetilde D_i(\phi)\bigr).
\end{equation}

\begin{lemma}\label{l:section not zero}
Given a triple \((i:L\hookrightarrow E,\phi)\) with \((E,\phi)\in \mathcal M_H(n,\ell)\), assume that the corresponding spectral curve \(\widetilde C\) is smooth. Then the composition \(f:\pi^*L\to \mathcal L\) is never identically zero. Equivalently, \(s_i(\phi)\not\equiv 0\).
\end{lemma}

\begin{proof}
It suffices to show that for any \(p\in C\) outside the branch locus of \(\pi\), the restriction \(f|_{\pi^{-1}(p)}\) is nonzero. Write \(\pi^{-1}(p)=\{p_1,\dots,p_n\}\), and let \(\lambda_{p_i}\) be the eigenvalue of \(\phi|_p\) corresponding to \(p_i\). Since \(p\) is not a branch point, we have a decomposition
\(
E|_p=\bigoplus_{i=1}^n \mathbb C e_i,
\)
where \(e_i\) is an eigenvector with eigenvalue \(\lambda_{p_i}\). Suppose by contradiction that \(f|_{\pi^{-1}(p)}=0\). Then for each \(i\), the image of \(L|_p\) in \(\mathcal L|_{p_i}\) vanishes, equivalently
\[
L|_p\subset \operatorname{Im}\bigl((\phi-\lambda_{p_i}\mathrm{id})|_p\bigr)=\bigoplus_{j\ne i}\mathbb C e_j \subset E|_p.
\]
Since this holds for every \(i\), we get
\[
L|_p\subset \bigcap_{i=1}^n \bigoplus_{j\ne i}\mathbb C e_j=0,
\]
a contradiction because \(L|_p\) is one-dimensional.
\end{proof}

\begin{proposition}\label{prop-Nm(D)}
Given a triple \((i:L\hookrightarrow E,\phi)\) with \((E,\phi)\in \mathcal M_H(n,\ell)\), assume that the corresponding spectral curve \(\widetilde C\) is smooth. Then
\[
\pi_\ast \bigl(\widetilde D_i(\phi)\bigr)=D_i(\phi).
\]
Conversely, the embedding \(i:L\hookrightarrow E\) can be recovered from \(\widetilde D_i(\phi)\).
\end{proposition}

\begin{proof}
By \Cref{l:section not zero}, \eqref{e: recons spectral line bundle} always holds. Pushing forward along \(\pi:\widetilde C\to C\), we obtain
\[
0\to \pi_*\pi^*L\xrightarrow{\,\pi_*f\,}\pi_*\mathcal L\to \pi_*\bigl(\mathcal L|_{\widetilde D_i(\phi)}\bigr)\to 0.
\]
By the projection formula and the spectral correspondence,
\[
\pi_*\pi^*L\cong L\otimes \pi_*\mathcal O_{\widetilde C}\cong L\otimes \bigoplus_{j=0}^{n-1}K_C^{-j},
\qquad
\pi_*\mathcal L\cong E.
\]
Under these identifications, \(\pi_*f\) is precisely \(\bigoplus_{j=0}^{n-1}\phi^j\circ i\). Hence \(\det(\pi_*f)=s_i(\phi)\), so its vanishing divisor is \(D_i(\phi)\). On the other hand, the cokernel is the torsion sheaf \(\pi_*(\mathcal L|_{\widetilde D_i(\phi)})\), whose associated divisor on \(C\) is \(\pi_\ast(\widetilde D_i(\phi))\). Therefore
\(
\pi_\ast \bigl(\widetilde D_i(\phi)\bigr) = D_i(\phi).
\)
Finally, \(\widetilde D_i(\phi)\) determines the inclusion
\(
\pi^*L\cong \mathcal L(-\widetilde D_i(\phi))\hookrightarrow \mathcal L.
\)
Pushing forward gives \(\pi_*f:\pi_*\pi^*L\to \pi_*\mathcal L=E\), and restricting \(\pi_*f\) to the first direct summand \(L\subset \pi_*\pi^*L\) recovers the original embedding \(i:L\hookrightarrow E\).
\end{proof}

\paragraph{Filtrations induced by triples}
Given a triple \((i:L\hookrightarrow E,\phi)\), together with the divisor \(D_i(\phi)\), recall that for each \(j\ge 0\) the morphism \(\phi^j\circ i\) can be viewed as
\(
\phi^j\circ i:L K_C^{-j}\longrightarrow E.
\)
For \(j\ge 1\), define
\[
L_j:=\operatorname{Im}(\phi^{j-1}\circ i)\subset E,
\]
which is a rank-one subsheaf of \(E\), with \(L_1=L\). Then consider the subsheaves
\[
\mathcal F_j:=\operatorname{Span}_{\mathcal O_C}\{L_1,\dots,L_j\}\subset E,\qquad j=1,\dots,n.
\]
By construction, each \(\mathcal F_j\) is a locally free \(\mathcal O_C\)-submodule of \(E\) of rank \(j\), since away from the divisor \(D_i(\phi)\) the subsheaves \(L_1,\dots,L_j\) are linearly independent. It follows that \(\mathcal F_j\) splits as
\begin{equation}\label{split-subbundle}
\mathcal F_j=\bigoplus_{s=1}^j L_s\cong \bigoplus_{s=1}^j LK_C^{-s+1}.
\end{equation}
Here the second isomorphism follows from the fact that \(\phi\) induces an \(\mathcal O_C\)-isomorphism
\begin{equation}\label{eqn-induced-graded-morphism}
\phi:L_s\xrightarrow{\ \cong\ }L_{s+1}K_C.
\end{equation}
Indeed, surjectivity is immediate from the definition of \(L_{s+1}\), while injectivity follows because \(\ker(\phi|_{L_s})\) is supported on \(D_i(\phi)\); hence \(\phi|_{L_s}\) is injective in the sheaf-theoretic sense.

Therefore, starting from a triple \((i:L\hookrightarrow E,\phi)\), we obtain an increasing filtration of subsheaves compatible with \(\phi\),
\begin{equation}\label{eqn-sheaves-filtration Higgs}
0=\mathcal F_0\subset \mathcal F_1\subset \cdots \subset \mathcal F_n\subset E.
\end{equation}

\medskip
\begin{proposition}\label{prop-filtration-Higgs}
Given a triple $(i:L \hookrightarrow E, \phi)$ with $D=D_i(\phi)$, over $C-D$ we have a full filtration of subbundles,
$$
0 \subset \mathcal{F}_1\mid_{C-D} \subset \mathcal{F}_2\mid_{C-D} \subset \dots \subset \mathcal{F}_n\mid_{C-D}=E\mid_{C-D}. 
$$
When $D$ is a reduced divisor, each \(\mathcal{F}_i\) with \(i<n\) is a rank-\(i\) subbundle \(E_i \subset E\), 
and one has a full filtration of subbundles for $E$, 
\begin{align}\label{filtration-Higgs}
0 \subset i(L) = E_1 \subset E_2 \subset \dots \subset E_{n-1} \subset E. 
\end{align}
For $1 \leq j \leq n-1$, $E_j$
are split bundles according to \eqref{split-subbundle}.
The Higgs field $\phi$ is compatible with the filtration \eqref{filtration-Higgs}, i.e. $\phi(E_j) \subset E_{j+1}\otimes K_C$.
There exists local frames adapted to \eqref{filtration-Higgs} in which $\phi$ is gauge equivalent to the form
\begin{equation} \label{eqn-Higgs-gauged}
    \phi = \begin{pmatrix}
        &0  &\ast &\dots &\ast &\ast &\ast  \\
        &1 &0 &\dots &\ast &\ast &\ast \\
        &0 &1 &\dots &\ast &\ast &\ast \\
        &\vdots  &\vdots &\dots &\vdots  &\vdots &\vdots \\
        &0  &0 &\dots &0 &\ast &\ast \\
        &0  &0 &\dots &1 &-a_n &\ast \\
        &0  &0 &\dots &0 &s_i(\phi) &a_n 
    \end{pmatrix}.
\end{equation}
In addition, for $p < D$, we have $a_n(p)$ as the eigenvalue 
of $\phi\mid_p$ that determines $\widetilde{D}_i(\phi)$ in \eqref{e: recons spectral line bundle}.
\end{proposition}
\begin{proof}
The first statement is true by construction. For the second statement, let \(I(p)\) be the smallest index \(i\) such that 
\(\dim \mathcal{F}_i\mid_p < i\). We claim that \(I(p) = n\), for all $p<D$. This implies that the filtration \eqref{filtration-Higgs} 
is indeed a filtration by subbundles.

By way of contradiction, assume \(I(p)=m < n\) for some $p$. 
Equivalently, $i \wedge \phi\circ i \wedge \dots \wedge \phi^{m-1}\circ i$ vanishes at $p<D$. 
This implies that $s_i(\phi)$ acquires a zero of order at least $n-m+1$ at $p$, which contradicts the fact that $D$ is reduced.

By construction, in the local frames defined by the line bundle summands of the split bundles $\mathcal{F}_j$, $1 \leq j \leq n-1$, the Higgs field $\phi$ takes the form as in \eqref{split-subbundle}. 

Finally, for \(p < D\), let
\(\widetilde p \in \pi^{-1}(p)\cap\widetilde{D}_i(\phi)\) be the corresponding eigenvalue
\(a_n(p)\) of \(\phi|_p\). Moreover,
\[
\pi^*(L)|_{\tilde p}
\subset
\operatorname{Im}\bigl(\pi^*(\phi)-a_n(p)\bigr)\big|_{\tilde p}.
\]
From the local form \eqref{eqn-Higgs-gauged} and the rank comparison, 
\(\operatorname{Im}\bigl(\pi^*(\phi)-a_n(p)\bigr)\big|_{\tilde p} = \pi^*E_{n-1}|_{\tilde p}\). Therefore, it follows that the two nontrivial diagonal elements in $\phi$ should be $a_n(p)$.

\end{proof}

\begin{remark}
The conclusions of \eqref{filtration-Higgs} and \eqref{eqn-Higgs-gauged} admit a natural generalization to the case where the divisor \(D_i(\phi)\) associated to a triple \((i:L\hookrightarrow E,\phi)\) is non-reduced.
\end{remark}

\subsection{Hecke transform for Higgs bundles}
\paragraph{Hecke transformations}
We now recall the Hecke transformation for vector bundles and Higgs bundles. Let \(c\in C\), and let \(E\) be a holomorphic vector bundle on \(C\). For a subspace
\(
V\subset E|_c,
\)
one defines the Hecke transformation of $E$ at $V$ as follows:

\begin{definition}
The \emph{Hecke transform} $E_V$ of \(E\) at \(V\subset E|_c\) is characterized by the short exact sequence
\[
0 \longrightarrow E_V \longrightarrow E \longrightarrow W \longrightarrow 0.
\]
where \(W:=(E|_c/V)\otimes_{\mathbb C}\mathcal O_c\) is the skyscraper sheaf supported at \(c\) with fiber \(E|_c/V\). 
\end{definition}
It is clear that \(E_V\) is again a vector bundle of the same rank as $E$ , and
\[
\deg(E_V)=\deg(E)-\bigl(\operatorname{rk}(E)-\dim V\bigr).
\]
Moreover, one always has a natural inclusion
\(
E(-c)\subset E_V\subset E.
\)

\begin{example}
If \(E=F\oplus G\) and \(V=F|_c\subset E|_c\), then
\(
E_V \cong F\oplus G(-c).
\)
\end{example}

Let \((E,\phi)\) be a Higgs bundle on \(C\), with \(V\subset E|_c\)
be a \(\phi_c\)-invariant subspace. Then \(\phi_c\) induces an endomorphism on the quotient
\[
\overline{\phi}_c:W\longrightarrow W\otimes K_C|_c.
\]

\begin{definition}
The \emph{Hecke transform} of the Higgs bundle \((E,\phi)\) at an \(\phi_c\)-invariant subspace
\(V\subset E|_c\) is the unique Higgs bundle \((E_V,\phi_V)\) making the diagram
\[
\begin{tikzcd}
0 \arrow[r] & E_V \arrow[r] \arrow[d,"\phi_V"'] & E \arrow[r,] \arrow[d,"\phi"] & W_c \arrow[r] \arrow[d,"\overline{\phi}_c"] & 0 \\
0 \arrow[r] & E_V\otimes K_C \arrow[r] & E\otimes K_C \arrow[r] & W_c\otimes K_C \arrow[r] & 0
\end{tikzcd}
\]
commutative. 
\end{definition}

By construction, the underlying bundle of the Hecke transform at $V$ of $(E, \phi)$ is the Hecke transform \(E_V\). An important feature is that the Hecke transform of Higgs bundles preserves the characteristic polynomial. Since \(\phi_V\) and \(\phi\) agree away from the point \(c\), their characteristic polynomials coincide globally.

\paragraph{Basic properties of Hecke transforms.}
Let \((E,\phi)\) be a rank-\(n\) Higgs bundle equipped with a full filtration
\begin{equation}\label{filtration}
0=E_0\subset E_1\subset \cdots \subset E_n=E
\end{equation}
which is compatible with the Higgs field, in the sense that
\begin{equation}\label{compatability}
\phi(E_i)\subset E_{i+1}\otimes K_C,
\qquad 0\le i\le n-1.
\end{equation}
For each \(1\le i\le n-1\), the Higgs field induces a morphism between the successive quotients
\begin{equation}\label{morphism-c-i}
c_i:{E_i}/{E_{i-1}}
\longrightarrow
{E_{i+1}}/{E_i}\otimes K_C.
\end{equation}
Fix a point \(p\in C\) such that all fiber maps \((c_i)_p\) are nonzero.

The following result shows that the Hecke transform of a filtered Higgs bundle again carries a natural compatible filtration.

\begin{proposition}[{\cite[Proposition~4.5]{HH22}}]\label{HH-Hecke-prop}
Let \((E,\phi)\) be a Higgs bundle with a full compatible filtration
\eqref{filtration}--\eqref{compatability}, and let \(p\in C\) be a point at which all morphisms
\eqref{morphism-c-i} are nonvanishing.
Let \(V\subset E|_p\) be a \(k\)-dimensional eigenspace of \(\phi|_p\).
Then the Hecke transform
\(
(E',\phi') \coloneqq (E_V, \phi_V)
\)
carries an induced full filtration
\begin{equation}\label{induced-filtration}
0=E_0'\subset E_1'\subset \cdots \subset E_n'=E',
\end{equation}
where, for \(1\le i\le n-1\), \(E_i'\) is the Hecke transform of \(E_i\) at
\(
V_i:=E_i|_p\cap V\subset E_i|_p.
\)
Equivalently, for each \(i\) there is a commutative diagram with exact rows
\[
\begin{tikzcd}
0 \arrow{r} & E_i' \arrow{r} \arrow[hook]{d} & E_i \arrow{r} \arrow[hook]{d}
& E_i|_p/V_i \arrow{r} \arrow[hook]{d} & 0 \\
0 \arrow{r} & E_{i+1}' \arrow{r} & E_{i+1} \arrow{r}
& E_{i+1}|_p/V_{i+1} \arrow{r} & 0 .
\end{tikzcd}
\]
The induced filtration \eqref{induced-filtration} is compatible with \(\phi'\).
Moreover, if \(c_i'\) denotes the morphism induced by \(\phi'\) on the successive quotients of
\eqref{induced-filtration}, then \( c_i'=c_i, \text{ for } i\neq n-k,\)
while \( c_{n-k}'=s_p\,c_{n-k},\)
where \(s_p\) is the canonical section of \(\mathcal O_C(p)\) vanishing simply at \(p\).
Finally, the Hecke transform of \((E',\phi')\) at the subspace \(E'_{n-k}|_p\subset E'|_p\) is
canonically isomorphic to
\(
(E,\phi)\otimes \mathcal O_C(-p).
\)
\end{proposition}

We now apply this general result to the filtration arising from a triple.
By Proposition~\ref{prop-filtration-Higgs}, a triple
\(
(i:L\hookrightarrow E,\phi)
\)
with divisor \(D_i(\phi)\) determines a full compatible filtration of \(E\).
This filtration can be obtained from the standard split
filtration on an element in the Hitchin section by successive Hecke transforms at the eigenlines
prescribed by the divisor \(\widetilde D_i(\phi)\) as we explain now.

\begin{definition}
Fix a line bundle \(L\). Define the map
\[
\mathfrak s_{L}:\bigoplus_{j=2}^n H^0(C,K_C^j)\longrightarrow \mathcal M_H(n,\ell), \quad \text{with }  h \circ \mathfrak s_{L}  =\operatorname{id}
\]
by sending \((a_2,\dots,a_n)\) to the Higgs bundle \((E_L,\phi_L)\), where
\[
E_L=\bigoplus_{j=0}^{n-1} L \otimes K_C^{-j},
\qquad
\phi_L=
\begin{pmatrix}
0 & 0 & \cdots & 0 & a_n\\
1 & 0 & \cdots & 0 & a_{n-1}\\
0 & 1 & \ddots & \vdots & \vdots\\
\vdots & \vdots & \ddots & 0 & a_2\\
0 & 0 & \cdots & 1 & 0
\end{pmatrix}.
\]
We denote by
\(
\mathrm{Hit}_L:=\operatorname{Im}(\mathfrak s_L)\subset \mathcal M_H(n,\ell)
\)
the Hitchin section determined by \(L\).
\end{definition}

\begin{proposition}\label{l:Hecke-triple-Hitchin}
Let \((i:L\hookrightarrow E,\phi)\) be a triple with associated divisor 
\(
\widetilde D_i(\phi)=\sum_{k=1}^d \widetilde p_k
\)
on the smooth spectral curve \(\pi:\widetilde C\to C\) associated to \((E,\phi)\), and assume that \(\widetilde D_i(\phi)\) is disjoint from the ramification divisor of \(\pi\). 
Let 
\(
D=\pi_\ast \bigl(\widetilde D_i(\phi)\bigr)=\sum_{k=1}^d p_k.
\)
Then \((E,\phi)\) is obtained from \(\mathfrak s_{L(D)}(\widetilde C)\), which is the Higgs bundle in the Hitchin section with spectral curve \(\widetilde C\), by successive lower Hecke transforms at the one-dimensional eigenspaces  
\[
V_k\subset E_{L(D)}|_{p_k},\qquad k=1,\dots,d,
\]
determined by the points \(\widetilde p_k\in \pi^{-1}(p_k)\). Moreover, the resulting filtration coincides with \eqref{filtration-Higgs}. Equivalently, \((E,\phi)\) is obtained from \(\mathfrak s_L(\widetilde C)\) by successive upper Hecke transforms along the corresponding hyperplanes.
\end{proposition}

\begin{proof}
    The proposition follows from the following general fact. Given \((E,\phi)\) with smooth spectral curve
\(
\pi:\widetilde C\to C,
\)
let \(\mathcal L\) be the spectral line bundle corresponding to \((E,\phi)\). At a point \(c\in C\) away from the branch locus, choosing a \(\phi_c\)-invariant subspace
\(
V\subset E|_c
\)
amounts to choosing a sum of eigenspaces of \(\phi_c\), equivalently an effective divisor
\(
D\subset \pi^{-1}(c)
\)
on the spectral curve. Under the spectral correspondence, the Hecke transform of \((E,\phi)\) at \(V\) corresponds to changing the spectral line bundle to
\(
\mathcal L(-D)
\)
on \(\widetilde C\). Therefore, together with \eqref{e: recons spectral line bundle}, the result follows.
\end{proof}

\subsection{Lagrangian subvarieties in Hitchin moduli spaces}

Consider the moduli spaces
$\mathcal{M}_{H}(n, \ell)$ and $\mathcal{M}_H(n, \Lambda)$ 
of $\mathrm{GL}_n$-Higgs bundles $(E, \phi)$ with $\deg(E) = \ell$ 
and $\mathrm{SL}_n$-Higgs bundles $(E, \phi)$ with $\det(E) = \Lambda$ respectively.
We define subvarieties of these moduli spaces 
by fixing the divisors $D_i(\phi)$ constructed in \eqref{eqn-def-s-i-phi} and studying triples with divisor $D_i(\phi)$.

\begin{definition}\label{def: Lag in Higgs}
Given an integer $\ell$, an effective divisor $D$ and a line bundle $L$ with
\begin{equation}\label{eqn-GLn-constraint-d}
\deg(D) = \ell + n(n-1)(g-1) - n \deg(L),
\end{equation}
let \(\mathbb{L}_{{H}}(L, D)\) be the subvariety of \(\mathcal{M}_{H}(n, \ell) \) defined as 
\begin{align}\label{eqn-L_D higgs}
    \mathbb{L}_{H}(L, D) = 
\left\{ (E, \phi) \in \mathcal{M}_{H}(n, \ell) 
\mid \exists i: L \hookrightarrow E, \text{ } D_i(\phi) = D\right\}.
\end{align}

Given a fixed line bundle $\Lambda$, 
an effective divisor $D$ and a line bundle $L$ with
\begin{equation}\label{eqn-SLn-constraint-D}
L^n \cong  \Lambda K_C^{n \choose 2}(-D),    
\end{equation}
let \(\mathbb{L}_{H}(L, D)\) be the subvariety
of \(\mathcal{M}_{H}(n, \Lambda) \) defined in the same way as \eqref{eqn-L_D higgs}.
Let $\mathbb{L}_{H}(D)$ be the union of all $n^{2g}$ such subvarieties, i.e.
\begin{align}
\begin{split}
    \mathbb{L}_{H}(D) &= \bigcup_{L^n \cong  \Lambda K_C^{n \choose 2}(-D)} \mathbb{L}_{H}(L, D) \\
    &= 
\left\{ (E, \phi) \in \mathcal{M}_{H}(n, \Lambda) 
\mid \exists i: L \hookrightarrow E, D_i(\phi) = D\right\}.
\end{split}
\end{align}
\end{definition}

It follows by the preceding definition and \Cref{l:Hecke-triple-Hitchin}, that when \(D=\varnothing\) it holds that 
\(
\mathbb L_H(L,\varnothing)=\mathrm{Hit}_L .
\)
The next proposition shows that \(\mathbb L_H(L,D)\) is related to \(\mathbb L_H(L,\varnothing)\) by upper Hecke transforms along \(D\).

\begin{proposition}\label{p: relation to Hecke}
Assume that \(L\) and the effective divisor \(D\) satisfy the compatibility condition defining \(\mathbb L_H(L,D)\). Then \(\mathbb L_H(L,D)\) is obtained from \(\mathrm{Hit}_L\) by upper Hecke transforms along \(D\). Moreover, there is a natural morphism
\[
\mathfrak r_D:=\mathfrak s_L\circ h|_{\mathbb L_H(L,D)}:\mathbb L_H(L,D)\longrightarrow \mathbb L_H(L,\varnothing)=\mathrm{Hit}_L,
\qquad
(\widetilde C,\mathcal L)\longmapsto \mathfrak s_L(\widetilde C),
\]
and \(\mathfrak r_D\) is finite.
\end{proposition}

\begin{proof}
The Hecke transform statement follows directly from \Cref{l:Hecke-triple-Hitchin}. On the other hand, by \eqref{e: recons spectral line bundle} and \Cref{prop-Nm(D)}, the fiber of \(\mathfrak r_D\) over the Hitchin section \(\mathfrak s_L(\widetilde C)\) is naturally identified with
\[
\left\{
\mathcal O_{\widetilde C}(\widetilde D)\in \operatorname{Pic}(\widetilde C)\;:\;
\widetilde D \text{ is effective},\ \pi_\ast(\widetilde D)=D
\right\},
\]
which is a finite set. In particular, when \(D\) is reduced,  \( \mathfrak r_D \)
is generically \(n^{\deg (D)}\!:\!1\).
\end{proof}

\begin{corollary}\label{c: non-empty and half dim Higgs}
Let \(G=\mathrm{GL}_n\) or \(\mathrm{SL}_n\), and let \(D\) be an effective divisor on \(C\) such that \(L\) and \(D\) satisfy the compatibility condition defining \(\mathbb L_H(L,D)\). Then \(\mathbb L_H(L,D)\subset \mathcal M_H(G)\) is non-empty and has dimension
\[
\dim \mathbb L_H(L,D)=\frac{1}{2}\dim \mathcal M_H(G).
\]
\end{corollary}

\begin{proof}
    The statement follows from \Cref{p: relation to Hecke}, namely the fact that $\mathfrak r_D$ is finite and the known fact that $\mathrm{Hit}_L$ is a half dimensional subvariety of $\mathcal{M}_{H}(G)$.
\end{proof}

As we show in the next subsection it turns out that 

\begin{theorem}\label{thm: Lag Higgs}
The subvariety $\mathbb{L}_{H}(L,D)\subset \mathcal{M}_{H}(n, \ell)$ is holomorphic Lagrangian with respect to the holomorphic symplectic form $\Omega_H$. The same statement holds for $\mathbb L_H(D)\subset\mathcal M_H(n,\Lambda)$.
\end{theorem}

\begin{remark} 
The proof of \Cref{thm: Lag Higgs} actually implies a slight generalization of a result by Hausel-Hitchin which states that Hecke transformations send Lagrangian subvarieties to Lagrangian subvarieties by lifting some generic conditions imposed in \cite[Theorem~4.21]{HH22}. 
Indeed, by \Cref{p: relation to Hecke}, the subvariety \(\mathbb{L}_{H}(L,D)\) is obtained from the Hitchin section \(\mathrm{Hit}_L\) via Hecke transformations, and since \(\mathrm{Hit}_L\) is a holomorphic Lagrangian subvariety so must be \(\mathbb{L}_{H}(L,D)\).
\end{remark}

\subsection{Lagrangian correspondence for Hitchin moduli spaces}
In this subsection we define subvarieties
\(\mathbb{L}_{H}(L,d)\) and \(\mathbb{L}_{H}(d)\) in
\(
\mathrm{Hilb}^d(T^*C) \times \mathcal{M}_{H}(n, \ell)
\)
and
\(
\mathrm{Hilb}^d(T^*C) \times \mathcal{M}_H(n,\Lambda),
\)  
respectively. These subvarieties provide Lagrangian correspondences between $\mathrm{Hilb}^d(T^*C)$ and $\mathcal{M}_{H}(n, \ell)$ and respectively between $\mathrm{Hilb}^d(T^*C)$ and $\mathcal{M}_H(n,\Lambda)$. More precisely,

\begin{definition}\label{def:lag on extended space}
Let $\ell,d \in \mathbb{Z}$ and let $L$ be a line bundle on $C$ satisfying
$$
d = \ell + n(n-1)(g-1) - n\,\deg(L).
$$
We define $\mathbb{L}_H(L,d)$ to be the subvariety of
\(
\operatorname{Hilb}^d(T^\ast C)\times \mathcal{M}_{H}(n, \ell)
\)
consisting of triples $(\widetilde{D},E,\phi)\cong (\widetilde{D},\widetilde{C},\mathcal{L})$ such that
\begin{align}
\mathbb{L}_H(L,d)
:=
\left\{
    (\widetilde{D}, E, \phi)\ \middle|\ 
    \begin{array}{l}
        \widetilde{D}\in \operatorname{Sym}^d(\widetilde{C}) \subset \operatorname{Hilb}^d(T^\ast C),  \\[2pt]
        \exists\, i: L \hookrightarrow E \text{ such that }\mathcal{L} \cong \pi^\ast(L)(\widetilde{D})
    \end{array}
\right\}.
\end{align}
Here we identify $\operatorname{Sym}^d(\widetilde{C})$ with the set of effective divisors of degree $d$ on $\widetilde{C}$, since $\widetilde{C}$ is smooth.

Fix a line bundle $\Lambda$ on $C$, with $\deg(\Lambda)=\ell$. For $d$ satisfying $n \mid d-\ell$, we define $\mathbb{L}_H(d)$ to be the subvariety of
\(
\operatorname{Hilb}^d(T^\ast C)\times \mathcal{M}_H(n,\Lambda)
\)
consisting of triples $(\widetilde{D},E,\phi) \cong (\widetilde{D},\widetilde{C},\mathcal{L}) $ such that
\begin{align}
\mathbb{L}_H(d)
:=
\left\{
    (\widetilde{D}, E, \phi)\ \middle|\ 
    \begin{array}{l}
        \widetilde{D}\in \operatorname{Sym}^d(\widetilde{C}), \exists\, i:L \hookrightarrow E \text{ such that } \\[2pt]
        L^{n} \cong \Lambda K_C^{\binom{n}{2}}\bigl(-\pi_\ast(\widetilde{D})\bigr),  
        \mathcal{L} \cong \pi^\ast(L)(\widetilde{D})
    \end{array}
\right\}.
\end{align}
The map $\pi_\ast:\operatorname{Sym}^d(\widetilde C)\to \operatorname{Sym}^d(C)$, which coincides with the restriction of the norm map to effective divisors, is induced by $\pi:T^\ast C\to C$.
We have also identified $\operatorname{Sym}^d(C)$ with the set of effective divisors of degree $d$ on $C$.
\medskip
Note that if $(\widetilde{D},E,\phi)\in \mathbb{L}_H(L,d)$ (resp.\ $\mathbb{L}_H(d)$), then $(E,\phi)\in \mathbb{L}_H\!\bigl(L,\pi_\ast(\widetilde{D})\bigr)$ (resp.\ $\mathbb{L}_H\!\bigl(\pi_\ast(\widetilde{D})\bigr)$).
\end{definition}

\begin{theorem}\label{thm: Lag corres Higgs}
The subvariety $\mathbb{L}_H(L,d)$ defines a
\emph{Lagrangian correspondence} between
\(
\mathcal{M}_{H}(n, \ell)
\)
and
\(
\mathrm{Hilb}^d(T^*C)
\).
Specifically, $\mathbb{L}_H(L,d)$ is a Lagrangian subvariety of
\(
\mathrm{Hilb}^d(T^*C)\times \mathcal{M}_{H}(n, \ell),
\)
when the ambient space is equipped with the holomorphic symplectic form
\(
\omega_{T^\ast C}^{[d]} \oplus \Omega_H.
\)
The same statement holds with $\mathcal M_H(n,\Lambda)$ in place of $\mathcal M_H(n,\ell)$.
\end{theorem}

Before presenting the proof, we provide a description of $\Omega_H$ for $\mathcal{M}_{H}(n, \ell)$ in terms of spectral data based on \cite{Hur96}. 
Recall that for $b \in B$, 
$
h^{-1}(b) \cong \operatorname{Pic}^{\ell + n(n-1)(g-1)}(C_b),
$
where $C_b$ is the associated spectral curve. At a point $(\widetilde{C}, \mathcal{L}) \in \mathcal{M}_{H}(n, \ell)$, we have the short exact sequence:
\begin{equation} \label{e:ses tangent}
  0 \longrightarrow H^1(\widetilde{C}, \mathcal{O}_{\widetilde{C}}) \longrightarrow {T_{ (\widetilde{C}, \mathcal{L})}\mathcal{M}_{H}} \longrightarrow H^0(\widetilde{C}, K_{\widetilde{C}}) \longrightarrow 0.
\end{equation}
Here, the tangent space to the Hitchin base corresponds to first-order deformations of the spectral curve, which are captured by $H^0(\widetilde{C}, N_{\widetilde{C}/\mathrm{Tot}(K_C)})$. Since $\mathrm{Tot}(K_C)$ has trivial canonical bundle, by the adjunction formula we have:
$
N_{\widetilde{C}/\mathrm{Tot}(K_C)} \cong K_{\widetilde{C}}.
$
Moreover, we have:
$$
H^0(\widetilde{C}, K_{\widetilde{C}}) \cong H^0(C, \pi_* K_{\widetilde{C}}) \cong \bigoplus_{i=1}^n H^0(C, K_C^i).
$$
On the other hand, the tangent space to the Jacobian $\operatorname{Pic}(\widetilde{C})$ is given by $H^1(\widetilde{C}, \mathcal{O}_{\widetilde{C}})$.

The two ends of the exact sequence \eqref{e:ses tangent} are naturally dual via Serre duality. Therefore, any splitting of \eqref{e:ses tangent} defines a natural symplectic form as the canonical skew-symmetric form on a vector space and its dual. For simplicity, we specify the  splitting considered in \cite{Hur96}. 
One extends the line bundle $\mathcal{L}$ from $\widetilde{C}$ to a neighborhood of $\widetilde{C}$ inside $\mathrm{Tot}(K_C)$, thereby producing a background bundle that can be restricted to a family of varying curves. This extension induces a splitting of the exact sequence \eqref{e:ses tangent} and defines a symplectic form $\Omega_S$ on $\mathcal{M}_{H}(n, \ell)$. We detail this construction shortly.

To construct the symplectic form $\Omega_S$, we proceed as follows. Recall the projection 
$\pi: \mathrm{Tot}(K_C) \to C.$ Given a spectral data \((\widetilde{C},\mathcal{L}) \in \mathcal{M}_{H}(n, \ell)\), we choose a set of points \(\{p_1, \dots, p_d\} \subset C\) such that over the open subset \(U = C \setminus \{p_1, \dots, p_d\}\), the canonical bundle \(K_C\) is trivial, and the line bundle \(\mathcal{L}\) can be trivialized on the restriction \(\pi^{-1}(U) \cap \widetilde{C}\).

 Since both \(K_C\) and \(\mathcal{L}\) are trivial over this region, we identify \(\pi^* K_C \cong \mathcal{L}\) on \(\pi^{-1}(U) \cap \widetilde{C}\). We then extend \(\mathcal{L}\) to a line bundle on a neighborhood of \(\widetilde{C}\) in \(\pi^{-1}(U) \subset \mathrm{Tot}(K_C)\) by letting it to be the trivial bundle \(\pi^* K_C\) over this region.

Next, for each point \(p_i\), we assume that it isn't in ramification divisor. Choose a small neighborhood \(V_i \subset C\) centered at \(p_i\), and lift it to disjoint open sets \(V_i^j\) in \(\widetilde{C}\), one for each preimage \(p_i^j\) (\(j = 1, \dots, n\)). Let \(\lambda\) be a coordinate on \(V_i\), and let \(z_i\) denote the fiber coordinate on \(\pi^{-1}(V_i)\). We can identify $V_i$ with $V_i^j$, by sending $ \lambda \to (\lambda,z_i^j(\lambda))$.
Let \( f_i^j(\lambda) \) denote the transition function of \( \mathcal{L} \) from \( \pi^{-1}(U) \) to \( V_i^j \).
 Let \(N_i^j\) be small open neighborhoods of \(V_i^j\) in \(\mathrm{Tot}(K_C)\). Using same transition functions \(f_i^j(\lambda)\), we extend \(\mathcal{L}\) to a line bundle on a neighborhood of \(\widetilde{C}\) in \(\mathrm{Tot}(K_C)\) by gluing with these transition functions from $\pi^{-1}(U)$ to $N_i^j$. This gives an extension of \(\mathcal{L}\).

Now consider a two-parameter deformation \((\widetilde{C}(t_1, t_2), \mathcal{L}(t_1, t_2))\) of the spectral data, with \((\widetilde{C}, \mathcal{L}) = (\widetilde{C}(0,0), \mathcal{L}(0,0))\). Denote \((\widetilde{C}, \mathcal{L})_{t_i} :=\partial_{t_i}(\widetilde{C}(t_1,t_2),\mathcal{L}(t_1,t_2))\Big|_{t_1=t_2=0} \in {T_{ (\widetilde{C}, \mathcal{L})}\mathcal{M}_{H}}.\) Then the symplectic form \(\Omega_S\) is given by:

\begin{equation}\label{e:def sym form spectral}
    \Omega_S \left( (\widetilde{C},\mathcal{L})_{t_1}, (\widetilde{C},\mathcal{L})_{t_2}\right) = \sum_i \mathrm{Res}_{p_i} \sum_{j=1}^{n} \left( \partial_{t_1} z_i^j \frac{\partial_{t_2} f_i^j}{f_i^j} - \partial_{t_2} z_i^j \frac{\partial_{t_1} f_i^j}{f_i^j}  \right) \mathrm{d} \lambda.
\end{equation}
The same formula \eqref{e:def sym form spectral} applies to $\mathcal{M}_H(n,\Lambda)$.

\begin{proposition}[\cite{Hur96}, Proposition 4.15]
The symplectic form $\Omega_S$ on $\mathcal{M}_{H}(n, \ell)$ is independent of the choice of extension of $\mathcal{L}$, since any two such splittings of the tangent sequence differ by a self-adjoint map
$
H^0(\widetilde{C}, K_{\widetilde{C}}) \longrightarrow H^1(\widetilde{C}, \mathcal{O}_{\widetilde{C}}),
$
which preserves the skew-symmetry of $\Omega_S$. Moreover, $\Omega_S = \Omega_H$ on $\mathcal{M}_{H}(n, \ell)$.
\end{proposition}

\begin{proof}[Proof of \Cref{thm: Lag corres Higgs}]
By a straightforward dimension count, one verifies that
$\mathbb{L}_H(L,d)$ is a half-dimensional subspace of 
$\mathrm{Hilb}^d(T^*C)\times \mathcal{M}_{H}(n, \ell)$.
Hence, it suffices to show that $\mathbb{L}_H(L,d)$ is isotropic.

For $(\widetilde{D}, \widetilde{C}, \mathcal{L}) \in \mathbb{L}_H(L,d)$ such that $\widetilde{D} \in \operatorname{Sym}^d(\widetilde{C})\setminus \Delta \subset \operatorname{Hilb}^d(T^\ast C)$ and is away from the ramification divisor $R_\pi$ of spectral cover, one considers a two-parameter deformation
\[
(\widetilde{D}(t_1, t_2), \widetilde{C}(t_1, t_2), \mathcal{L}(t_1, t_2)) \in \mathbb{L}_H(L, d),
\]
with $(t_1, t_2)$ sufficiently small so that $\widetilde{D}(t_1, t_2)$ remains disjoint from $R_{\pi_{t_1,t_2}}$.

To compute \eqref{e:def sym form spectral}, we need to describe the local coordinates $z_i^j(\lambda; t_1, t_2)$ of the deformed spectral curve and the transition functions $f_i^j(\lambda; t_1, t_2)$ of $\mathcal{L}(t_1, t_2)$. Set
\[
D(t_1, t_2) := \pi\left(\widetilde{D}(t_1, t_2)\right) = \sum_i p_i(t_1, t_2),
\quad \lambda(p_i(t_1, t_2)) =u_i(t_1, t_2)
\]
with $u_i(0,0) = 0$, where $(V_i,\lambda)$ is an open chart over $p_i(0,0)$. Let $z_i$ denote the fiber coordinate on $\mathrm{Tot}(K_C)$ over $V_i$. The spectral curve $\widetilde{C}(t_1, t_2)$ in a neighborhood of $\pi^{-1}(V_i)$ is defined by the polynomial:
\(
    E(\lambda, z_i; t_1, t_2) = 0,
\)
with each branch over $V_i$ corresponds to a root
\(
z_i^j(\lambda; t_1, t_2), 
\)
for $j = 1, \dots, n$.
For simplicity, we assume that $\widetilde{D}(t_1,t_2)$ is reduced and lies on the branch corresponding to $z_i^1$. The general case follows from the same computation.

By \eqref{e: recons spectral line bundle}, 
\[
\mathcal{L}(t_1, t_2) \cong \pi_{t_1,t_2}^*(L)(\widetilde{D}(t_1,t_2)),
\]
where $\pi_{t_1,t_2}:  \widetilde{C}(t_1, t_2) \to C$. 
For simplicity, we omit the contribution from $\pi_{t_1,t_2}^*(L)$. Thus the transition functions of $\mathcal{L}(t_1, t_2)$ from $\pi^{-1}(U)$ to $N_i^j$ satisfy
\begin{align*}
f_i^1(\lambda; t_1, t_2) &= E(\lambda, z_i^1(u_i(t_1,t_2); t_1, t_2); t_1, t_2), \\
f_i^j(\lambda; t_1, t_2) &\equiv 1, \quad \text{for } j \neq 1.
\end{align*}

Differentiating the polynomial equation gives
\begin{align*}
    \left. \partial_{t_k} z_i^1(\lambda)  \right|_{(t_1,t_2)=(0,0)}
    = - \frac{\left( \partial_{t_k} E \right) (\lambda, z_i^1(\lambda))}{\left( \partial_z E \right) (\lambda, z_i^1(\lambda))} , \text{ for } k=1,2.
\end{align*}
and
\begin{align*}
    \left. \frac{\partial_{t_k} f_i^1}{f_i^1} \right|_{(t_1,t_2)=(0,0)} 
    & = \frac{
        \left( \partial_z E \right)(\lambda, z_i^1(0)) \cdot \left( \partial_{t_k} z_i^1(0) + \partial_\lambda z_i^1(0) \cdot \partial_{t_k} u_i(0,0) \right) + \left( \partial_{t_k} E \right)(\lambda, z_i^1(0))
    }{
        E(\lambda, z_i^1(0))
    }, \\
\left. \frac{\partial_{t_k} f_i^j}{f_i^j} \right|_{(t_1,t_2)=(0,0)} &= 0, \text{ for } j \neq 1, k=1,2.
\end{align*}
Substituting into \eqref{e:def sym form spectral}, we compute
\begin{align*}
&\Omega_H \left( (\widetilde{C},\mathcal{L})_{t_1}, ( \widetilde{C},\mathcal{L})_{t_2}\right)
= \sum_{i=1}^d \operatorname{Res}_{p_i} \left(
    \partial_{t_1} z_i^1 \cdot \frac{\partial_{t_2} f_i^1}{f_i^1}
    - \partial_{t_2} z_i^1 \cdot \frac{\partial_{t_1} f_i^1}{f_i^1}
\right) \, d\lambda \\
=& \sum_{i=1}^d \operatorname{Res}_{\lambda = 0} 
\left[
    \frac{\left( \partial_z E \right)(\lambda, z_i^1(0)) \cdot \partial_\lambda z_i^1(0)}{E(\lambda, z_i^1(0))}
    \left(
        \partial_{t_1} z_i^1(0) \cdot \partial_{t_2} u_i(0,0)
        - \partial_{t_2} z_i^1(0) \cdot \partial_{t_1} u_i(0,0)
    \right)
\right] d\lambda \\
=& \sum_{i=1}^d -\left(
    \partial_{t_1} z_i^1(0) \cdot \partial_{t_2} u_i(0,0)
    - \partial_{t_2} z_i^1(0) \cdot \partial_{t_1} u_i(0,0)
\right) \\
=& -\omega^{[d]} \left( \widetilde{D}_{t_1}, \widetilde{D}_{t_2} \right).
\end{align*}
Here, the third equality uses the asymptotic behavior
\[
- \frac{(\partial_z E)(\lambda, z_i^1(0)) \cdot \partial_\lambda z_i^1(0)}{E(\lambda, z_i^1(0))} = \frac{1}{\lambda}+\mathrm{O}(1),
\quad \text{as } \lambda \to 0,
\]
which follows from the local expression \( E(\lambda, z_i) = \prod_{j=1}^n (z_i - z_i^j(\lambda)) \); the last equality follows from the fact that a point in $\widetilde{D}(t_1, t_2)$ corresponds to coordinates $(u_i(t_1,t_2), z_i^1(u_i(t_1,t_2); t_1, t_2))$, and the canonical symplectic form on $T^*C$ over $\pi^{-1}(V_i)$ is $\omega_{T^\ast C} = dz_i \wedge d\lambda$.
\end{proof}

The Lagrangian property of \( \mathbb{L}_{{H}}(L, D)\) \text{and} \(\mathbb{L}_{{H}}(D)\) as stated in \Cref{thm: Lag Higgs} follows from \Cref{thm: Lag corres Higgs}.

\begin{proof}[Proof of \Cref{thm: Lag Higgs}]
    Let $\mathrm{pr}_1$ and $\mathrm{pr}_2$ denote the projection of $\operatorname{Hilb}^d(T^\ast C)\times \mathcal{M}_{H}(n, \ell)$ to $\operatorname{Hilb}^d(T^\ast C)$ and $\mathcal{M}_{H}(n, \ell)$ respectively.
    Recall that we have a natural map $\pi_\ast: \operatorname{Hilb}^d(T^\ast C) \to \operatorname{Hilb}^d(C)$  induced by the projection $\pi:T^\ast C\to C$. 
    It follows that 
    $$\mathbb{L}_H(L,D) = \mathrm{pr}_2 \left( \mathrm{pr}_1^{-1}\pi_\ast^{-1}(D) \cap \mathbb{L}_H(L,d) \right)$$ 
    is isotropic, since $\pi_\ast^{-1}(D)$ is isotropic in  $\operatorname{Hilb}^d(C)$. Combined with \Cref{c: non-empty and half dim Higgs}, this completes the proof.
\end{proof}

\section{Lagrangian subvarieties and Lagrangian correspondences for de Rham moduli spaces}\label{LDR}

In this section, we formulate and prove our main results for the de Rham moduli space. As introduced in \Cref{sect-triples}, the main objects of our consideration are triples of the form $(L \overset{i}{\hookrightarrow} E, \nabla)$, where $(E, \nabla)$ are pairs of holomorphic bundles with holomorphic connections and $L$ is a line subbundle of $E$. 
We start by introducing the sections $s_i(\nabla)$ of the line bundle $K_C^{n \choose 2} L^{-n} \det(E)$ induced by such triples and showing how they give rise to a filtration of $E$. We also make the critical observation that for generic divisor $D$ the forgetful assignment $(L \overset{i}{\hookrightarrow} E, \nabla) \mapsto (E, \nabla)$ with $D_i(\nabla)=D$ fixed is generically finite. 
We continue by identifying triples with a subset of opers with regular singularities. We use this identification to show that $\mathbb L_{dR}(D)$ are non-empty and of the right dimension. We finish the section with the proof of the Lagrangian correspondence for $M_{dR}$. 
\smallskip

\subsection{Holomorphic connections with line subbundles}
Given a triple \( (i: L \hookrightarrow E, \nabla) \), where \( (E,\nabla) \in  \mathcal{M}_{dR}(n)  \) or \( \mathcal{M}_{dR}(n, \mathcal{O}_C)  \), 
 we shall define (cf. \Cref{lemma-morphsim-well-defined}) a map 
\begin{subequations}\label{eqn-section-dR}
\begin{equation}
    s_i(\nabla): L^{n} \longrightarrow \det(E)K_C^{\binom{n}{2}},
\end{equation}
which serves as the counterpart to $s_i(\phi)$ defined in the Higgs side \eqref{eqn-def-s-i-phi}.

We begin by considering a flat frame of \( \nabla \), which yields a basis \( \{\mathbf{e_1}, \dots, \mathbf{e_n}\} \) of the local system \( E^\nabla := \ker \nabla \). Let \( U \subseteq C \) be a local chart with holomorphic coordinate \( z \), and let \( s \in \Gamma(U, L) \) be a nowhere vanishing local section. Via the inclusion \( i: L \hookrightarrow E \), we have \( i(s) \in \Gamma(U, E) \), which is expressed in terms of the local flat frame as
\(
    i(s) = \sum_{j=1}^n a_j(z) \mathbf{e_j},
\)
for \( a_j(z) \in \mathcal{O}_{U} \). For any \( f_i(z)\in \mathcal{O}_{U} \), we define the morphism \( s_i(\nabla): L^{n} \rightarrow \det(E)K_C^{ \binom{n}{2}}
\) locally by
\begin{align}\label{e:wronskian comp}
\begin{split}
s_i(\nabla)\left( f_1(z) s \otimes \cdots \otimes f_n(z) s \right)
:=\ 
&\det\left(
\begin{array}{cccc}
f_1 a_1 & \left(f_2 a_1\right)' & \cdots & \left(f_n a_1\right)^{(n-1)} \\
f_1 a_2 & \left(f_2 a_2\right)' & \cdots & \left(f_n a_2\right)^{(n-1)} \\
\vdots & \vdots & \ddots & \vdots \\
f_1 a_n & \left(f_2 a_n\right)' & \cdots & \left(f_n a_n\right)^{(n-1)}
\end{array}
\right) \otimes (\mathrm{d}z)^{\binom{n}{2}} \\
=\ 
&\left( \prod_{j=1}^n f_j(z) \right) \cdot \operatorname{Wr}(a_1(z), \dots, a_n(z)) \otimes (\mathrm{d}z)^{\binom{n}{2}},
\end{split}
\end{align}
\end{subequations}
where \( \operatorname{Wr}(a_1, \dots, a_n) \) denotes the Wronskian determinant.

\begin{lemma}\label{lemma-morphsim-well-defined}
\( s_i(\nabla) \) defines a global section of the line bundle \( L^{-n} \det(E)  K_C^{\binom{n}{2}} \).
\end{lemma}
\begin{proof}
The claim follows immediately from \eqref{e:wronskian comp}:
changes of the local flat frames $(\mathbf{e}_i)_{i=1}^n$,   
the local coordinate $z$ 
and the section $s \in \Gamma(U, E)$
induce a multiplicative factor in $s_i(\nabla)$,
compatible with the scaling of sections of 
\( L^{-n} \det(E)  K_C^{\binom{n}{2}} \).
\end{proof}

\begin{remark}
(a) For a $\lambda$-connection $(E, \nabla_\lambda)$ and a line subbundle $L \hookrightarrow E$, one can analogously define the section $s_i(\nabla_\lambda)$ of \( L^{-n} \det(E)  K_C^{\binom{n}{2}} \).
\\
(b) For $1 \leq k \leq n$, 
even though $\nabla^k\circ i$ does not define a bundle morphism 
as $\phi^k \circ i$ does,
their ``wedge product''
$$i \wedge (\nabla \circ i) \wedge \cdots \wedge (\nabla^{n-1} \circ i) $$
is well-defined and gives the morphism $s_i(\nabla)$.
\end{remark}

\begin{lemma}
If $(E,\nabla)$ is irreducible (i.e. admits no proper nonzero $\nabla$-invariant subbundle), then for any $i:L\hookrightarrow E$, the morphism $s_i(\nabla)\not\equiv 0$.
\end{lemma}

\begin{proof}
We argue by contrapositive. Suppose there exists $i:L\hookrightarrow E$ with $s_i(\nabla)\equiv 0$.
Consider the dual holomorphic connection $(E^\vee, \nabla^\vee)$.
It is clear that 
sections of the rank-$(n-1)$ subbundle $\ker(L) \subset E^\vee$ which are also flat w.r.t. $\nabla^\vee$ 
define a subbundle which is $\nabla^\vee$-invariant.
To this end, we will show that such $\nabla^\vee$-invariant subbundle exists. Note that locally $s_i(\nabla)\equiv 0$ implies $\operatorname{Wr}(a_1,\dots,a_n)\equiv 0$ 
(cf. \eqref{e:wronskian comp}).
This means that the functions 
$a_1(z), \dots, a_n(z)$ are $\mathbb{C}$-linearly dependent. If we denote $\mathrm{rank}_{\mathbb C} \{ a_1(z), \dots, a_n(z) \} = k$,
it follows that there are $\mathbb{C}$-linearly independent constant sections $c_j \in \Gamma(U, \ker(\nabla^\vee))$, 
$j = 1, \dots, c_{n-k}$ such that the pairing $\left( c_j, i(s) \right)$ vanishes identically on $U$.
These sections $c_j$ precisely generate the $\nabla^\vee$-invariant bundle we seek.
So $(E^\vee, \nabla^\vee)$ and hence $(E, \nabla)$ are reducible.
\end{proof}

In the rest of this paper, we will assume that $(E,\nabla)$ is irreducible,
i.e. we work in the smooth part of the de Rham moduli space.
For a triple 
\((i: L \hookrightarrow E, \nabla) \), denote by $D_i(\nabla)$ 
the divisor defined by the vanishing locus of $s_i(\nabla)$. 
Fixing $D$ implies that  
$
    L^{-1} = K_C^{-\frac{n-1}{2}} \otimes \mathcal{O}(D)^{1/n},
$
up to the choice of $n$--th root bundle.

\paragraph{Filtrations}
Given a triple $(i\colon L\hookrightarrow E,\nabla)$, the connection $\nabla$
endows $E$ with a left $\mathcal D_C$--module structure, where
$\mathcal D_C$ denotes the sheaf of differential operators on $C$.
Composing the action map with the inclusion $i$ yields a natural morphism
\(
\mathcal D_C\otimes_{\mathcal O_C} L \rightarrow E .
\)
For each $i\ge 1$, define the subsheaf
\begin{equation}\label{eq:def-Fi}
\mathcal F_i
\;:=\;
\operatorname{Im}\!\left(
\mathcal D_C^{\le i-1}\otimes_{\mathcal O_C} L \rightarrow E
\right),
\end{equation}
where $\mathcal D_C^{\le i-1}$ denotes differential operators of order
$\le i-1$ on $C$.
This yields an increasing filtration of subsheaves, which are compatible with $\nabla$,
\begin{equation}\label{eqn-sheaves-filtration}
0=\mathcal F_0\subset \mathcal F_1\subset \cdots \subset \mathcal F_n\subset E.
\end{equation}

We next relate the subsheaves $\mathcal F_i$ to jet bundles.
\begin{proposition}\label{p:oper bundle out D}
There is a canonical isomorphism 
\(
\mathcal F_i^\vee \;\cong\; J^{i-1}(L^{-1}).
\)
\end{proposition}

\begin{proof}
Locally on an open set $U$ with coordinate $z$ and a nowhere
vanishing section $s\in\Gamma(U,L)$, the $\mathcal D_C$--module structure
induced by $\nabla$ gives
\[
\mathcal F_i|_U
=
\operatorname{Span}_{\mathcal O_U}
\bigl\{
s,\ \nabla_{\partial_z}s,\ \ldots,\ (\nabla_{\partial_z})^{\,i-1}s
\bigr\}.
\]
In particular, $\mathcal F_i$ is locally free of rank $i$. This implies that the morphism $\mathcal \mathcal D_C^{\le i-1}\otimes_{\mathcal O_C} L \rightarrow E$ is sheaf theoretically injective.

On the other hand, for any $\mathcal O_C$--modules $F_1,F_2$, differential
operators of order $\le i-1$ satisfy the identification
\(
\operatorname{Diff}_C^{\le i-1}(F_1,F_2)
\;\cong\;
\operatorname{Hom}\!\bigl(J^{i-1}(F_1),F_2\bigr).
\)
Applying this with $F_1=L^{-1}$ and $F_2=\mathcal O_C$, taking the global section, we obtain canonical isomorphisms
\[
\mathcal F_i
= \mathcal D_C^{\le i-1}\otimes_{\mathcal O_C} L
=
\operatorname{Diff}_C^{\le i-1}(L^{-1},\mathcal O_C)
\;\cong\;
\operatorname{Hom}\bigl(J^{i-1}(L^{-1}),\mathcal O_C\bigr)
=
\bigl(J^{i-1}(L^{-1})\bigr)^\vee.
\]
This finishes the proof.
\end{proof}

By construction, the filtration \eqref{eqn-sheaves-filtration}
satisfies the Griffiths transversality condition
\[
\nabla \mathcal{F}_i \subset \mathcal{F}_{i+1} \otimes K_C.
\]
Furthermore, one has the induced $\mathcal{O}_{C}$--isomorphism,
\begin{equation}\label{eqn-induced-graded-morphism}
\overline \nabla: \mathcal{F}_i / \mathcal{F}_{i-1} \overset{\cong }\longrightarrow \mathcal{F}_{i+1} / \mathcal{F}_i \otimes K_C.   
\end{equation}

\medskip
In the rest of this paper, for simplicity, we will focus on the case where 
$D_i(\nabla)$ is reduced.
In this case, the following lemma states that  
except for $\mathcal{F}_n$ which is a lower Hecke transform of $E$,
the sheaves $\mathcal{F}_i$ with $i \leq n-1$ are subbundles of $E$.

\begin{proposition}\label{l:lower hecke}
Given a triple $(i: L \hookrightarrow E, \nabla)$,
with \(D = D_i(\nabla)\) a reduced divisor, each \(\mathcal{F}_i\) with \(i \leq n-1\) is a rank-\(i\) subbundle \(E_i \subset E\), 
and \eqref{eqn-sheaves-filtration} 
induces a filtration of subbundles for $E$, 
\begin{align}\label{e: filtration}
0 \subset i(L) = E_1 \subset E_2 \subset \dots \subset E_{n-1} \subset E. 
\end{align}
The connection $\nabla$ preserves the filtration \eqref{e: filtration} and thus is gauge equivalent to the form
\begin{equation}\label{eqn-nabla-filtration}
    \nabla = \mathrm{d} +
    \begin{pmatrix}
        &0  &\ast &\dots &\ast &\ast &\ast  \\
        &1 &0 &\dots &\ast &\ast &\ast \\
        &0 &1 &\dots &\ast &\ast &\ast \\
        &\vdots  &\vdots &\dots &\vdots  &\vdots &\vdots \\
        &0  &0 &\dots &0 &\ast &\ast \\
        &0  &0 &\dots &1 &\ast &\ast \\
        &0  &0 &\dots &0 &s_i(\nabla) &\ast 
    \end{pmatrix}.
\end{equation}
Furthermore, $\mathcal{F}_n$ is a lower Hecke transform of \(E\)  
at the subspaces $E_{n-1}\mid_p < E\mid_p$ for each \(p < D\), 
and fits into the s.e.s.
\begin{align}\label{eqn-Fn-E}
    0 \longrightarrow \mathcal{F}_n \longrightarrow E \longrightarrow L K_C^{-(n-1)}(D)\mid_D \longrightarrow  0,
\end{align}
where $L K_C^{-(n-1)}(D)\mid_D$ is the skyscraper sheaf.
\end{proposition}

\begin{proof}
For each \(p < D\), choose a local coordinate \(z\), such that \(p\) corresponds to \( z = 0\). 
For simplicity, consider the inclusion \(\mathcal{F}_i \hookrightarrow E\) as sheaves,
and let \(I(p)\) be the smallest index \(i\) such that 
\(\dim \mathcal{F}_i\mid_p < i\). 
We claim that \(I(p) = n\), for all $p < D$.

By \eqref{e:wronskian comp}, definition of $D_i(\nabla)$ 
and the fact that $D$ is reduced, the Wronskian 
\(\operatorname{Wr}(a_1, \dots, a_n)\) 
has a simple zero at \( z = 0\).
By way of contradiction, assume \(I(p)=i < n\). By \eqref{eq:def-Fi}, there exist constants \( c_k \) such that
\[
\boldsymbol{a}^{(i-1)}(0) = \sum_{k=0}^{i-2} c_k \boldsymbol{a}^{(k)}(0).
\]
Replacing the \(i\)-th column of the Wronskian matrix by
\[
\boldsymbol{a}^{(i-1)}(z) 
\;\longmapsto\; 
\boldsymbol{a}^{(i-1)}(z) - \sum_{k=0}^{i-2} c_k\, \boldsymbol{a}^{(k)}(z)
= z \, \mathbf{v}(z),
\]
for some column vector \(\mathbf{v}(z)\), we obtain
\[
\operatorname{Wr}(a_1, \dots, a_n)(z) 
= z^{\,n - i} \,
\det\!\left(
\boldsymbol{a}(z), 
\dots, 
\boldsymbol{a}^{(i-2)}(z), 
\boldsymbol{v}(z), 
\boldsymbol{v}'(z), 
\dots, 
\boldsymbol{v}^{(n - i - 1)}(z)
\right).
\]
This shows that 
\(\operatorname{Wr}(a_1, \dots, a_n)\) 
vanishes to order at least \(n - i > 1\) at \(z = 0\), 
contradicting the fact that the Wronskian has a simple zero there. 
Therefore, \(I(p) = n\), which implies that each $\mathcal{F}_i$ is a subbundle $E_i \subset E$, for $i<n$.

Since at $p < D$, the sheaf inclusion
$\mathcal{F}_n \hookrightarrow E$
evaluates to $E_{n-1}\mid_p$, 
we have $\mathcal{F}_n$ as the result of consecutive 
lower Hecke transforms of $E$ at $E_{n-1}\mid_p < E\mid_p$ 
for each $p < D$.
It follows from \eqref{eqn-induced-graded-morphism} that 
$$
\mathcal{F}_i / \mathcal{F}_{i-1} \cong L K_C^{-(i-1)}, 
\qquad \qquad i = 1, \dots, n.
$$
The following diagram
\[
\begin{tikzcd}[column sep=small, row sep=small]
0 \arrow[r] 
  & E_{n-1} \arrow[r] \arrow[d, equal] 
  & \mathcal{F}_{n} \arrow[r] \arrow[d] 
  & L K_C^{-(n-1)} \arrow[r] \arrow[d] 
  & 0 \\
0 \arrow[r] 
  & E_{n-1} \arrow[r] \arrow[d] 
  & E \arrow[r] \arrow[d] 
  & L K_C^{-(n-1)}(D) \arrow[r] \arrow[d] 
  & 0 \\ 
  & 0 \arrow[r] 
  & L K^{-(n-1)}(D)\!\big|_{D} \arrow[r, equal] 
  & L K^{-(n-1)}(D)\!\big|_{D}  \arrow[r] 
  & 0
\end{tikzcd}
\]
is commutative. 
This gives the s.e.s. \eqref{eqn-Fn-E}.
\end{proof}

\subsection{Opers with apparent singularities}

In this section, we introduce the subvarieties $\mathbb L_{dR}(L,D)$ (resp.\ $\mathbb L_{dR}(D)$) of the de Rham moduli spaces $\mathcal M_{dR}(n)$ (resp.\ $\mathcal{M}_{dR}(n, \mathcal{O}_C)$). Later, we will show (\Cref{thm: lag de rham}) that these are the de Rham counterparts of the Lagrangian subvarieties  $\mathbb L_H(L,D)$ (resp.\ $\mathbb L_H(D)$) on the Higgs side introduced in \Cref{def: Lag in Higgs}.

In the $\mathrm{SL}_n$ case, we identify the moduli of triple $\widetilde{\mathbb{L}_{dR}}(L,D)$ with the moduli of opers with prescribed coweight
\(
\mathrm{Op}_{\mathrm{PSL}_n,D}(C)_{\boldsymbol{\omega}_{n-1}^{\vee}},
\)
and show that such triples can be further identified with a differential operator with vanishing principle symbol at $D$. We do not treat the $\mathrm{GL}_n$ case since it is completely analogous.

\begin{definition}\label{def: Lag in dR}
\begin{itemize}
\item Given a triple $(L \overset{i}{\hookrightarrow} E, \nabla)$ such that $s_i(\nabla)$ does not identically vanish,
we say that the divisor $D_i(\nabla)$ is the divisor of \textit{apparent singularities} associated to the triple.
\item 
Let $G = \mathrm{GL}_n$.
Given a line bundle $L$ and an effective divisor $D$ of degree 
$\deg(D) = n(n-1)(g-1) - n \deg(L)$, 
let \(\mathbb{L}_{dR}(L, D)\) be the subvariety of \(\mathcal{M}_{dR}(n)\) given by 
\begin{align}\label{eqn-L_D}
    \mathbb{L}_{dR}(L, D) = 
\left\{ (E, \nabla) \in \mathcal{M}_{dR}(n) 
\mid (E, \nabla) \text{ irreducible}, \exists i: L \hookrightarrow E, \text{ } D_i(\nabla) = D\right\}.
\end{align}
\item Let $G = \mathrm{SL}_n$.
Given an effective divisor $D$ and a line bundle $L$ such that 
$L^n \simeq K_C^{n \choose 2}(-D)$,
let \(\mathbb{L}_{dR}(L, D)\) be the subvariety
of \(\mathcal{M}_{dR}(n, \mathcal{O}_C)\) given by \eqref{eqn-L_D}.
Let $\mathbb{L}_{dR}(D)$ be the union of all such subvarieties, i.e.
\begin{align}
\begin{split}
    \mathbb{L}_{dR}(D) &= \bigcup_{L^n \simeq K_C^{n \choose 2}(-D)} \mathbb{L}_{dR}(L, D) \\
    &= 
\left\{ (E, \nabla) \in \mathcal{M}_{dR}(n, \mathcal{O}_C) 
\mid (E, \nabla) \text{ irreducible}, \exists i: L \hookrightarrow E, D_i(\nabla) = D\right\}.
\end{split}
\end{align}

\end{itemize}
\end{definition}

In the following proposition, we identify the space
\(
\mathrm{Op}_{\mathrm{PSL}_n,D}(C)_{\boldsymbol{\omega}_{n-1}^{\vee}}
\)
with the moduli space of triples.
Here
\(
\boldsymbol{\omega}_{n-1}^{\vee}:=(\omega_{n-1}^\vee,\dots,\omega_{n-1}^\vee)
\)
where $\omega_i^{\vee}$ denotes the $i$-th fundamental coweight of $\mathfrak{sl}_n$.

\begin{proposition}\label{identify oper}
Fix a line bundle $L$ satisfying $L^{\otimes n}\simeq K_C^{\binom{n}{2}}(-D)$. Then there is a natural isomorphism
\begin{equation}\label{e: identify oper}
\widetilde{\mathbb{L}_{dR}}(L,D):=\bigl\{([i]\colon L\hookrightarrow E,\nabla)\mid D_i(\nabla)=D \bigr\} 
\;\cong\;
\mathrm{Op}_{\mathrm{PSL}_n,D}(C)_{\boldsymbol{\omega}_{n-1}^{\vee}}
.
\end{equation}
\end{proposition}

\begin{proof}
By \Cref{l:lower hecke}, each triple $([i]\colon L\hookrightarrow E,\nabla)$ with
$D_i(\nabla)=D$ determines a filtration $E^\bullet$ such that
$\mathbb P(E,\nabla,E^\bullet)$ is a $\mathrm{PSL}_n$-oper of type $\boldsymbol{\omega}_{n-1}^{\vee}$.
This defines a map from the right-hand side of \eqref{e: identify oper} to the left.

Conversely, let $(\mathcal F,\nabla,\mathcal F_B)\in
\mathrm{Op}_{\mathrm{PSL}_n,D}(C)_{\boldsymbol{\omega}_{n-1}^{\vee}}$.
Identifying $\mathrm{PSL}_n\simeq\mathrm{PGL}_n$, consider the central extension
\[
1\longrightarrow \mathbb G_m \longrightarrow \mathrm{GL}_n \longrightarrow \mathrm{PGL}_n \longrightarrow 1 .
\]
The obstruction to lifting the $\mathrm{PGL}_n$-bundle $\mathcal F$ (together with its $B$-reduction)
to a $\mathrm{GL}_n$-bundle with Borel reduction lies in
\(
H^2(C,\mathbb G_m)=\mathrm{Br}(C),
\)
the Brauer group of $C$, which vanishes for a smooth complex curve. Hence $(\mathcal F,\mathcal F_B)$
lifts to a $\mathrm{GL}_n$-bundle with full flag $(\widetilde{\mathcal F},\widetilde{\mathcal F_B})$, equivalently to a rank-$n$ vector bundle $E$ with a filtration
$E^\bullet$. 

The generic oper conditions imply canonical isomorphisms
\[
E_i/E_{i-1}\cong E_1\otimes K_C^{i-1}\quad (1\le i\le n-1), \quad E_n/E_{n-1}\cong E_1\otimes K_C^{n-1}(D),
\]
and therefore
\(
\det(E)\cong E_1^{\otimes n}\otimes K_C^{-\binom{n}{2}}(D).
\)
Fixing $E_1\simeq L$ and using the assumption on $L$, we obtain $\det(E)\cong\mathcal O_C$, so $E$
admits a reduction of structure group to $\mathrm{SL}_n$. This reduction is unique under the constraint $E_1=L$. Accordingly, we regard $(\widetilde{\mathcal F},\widetilde{\mathcal F_B})$ as the $\mathrm{SL}_n$ object.

It remains to lift the connection.  
Recall that a connection on $\mathcal F$ is equivalently a splitting of the Atiyah exact sequence
\[
0\longrightarrow \mathrm{ad}(\mathcal F)
\longrightarrow \mathcal A_{\mathcal F}
\longrightarrow T_C
\longrightarrow 0,
\]
where $\mathrm{ad}(\mathcal F)=\mathcal F\times_{\mathrm{PSL}_n}\mathfrak{sl}_n$ is the adjoint bundle and
$\mathcal A_{\mathcal F}=(T\mathcal F)/\mathrm{PSL}_n$ is the Atiyah algebroid.
Since $\mathrm{SL}_n\to\mathrm{PSL}_n$ is a finite central quotient, we have canonical identifications
\[
\mathrm{ad}(\mathcal F)\cong \mathrm{ad}(\widetilde{\mathcal F}),\qquad
\mathcal A_{\mathcal F}\cong \mathcal A_{\widetilde{\mathcal F}},
\]
where $\widetilde{\mathcal F}$ denotes the $\mathrm{SL}_n$-lift.
Thus the splitting $\nabla:T_C\to\mathcal A_{\mathcal F}\cong \mathcal A_{\widetilde{\mathcal F}}$ uniquely lifts to a connection on $E$ compatible
with the filtration.

We therefore obtain a unique $\mathrm{SL}_n$-oper $(E,\nabla,E^\bullet)$ with $E_1=L$, and
$(E,\nabla)\in\mathbb L_{dR}(L,D)$. The two constructions are inverse to each other, proving
\eqref{e: identify oper}.
\end{proof}

\begin{proposition}\label{p: generic finite dR}
Fix a line bundle \(L\) such that
\(
L^{\otimes n}\simeq K_C^{\otimes \binom{n}{2}}(-D).
\)
Consider the natural forgetful morphism
\[
\pi(L,D)\colon
\widetilde{\mathbb{L}_{dR}}(L,D)
\longrightarrow
\mathbb{L}_{dR}(L,D),
\qquad
([i]\colon L\hookrightarrow E,\nabla)\longmapsto (E,\nabla),
\]
then for generic \(D\in \operatorname{Hilb}^d(C)\), the morphism \(\pi(L,D)\) is generically finite.
\end{proposition}

\begin{proof}
Recall the construction of Hodge moduli space
\(
\mathcal{M}_{\mathrm{Hod}}
=
\{(\lambda,E,\nabla_\lambda)\}/\!\sim,
\)
where \(\nabla_\lambda\) is a \(\lambda\)-connection. Consider the  moduli space of $\lambda$-connections together with line subbundles,
\[
\widetilde{\mathcal{M}}_{\mathrm{Hod}}
=
\bigl\{
(\lambda,[i]\colon L\hookrightarrow E,E,\nabla_\lambda)
\bigr\},
\]
where the embedding \([i]\colon L\hookrightarrow E\) is taken up to scaling.

Now define
\[
X =\Bigl\{
(\lambda,[i],E,\nabla_\lambda,D)
\;\Big|\;
\operatorname{div}\bigl(s_i(\nabla_\lambda)\bigr)=D
\Bigr\}\subset \widetilde{\mathcal{M}}_{\mathrm{Hod}}\times \operatorname{Hilb}^d(C).
\]
By construction, \(X\) is an algebraic family. Let
\(
F\colon X\to \mathcal{M}_{\mathrm{Hod}}\times \operatorname{Hilb}^d(C)
\)
be the natural forgetful map
and let
\(
Y:=F(X)\subset \mathcal{M}_{\mathrm{Hod}}\times \operatorname{Hilb}^d(C)
\)
be its image.

For each \((\lambda,D)\in \mathbb{C}\times \operatorname{Hilb}^d(C)\), one has
\[
F_{(\lambda,D)}\colon
X_{(\lambda,D)}= \widetilde{\mathbb{L}_{\lambda}}(L,D) = 
\bigl\{
([i]\colon L\hookrightarrow E,\nabla_\lambda)
\mid
D_i(\nabla_\lambda)=D
\bigr \}
\longrightarrow
Y_{(\lambda,D)}=\mathbb{L}_{\lambda}(D).
\]
In particular, by \Cref{identify oper}
\[
X_{(1,D)}
=
\operatorname{Op}_{\mathrm{PSL}_n,D}(C)_{\boldsymbol{\omega}_{n-1}^{\vee}},
\qquad
Y_{(1,D)}
=
\mathbb{L}_{dR}(L,D),
\qquad
F_{(1,D)}=\pi(L,D).
\]
On the Higgs side, namely at \(\lambda=0\), the forgetful morphism
\(
F_{(0,D)}\colon \widetilde{\mathbb{L}_{0}}(L,D) \to \mathbb{L}_0(L,D)
\)
is finite and generically one-to-one for generic \(D\).

Let
\(
Z
:=
\bigl\{
y\in Y \mid \dim F^{-1}(y)\ge 1
\bigr\}.
\)
Then the closure of \(F^{-1}(Z)\) is a closed subvariety of positive codimension in \(Y\). Therefore, for generic $D \in \operatorname{Hilb}^d(C)$ and $\lambda$ sufficiently small, the morphism
\(
F_{(\lambda,D)}
\)
is generically finite.

Finally, for \(\lambda\neq 0\), the fibers of the Hodge family are naturally identified by rescaling the
\(\lambda\)-connection. Therefore the generic-finiteness of \(F_{(\lambda,D)}\) for one nonzero \(\lambda\)
is equivalent to that of \(F_{(1,D)}=\pi(L,D)\). It follows that for generic \(D\in \operatorname{Hilb}^d(C)\), the morphism $\pi(L,D)$
is generically finite, which proves the statement.
\end{proof}

\paragraph{Associated differential operators}
Given a triple \( (i: L \hookrightarrow E, \nabla) \) 
with $D_i(\nabla)=D$ and $(E,\nabla)$ irreducible, we 
define an associated differential operator as follows.

Consider the dual connection $(E^\vee, \nabla^\vee)$ with a local system \( {({E}^{\nabla})}^{\vee} := \ker(\nabla^\vee)\)
and a nonzero injective morphism \( i^\vee: {({E}^{\nabla})}^{\vee} \rightarrow L^{-1} \). 
The injectivity follows from the irreducibility.
Following the same construction as in \Cref{p:oper equ d.o.} (which was also considered in \cite{ohtsuki82, iwasaki92, yoshi97, We15}), 
one can define a global differential operator of degree $n$,
\begin{subequations}\label{def: ass diff oper 0}
\begin{equation}\label{def: ass diff oper}
    \mathbf{D}: L^{-1} \longrightarrow L^{-(n+1)} \otimes K_C^{\binom{n+1}{2}} \cong L^{-1} \otimes K_C^{n}(D),
\end{equation}
which is given by the Wronskian determinant over any open set \( (U,z) \subset C \), 
\begin{equation}\label{e:def diff oper D}
\mathbf{D} y = \det \begin{pmatrix}
i^\vee(\mathbf{v}_1) & \cdots & i^\vee(\mathbf{v}_n) & y \\
i^\vee(\mathbf{v}_1)' & \cdots & i^\vee(\mathbf{v}_n)' & y' \\
\vdots & & \vdots & \vdots \\
i^\vee(\mathbf{v}_1)^{(n)} & \cdots & i^\vee(\mathbf{v}_n)^{(n)} & y^{(n)}
\end{pmatrix} \otimes (\mathrm{d}z)^{\binom{n+1}{2}}.
\end{equation}
Here \( \left( \mathbf{v}_i\right) \) is a local frame for \( {({E}^{\nabla})}^{\vee} \) and \( y = y(z) \in \Gamma(U, L^{-1}) \). The derivatives are taken with respect to \( z \).
\end{subequations}

We say that two triples $(L \overset{i_1}{\hookrightarrow} E_1, \nabla_1)$ and $(L \overset{i_2}{\hookrightarrow} E_2, \nabla_2)$
are equivalent if there is an isomorphism of holomorphic connections 
$(E_1, \nabla_1) \overset{\cong}{\rightarrow} (E_2, \nabla_2)$ 
that commutes with $i_1$ and $i_2$.

\begin{proposition}\label{SL diff oper}
There is a one-to-one correspondence between $(L \overset{i}{\hookrightarrow} E, \nabla)$ and the differential operator
\(
\mathbf{D}\in \operatorname{Diff}_C^{n} \bigl(L^{-1},\,L^{-1}K_C^{n}(D)\bigr)^\circ.
\)
Here $\operatorname{Diff}^{n}_{C} \bigl(L^{-1},\,L^{-1}K_C^{n}(D)\bigr)^{\circ}$ denotes the space of order-$n$ operators
$L^{-1}\to L^{-1}K_C^{n}(D)$
whose principal symbol $\sigma_n(\mathbf{D})$ has divisor equal to $D$.
\end{proposition}
\begin{proof}
By \eqref{e:def diff oper D}, the principal symbol of $\mathbf{D}$ is locally given by
\begin{equation}\label{eqn-principal-symbol}
\sigma_n(\mathbf{D})
=\det \begin{pmatrix}
i^\vee(\mathbf{v}_1) & \cdots & i^\vee(\mathbf{v}_n)\\
i^\vee(\mathbf{v}_1)' & \cdots & i^\vee(\mathbf{v}_n)'\\
\vdots & & \vdots\\
i^\vee(\mathbf{v}_1)^{(n-1)} & \cdots & i^\vee(\mathbf{v}_n)^{(n-1)}
\end{pmatrix}\otimes (\mathrm{d}z)^{\binom{n}{2}}.
\end{equation}
Choosing local frames $(\mathbf{e}_i)$ of $E$ dual to $(\mathbf{v}_i)$ and using \eqref{e:wronskian comp} for $s_i(\nabla)$ gives $\sigma_n(\mathbf{D})=s_i(\nabla)$. 

To show the correspondence, it remains to recover the triple from $\mathbf{D}$. Let $\ker(\mathbf{D})$ be the sheaf of local solutions; it is a local system with a natural evaluation map
\[
\mathrm{ev}:\ \ker(\mathbf{D})\otimes_{\mathbb{C}}\mathcal{O}_C \longrightarrow L^{-1}.
\]
The condition $\operatorname{div}(\sigma_n(\mathbf{D}))=D$ implies $\mathrm{ev}$ is surjective: otherwise at some $p\in C$ the map would factor through $L^{-1}(-p)$, forcing an extra zero of $\sigma_n(\mathbf{D})$ at $p$, a contradiction. Set $E^\vee:=\ker(\mathbf{D})\otimes_{\mathbb{C}}\mathcal{O}_C$ with its tautological flat connection. Dualizing these objects recovers a triple $(L \overset{i}{\hookrightarrow} E, \nabla)$, and by construction it satisfies $D_i(\nabla)=D$.
\end{proof}

\begin{remark}\label{rem-D-PSLn}[$\mathrm{PSL}_n$ differential operator]
The scalar differential operator $\mathbf D$ associated with a triple $(L\xrightarrow{i}E,\nabla)$ is 
\(
\mathbf D: L^{-1}\longrightarrow L^{-1}\otimes K_C^{\otimes n}(D),
\)
and therefore depends on the choice of the embedding $i$ up to scaling (cf.\ \eqref{e:def diff oper D}). Equivalently, twisting $L$ by a flat line bundle $N$ with $N^{\otimes n}\simeq\mathcal O_C$ changes the domain of $\mathbf D$ but not its local expression.

From the $\mathrm{PSL}_n$ point of view, this dependence is inessential. Indeed, passing from $\mathrm{SL}_n$ to $\mathrm{PSL}_n$ amounts to forgetting the choice of the $n$-th root $N$ and retaining only the induced projective generic $\mathrm{PSL}_n$--oper structure. In this setting, $\mathbf D$ is naturally regarded as a collection of compatible local scalar differential operators on $C$, defined up to twisting of its domain by such flat line bundles.

Accordingly, $n^{2g}$ distinct $\mathrm{SL}_n$-triples $(L\xrightarrow{i}E,\nabla)$ give rise to the same $\mathrm{PSL}_n$-oper and the same projective differential operator $\mathbf D$, differing only by tensoring with flat line bundles $N$ satisfying $N^{\otimes n}\simeq\mathcal O_C$.
\end{remark}

To summarize, by \Cref{identify oper,SL diff oper}, after fixing a line bundle $L$, one has the isomorphism
\begin{equation}\label{e:ident spaces}
\widetilde{\mathbb{L}_{dR}}(L,D)
\cong
\mathrm{Op}_{\mathrm{PSL}_n,D}(C)_{\boldsymbol{\omega}_{n-1}^{\vee}}
\overset{\vee}{\cong}
\mathrm{Op}_{\mathrm{PSL}_n,D}(C)_{\boldsymbol{\omega}_{1}^{\vee}}
\cong
\mathbb C^{*} \backslash \operatorname{Diff}_C^{n}\!\bigl(L^{-1},\,L^{-1}\otimes K_C^{n}(D)\bigr)^\circ / \operatorname{Jac}(C)[n].
\end{equation}
Here, \(\mathbb C^{*}\) acts by scaling the differential operator, while \(\operatorname{Jac}(C)[n]\) acts by tensoring \(L\) with an \(n\)-torsion line bundle.

\subsection{Residue parameters}
In this section, we relate the divisor of apparent singularities $D_i(\nabla)$ associated to a triple $(L \overset{i}{\hookrightarrow} E, \nabla)$ with the notion of \emph{apparent singularities} for the differential operators in $\operatorname{Diff}_C^{n}\!\bigl(L^{-1},\,L^{-1}\otimes K_C^{n}(D)\bigr)^\circ$.
We then study the differential operators via the inclusion 
of \Cref{lem:global-apparent-mono}. 
Finally, we define the \emph{residue parameters} of the apparent singularities, which play a role analogous to spectral parameters on the Higgs side.

\paragraph{Meromorphic $\mathrm{PSL}_n$--oper}
The inclusion in \Cref{lem:global-apparent-mono},
\[
\mathrm{Op}_{\mathrm{PSL}_n,D}(C)_{\boldsymbol{\omega}_{1}^{\vee}}
\subset
\mathrm{Op}_{\mathrm{PSL}_n,D}^{\mathrm{RS}}(C)
\subset
\mathrm{Op}_{\mathrm{PSL}_n}(C\setminus D),
\]
induces the inclusion on the level of differential operators
\begin{equation}\label{e:inclusion in diff operator}
\mathbf D \in \operatorname{Diff}_C^{n} (L^{-1},\,L^{-1}\otimes K_C^{n}(D))^\circ
 \to
 \mathbf D_m \in \operatorname{Diff}^n_{C\setminus D} (K_{C\setminus D}^{-\frac{n-1}{2}}, K_{C\setminus D}^{\frac{n+1}{2}}).
\end{equation}
For later use, we rigidify the section $s_i(\nabla)=\sigma_n(\mathbf{D}) \in H^0(C,\mathcal{O}(D))$ in \eqref{eqn-nabla-filtration}, for each $p_i\in D$, fix a local chart $(U_i,z)$ with $z(p_i)=0$ such that, over $U_i$, the section $\sigma_n(\mathbf{D})=z$.
The following lemma describes the relation between $\mathbf D$ and $\mathbf D_m$.

\begin{lemma}\label{p:trans of diff opers}
View $\mathbf D_m$ (resp.\ $\mathbf D$) as a collection of compatible local scalar differential operators which patch to a global operator on $C\setminus D$ (resp.\ on $C$). Fix the above local charts. Then the following hold.
\begin{itemize}
    \item 
$\mathbf D_m$ and $\mathbf D$ agree on $C\setminus D$.
Over $U_i$, one has
\begin{equation}\label{e:transform rule of mero oper0}
\mathbf D(z)
=
z^{\frac{n+1}{n}}\circ \mathbf D_m(z) \circ z^{-\frac{1}{n}} .
\end{equation}

\item Locally, write
\begin{equation}\label{local Dm}
\mathbf D_m
:= \partial^n + Q_2(z)\partial^{n-2} + \cdots + Q_n(z).
\end{equation}
Then, globally, each coefficient $Q_i$ is in an affine space modeled on
\(
H^0\!\left(C,K_C^i(iD)\right).
\)
\end{itemize}
\end{lemma}

\begin{proof}
The first statement follows directly from the \Cref{lem:global-apparent-mono}.
Locally on $U_i$, this inclusion is implemented by the gauge transformation
$z^{\omega_1^\vee}$ applied to the local normal form \eqref{eq:global-apparent-oper}, bringing the connection into the standard oper form \eqref{eq:global-RS-oper} on $U_i^\times$.
Translating this gauge transformation into the language of scalar differential operators yields \eqref{e:transform rule of mero oper0}. The second statement is a consequence of \Cref{thm: singular oper space}.
\end{proof}
In the following, on each chart $U_i$ we denote
\begin{equation}\label{e:local form of D}
\mathbf D|_{U_i}
:= z\Bigl(\partial^n + a_1(z)\partial^{n-1}
+ a_2(z)\partial^{n-2} + \cdots + a_n(z)\Bigr),
\qquad a_j(z)\in \frac{1}{z}\mathcal O_{U_i}.
\end{equation}
Combining \eqref{e:transform rule of mero oper0} with \eqref{local Dm} yields
\begin{equation}\label{e:transform rule of mero oper}
\partial^n + a_1(z)\partial^{n-1}
+ a_2(z)\partial^{n-2} + \cdots + a_n(z)
=
z^{\frac{1}{n}}\circ
\Bigl(\partial^n + Q_2(z)\partial^{n-2} + \cdots + Q_n(z)\Bigr)
\circ z^{-\frac{1}{n}}.
\end{equation}
In particular, one obtains
\(
a_1(z)=-\frac{1}{z}.
\)

\begin{definition}
With the notation above, we call the divisor $D=D_i(\nabla)$ the \textbf{apparent singularity}
of the differential operator $\mathbf D \in \operatorname{Diff}_C^{n} (L^{-1},\,L^{-1}\otimes K_C^{n}(D))^\circ$.
\end{definition}

\begin{remark}
In the ODE picture, an \textbf{apparent singularity} is a point where the coefficients of the differential operator have a pole, yet all local solutions are holomorphic (equivalently, the local monodromy is trivial).
This is exactly the situation here: by \eqref{e:local form of D}, the points in $D$ contribute poles to the coefficients of $\mathbf D$, but do not produce nontrivial local monodromy.
\end{remark}

\paragraph{Simple apparent singularity and residue parameters.}
In this part, for each $\mathbf D \in \operatorname{Diff}_C^{n} (L^{-1},\,L^{-1}\otimes K_C^{n}(D))^\circ$, with the fixed $\sigma_n(\mathbf{D})$. we introduce the residue parameter locally, over each point in $D$ with its local chart. We also check that the apparent singularities together with their residue parameters define points in $\mathcal{S}$ which is a $T^*C$--torsor. 

In \cite{DM07}, the authors characterize \emph{simple} apparent singularities for
scalar Fuchsian equations. We recall the corresponding local statement on a disk \(\mathbb D\) with coordinate \(z\).

\begin{proposition}[\cite{DM07}, Lemma~4.8]\label{p:AS for fuchsian}
Consider a scalar Fuchsian equation on \(\mathbb D\) of the form
\[
y^{(n)}+a_1(z)y^{(n-1)}+\cdots+a_n(z)y=0.
\]
Then \(z=0\) is a simple apparent singularity if and only if the coefficients
\(a_k(z)\) admit local expansions
\begin{align}\label{e: single part of a_i}
  a_1(z) &= -\frac{1}{z}+\delta_1+\mathrm O(z), \\
  a_k(z) &= \frac{c_k}{z}+\delta_k+\mathrm O(z), \qquad k=2,\dots,n,
\end{align}
for some constants \(c_k,\delta_k\in \mathbb C\), and these coefficients satisfy
\begin{align}\label{e:recursion formula}
\begin{split}
  v\, c_n+\delta_n &= 0, \\
  v\, c_{n-1}+\delta_{n-1}+c_n &= 0, \\
  &\;\;\vdots \\
  v\, c_3+\delta_3+c_4 &= 0, \\
  v\, c_2+\delta_2+c_3 &= 0,
\end{split}
\end{align}
where
\(
v:=c_2+\delta_1.
\)
\end{proposition}

\Cref{p:AS for fuchsian} applies directly to our setting under the fixed local chart
\((U_i,z)\) as in \eqref{e:local form of D}.  Indeed, \(p_i\) is an apparent
singularity of the scalar equation associated with
\(
\frac{1}{z}\mathbf D|_{U_i}.
\)
Therefore, the singular parts of the coefficients \(a_k(z)\) in
\eqref{e:local form of D} satisfy the relations \eqref{e:recursion formula}.

This motivates the following definition.
\begin{definition}[Residue parameter]\label{def: res para}
Let
\(
\mathbf D \in \operatorname{Diff}_C^{n}\bigl(L^{-1},\,L^{-1}\otimes K_C^{n}(D)\bigr)^\circ,
\)
and fix a local chart \((U_i,z)\) around \(p_i < D\) such that
\eqref{e:local form of D} holds.  Suppose the coefficients of
\(\frac{1}{z}\mathbf D|_{U_i}\) have expansions
\[
a_1(z)=-\frac{1}{z},\qquad
a_k(z)=\frac{c_k^{(i)}}{z}+\delta_k^{(i)}+ \mathrm O(z).
\]
We define the \emph{residue parameter} of \(\mathbf D\) at \(p_i\) w.r.t.
the coordinate \(z\) to be
\(
v_i^{(z)}:=c_2^{(i)}.
\)
\end{definition}

The following proposition shows that each pair \((p_i, v_i)\), for \(p_i < D\), naturally defines a point in an
affine line bundle \(\mathcal{S}\) modeled on \(T^*C\).
\begin{proposition}\label{p: def twist cotanbun}
Under a change of local coordinate \(z \mapsto w = \phi(z)\) with \(\phi(0)=0\), 
the residue parameters \(v_i\) transform according to
\begin{equation}\label{e:transition for affine bundle}
   v_i^{w}
=
v_i^{z}\,\phi'(0)
-
\frac{n^2+2}{2n}\,
\frac{\phi''(0)}{\phi'(0)} .
\end{equation}
where \(v_i^{w}\) (resp.\ \(v_i^{z}\)) denotes the value of the residue parameter 
in the \(w\)– (resp.\ \(z\)–) trivialization.  
This transformation law defines an affine line bundle \(\mathcal{S}\) modeled on \(T^*C\), 
and each pair \((p_i, v_i)\) determines a point in \(\mathrm{Tot}(\mathcal{S})\).
\end{proposition}

\begin{proof}
For \(p_i < D\), let \(U_i\) be the local neighborhood of $p_i$ with coordinate \(z\), such that \(z(p_i)=0\). 
By \eqref{e:transform rule of mero oper}, in the $z$--coordinate, one obtains
\[
a_1(z) = -\frac{1}{z}, \qquad
a_2(z) = \frac{n^2-1}{2n}\frac{1}{z^2}+Q_2(z)
       = \frac{c_2^{(i),z}}{z}+\delta_2^{(i),z}+ \mathrm O(z).
\]
In this coordinate, we have $\delta^{(i),z}_1=0$, and \(v_i^z=c_2^{(i),z}\).

Now change the coordinates by \(z=\phi(w)\), with \(\phi(0)=0\) and \(\phi'(0)\neq 0\). Set
\(
h(w):=\phi(w)\phi'(w).
\)
Then \eqref{e:transform rule of mero oper} implies that, in the \(w\)–coordinate,
\begin{align}\label{diff op in w}
  \mathbf{D}|_{U_i} 
  = h(w)^{\frac{n+1}{n}} \circ \mathbf{D}_m|_{U_i}
  \circ h(w)^{-\frac{1}{n}},
\end{align}
where the local expression of $\mathbf{D}|_{U_i}$ and $\mathbf{D}_m|_{U_i}$ are,
\[
  \mathbf{D}|_{U_i} = h(w) \left( \partial_w^n + \widetilde Q_2(w)\,\partial_w^{n-2} + \cdots + \widetilde Q_n(w) \right), 
  \qquad
  \mathbf{D}_m|_{U_i} = \partial_w^n + \widetilde a_1(w)\,\partial_w^{n-1}+\cdots+\widetilde a_n(w).
\]
From \Cref{prop: k differential} and \eqref{e:affine of oper}, \(Q_2\) transforms 
as a meromorphic projective connection:
\[
\widetilde Q_2(w) = Q_2(\phi(w))\,\phi'(w)^2 + \frac{n(n^2-1)}{12}\,\{\phi(w),w\},
\]
where \(\{\phi(w),w\}\) is the Schwarzian derivative.  
Then \eqref{diff op in w} yields
\[
\widetilde a_1(w) = -\frac{h'(w)}{h(w)} 
       = -\frac{1}{w} - \frac{3\phi''(0)}{2\phi'(0)} + O(w),
\]
\[
\widetilde a_2(w) = \frac{n^2-1}{2n}\frac{h'(w)^2}{h(w)^2}
        - \frac{n-1}{2}\frac{h''(w)}{h(w)}+ \widetilde Q_2(w)
        = \left(c_2^{(i),z}\phi'(0)-\frac{(n-1)(n-2)}{2n}\frac{\phi''(0)}{\phi'(0)}\right)\frac{1}{w}+O(1).
\]
Thus in the \(w\)-coordinate the residue parameter satisfies \eqref{e:transition for affine bundle}
$$
v_i^w = c_2^{(i),w}+\delta_1^{(i),w}
      = v_i^z\,\phi'(0)-\frac{n^2+2}{2n}\frac{\phi''(0)}{\phi'(0)}.
$$
This shows that \(\{v_i\}\) transforms affinely under coordinate change.  
Consequently one obtains an affine line bundle \(\mathcal{S}\) over \(C\), 
with transition functions on overlaps \(U_i \cap U_j\) given by
\begin{equation}\label{e:transition for aff bun}
a_j 
   = a_i\,\frac{dz_i}{dz_j}
     - \frac{n^2+2}{2n}\,
       \frac{d^2 z_i / dz_j^2}{dz_i / dz_j},
\end{equation}
where \(a_i \in \Gamma(U_i, \mathcal{S})\) and \(a_j \in \Gamma(U_j, \mathcal{S})\).  
These transition rules satisfy the Čech cocycle condition for a torsor under \(T^*C\), 
thereby defining the affine bundle \(\mathcal{S}\).
\end{proof}

\begin{remark}\label{r:torsor class}
Starting from \eqref{e:transition for aff bun}, define local holomorphic $1$--forms
\(
A_i:=a_i\,\mathrm d z_i \in \Omega_C^1(U_i).
\)
Then \eqref{e:transition for aff bun} is equivalent to
\(
A_j = A_i + \alpha_{ij},
\)
where
\begin{equation}\label{eq:alphaij}
\alpha_{ij}:= - \frac{n^2+2}{2n}\, \mathrm d\!\left(\log\frac{\mathrm d z_i}{\mathrm d z_j}\right) \in \Omega_C^1(U_i\cap U_j).
\end{equation}
One can show that the collection $\{\alpha_{ij}\}$ satisfies the \v{C}ech cocycle condition, hence determines a class
\(
[\alpha]\in H^1(C,\Omega_C^1).
\)
Moreover, since ${\mathrm d z_i}/{\mathrm d z_j}$ are the transition functions of $K_C$, one has
\(
[\alpha]=-\frac{n^2+2}{2n}\,c_1(K_C)\in H^1(C,\Omega_C^1).
\)
In particular, this class characterizes the $T^\ast C$--torsor $\mathcal{S}$.
\end{remark}

\subsection{Lagrangian subvarieties in de Rham moduli spaces}
In this section, we show that the subvarieties $\mathbb L_{dR}(L,D)$ and $\mathbb L_{dR}(D)$ defined in \Cref{def: Lag in dR} are holomorphic Lagrangian subvarieties in $\mathcal M_{dR}$.

For simplicity, we restrict to the  single connected component $\mathbb L_{dR}(L,D)$, that is fixing the line bundle $L$ satisfies
\(
L^{\otimes n}\simeq K_C^{\binom{n}{2}}(-D).
\)
The argument proceeds in two steps. First, we compute the dimension of $\mathbb L_{dR}(L,D)$ by comparing it with
$\mathrm{Op}_{\mathrm{PSL}_n,D}(C)_{\boldsymbol{\omega}_{n-1}^{\vee}}$ via \Cref{identify oper,p: generic finite dR}. Second, we verify that $\mathbb L_{dR}(L,D)$ is isotropic with respect to the Atiyah--Bott symplectic form on $\mathcal{M}_{dR}(n, \mathcal{O}_C)$.

We first fix some notation. Let \( D=\sum_{i=1}^d p_i \) be a reduced effective divisor on $C$ with $n\mid d$.
In the rest of this section, for each \(p_i\in D\) we fix a local coordinate chart \((U_i,z)\) so that \Cref{p:trans of diff opers} applies.

For each $1\le i\le d$, $1 \le j \le n$, and $1\le k\le j$, let
\(
\xi_i^{(j,k)}\in H^0 \bigl(C,K_C^{j}(k p_i)\bigr)
\)
be a meromorphic $j$--differential whose Laurent part at $p_i\in D$ under local chart $U_i$ is
\[
\xi_i^{(j,k)}=\Bigl(\frac{1}{z^{k}}+\mathrm O(1)\Bigr)(\mathrm d z)^{j}.
\]
Such differentials $\xi_i^{(j,k)}$ exist by Riemann--Roch.

For later use, given a meromorphic $j$--differential $\omega$ and a point $p_m < D$, we denote by
\(
\mathrm{c}^i_{m}(\omega), i \geq 0
\)
the coefficient of $z^{i}$ in the local Laurent expansion of $\omega$ at $p_m$. More precisely, in the chart $U_m$,
\[
\omega=\Bigl(\cdots+\mathrm{c}^0_{m}(\omega)+\cdots+\mathrm{c}^i_{m}(\omega)z^i+ \mathrm O(z^{i+1})\Bigr)(\mathrm d z)^j.
\]

\medskip
The following proposition characterizes the meromorphic differential operators
corresponding to $\operatorname{Diff}_C^{n} (L^{-1},\,L^{-1}\otimes K_C^{n}(D))^\circ$ in \Cref{p:trans of diff opers}.

\begin{proposition}\label{p:AP-to-algebraic-v}
Adopt the notation above. Fix a holomorphic $\mathrm{PSL}_n$--oper
\(
\mathbf D_{\mathrm{reg}} \in \operatorname{Diff}_C^{n}\bigl(K_C^{-\frac{n-1}{2}},\,K_C^{\frac{n+1}{2}}\bigr)
\)
and meromorphic higher differentials $\{ \xi_i^{(j,k)} \}_{i,j,k}$.
The following hold,
\begin{enumerate}
\item For any
\(
\mathbf D\in \operatorname{Diff}_C^{n}\!\bigl(L^{-1},\,L^{-1}\otimes K_C^{n}(D)\bigr)^{\circ},
\)
let $\mathbf D_m$ denote its image under \eqref{e:inclusion in diff operator}.
Then there exist $\boldsymbol v =(v_1,\dots,v_d)\in \mathbb C^d$, $b\in B\bigoplus_{j=2}^n H^0(C,K_C^j)$ such that
\(
\mathbf D_m=\mathbf D_m(\boldsymbol v, b),
\)
and over fixed local charts $\{U_i\}_i$, these data solve an algebraic equation in $\boldsymbol v$ of the form 
\begin{equation}\label{nonlinear equ for v}
\boldsymbol v^n+\sum_{i=1}^{n-1}\boldsymbol v^{\,n-i-1} f_i(\boldsymbol v;b)=0,
\end{equation}
where each $f_i(\boldsymbol v;b)$ is a polynomial of degree $i$ in $\boldsymbol v$ whose coefficients depend on $b\in B$.

\item
(\textbf{Inverse construction})
Conversely, given $\boldsymbol v\in \mathbb C^d$ and $b\in B$ that are solutions to \eqref{nonlinear equ for v}, these data uniquely determine a meromorphic operator
\(
\mathbf D_m(\boldsymbol v,b) \in \operatorname{Diff}_{C\setminus D}^n \bigl(K_{C\setminus D}^{-\frac{n-1}{2}},\,K_{C\setminus D}^{\frac{n+1}{2}}\bigr)
\)
coming from an element 
\(
\mathbf D \in \operatorname{Diff}_C^{n} \bigl(L^{-1},\,L^{-1}\otimes K_C^{n}(D)\bigr)^{\circ} 
\)
via \eqref{e:inclusion in diff operator}.
\end{enumerate}
\end{proposition}

\begin{proof}
By \Cref{p:trans of diff opers}, any $\mathbf D\in \operatorname{Diff}_C^{n} \bigl(L^{-1},\,L^{-1}\otimes K_C^{n}(D)\bigr)^{\circ}$ injects to a meromorphic operator $\mathbf{D}_m$. 
In the local coordinate chart $(U_i,z)$ centered at $p_i\in D$, its local expression is given by \eqref{local Dm} and \eqref{e:local form of D}, namely
$$
\mathbf D_m|_{U_i}
        = \partial^n + Q_2(z)\partial^{n-2} + \cdots + Q_n(z).
$$
By \eqref{e:transform rule of mero oper} and \eqref{e: single part of a_i}, on $U_i$ the Laurent part of $a_k(z)$ is given by
\begin{align}\label{e: equ for a and Q}
a_k(z) \;=\; \sum_{j=0}^{k} c_{k,j}\,\frac{Q_j(z)}{z^{\,k-j}}= \frac{c_k^{(i)}}{z}+\delta_k^{(i)}+\mathrm O(z),
\quad \text{ with }Q_0(z)=1, Q_1(z)=\frac{1}{z}.
\end{align}
Equivalently, one can solve for $Q_k$ recursively as
\begin{align}\label{e: equ for a and Q;2}
    Q_k(z)=a_k(z)+ \sum_{j=0}^{k-1} d_{k,j} \frac{Q_j(z)}{z^{k-j}} ,
\quad k \ge 2.
\end{align}
Here the constants \(c_{i,j},d_{i,j}\) are universal binomial-type coefficients determined by the conjugation rule in \eqref{e:transform rule of mero oper0}.
For instance, on $U_i$ one has
\begin{align}
Q_2(z)
&=a_2(z)-\frac{n-1}{2n}\frac{1}{z^2}=-\frac{n^2-1}{2n}\frac{1}{z^2} + \frac{v_i}{z}+\delta_2^{(i)}+ \mathrm O(z), \label{e: eg expansion of Q2} \\
Q_3(z)
&=a_3(z)-\frac{n-2}{n}\frac{Q_2(z)}{z}
+\frac{(n-1)(n-2)}{6n^2}\frac{1}{z^3}. \label{e: eg expansion of Q3}
\end{align}
Moreover $v_i=c^{(i)}_2$, since $a_1(z)=-1/z$ by \eqref{e:transform rule of mero oper}.

Fix the reference oper $\mathbf D_{\mathrm{reg}}$, and write on the same chart $(U_i,z)$,
\[
  \mathbf{D}_{\mathrm{reg}}|_{U_i}
  := \partial^n + Q^{\mathrm{reg}}_2(z)\,\partial^{n-2} + \cdots + Q^{\mathrm{reg}}_n(z).
\]
By \Cref{p:trans of diff opers}, the difference $Q_i(z) - Q^{\mathrm{reg}}_i(z)$ glues into $w_i \in H^0 \left(C, K_C^i(iD)\right)$. In particular, specifying $\mathbf D_m$ is the same as specifying the global sections
$w_i$ for $i=2,\dots,n$ (as $\mathbf D_{\mathrm{reg}}$ fixed).

We give the following algorithm.

\textbf{Step 1.}  By \eqref{e: eg expansion of Q2}, the Laurent part of $w_2 \in H^0 \left(C, K_C^2(2D)\right)$ coincides with that of $Q_2(z)(\mathrm d z)^2$ in each chart $U_i$. Therefore, there exists a unique $b_2 \in H^0\left(C, K_C^2\right)$ satisfying
$$
w_2= \sum_{i=0}^d \left( -\frac{n^2-1}{2n}\xi_i^{(2,2)} + v_i \xi_i^{(2,1)} \right) + b_2.
$$
By \eqref{e: eg expansion of Q2}, constant terms in the expansion of $Q_2$ around $p_m\in D$ is given by
\begin{align}\label{e:identify del2}
\begin{split}
\delta^{(m)}_2 &= c^0_m(Q_2)=c^0_m(Q^{\mathrm{reg}}_2)+c^0_m(w_2)\\
&= c^0_m(Q^{\mathrm{reg}}_2)+\sum_{i=0}^d \left( -\frac{n^2-1}{2n}c^0_m \left(\xi_i^{(2,2)}\right) + v_i c^0_m \left( \xi_i^{(2,1)} \right) \right) + c^0_m(b_2)
\end{split}
\end{align}
Collecting these into vector
\(
\boldsymbol{\delta}_2=(\delta_2^1,\dots,\delta_2^d)^{T},
\)
\eqref{e:identify del2} becomes
\[
\boldsymbol{\delta}_2= f_1(\boldsymbol v;b_2) := \boldsymbol{a}_1+A_1 \boldsymbol{v},
\]
where \(\boldsymbol{a}_1\) and \(A_1\) depend only on \(c^0_m(Q^{\mathrm{reg}}_2),\{c^0_m(\xi_i^{(j,k)})\}_{j \leq 2}, c^0_m(b_2)\).

\medskip
\noindent
\textbf{Step 2.}  
The lowest relation in the apparent singularity condition \eqref{e:recursion formula} gives \(c_3^{(i)}=-v_i^2-\delta_2^{(i)}\). 
Together with \eqref{e: eg expansion of Q3}, this shows that the Laurent part of $w_3\in H^0\!\left(C,K_C^3(3D)\right)$ is determined by $\boldsymbol v$ and $b_2$ (since $Q_2(z)/z$ is already determined by $\boldsymbol v$ and $b_2$). Hence there exists a unique $b_3\in H^0(C,K_C^3)$ such that
$$
w_3= \sum_{i=0}^d \left( -\frac{(n-2)(n-1)(3n+4)}{6n^2}\xi_i^{(3,3)} - \frac{n-2}{n} v_i \xi_i^{(3,2)} - \frac{n-2}{n} \delta^{(i)}_2 \xi_i^{(3,1)} + c_3^{(i)}\xi_i^{(3,1)}\right) + b_3.
$$
Similarly to \eqref{e:identify del2}, constant terms of $a_3$ in \eqref{e: equ for a and Q} is given by
\[
\boldsymbol{\delta}_3 = f_2(\boldsymbol v;b_2,b_3) :=\boldsymbol{a}_2 + A_2\boldsymbol v + D_2 \left( B_2\boldsymbol v\circ C_2\,\boldsymbol v \right),
\]
where “\(\circ\)” denotes Hadamard (entrywise) product; the coefficients of $f_2(\boldsymbol v)$ are determined by \(c^0_m(Q^{\mathrm{reg}}_i)_{i \leq 3},\{c^l_m(\xi_i^{(j,k)})\}_{l+j \leq 3}, c^l_m(b_i)_{l+i\leq 3}\).

\medskip
\noindent
\textbf{Steps $3$ to $n-1$.}  
Proceeding inductively, suppose $b_2,\dots,b_k$ have been determined. Then the coefficients $c^{(i)}_k$ and $\delta^{(i)}_k$ can be expressed in terms of $\boldsymbol v$ and $b_2,\dots,b_k$.  The recursion \eqref{e:recursion formula} gives $c^{(i)}_{k+1} = -v_i c^{(i)}_{k}-\delta^{(i)}_{k}$ and together with \eqref{e: equ for a and Q;2} it follows that the Laurent part of $w_{k+1}$ is determined by $\boldsymbol v$ and $b_2,\dots,b_k$. Hence there exists a unique $b_{k+1} \in H^0\left(C, K_C^{k+1}\right)$, such that 
\begin{equation}\label{e:construct w_k}
w_{k+1}=\operatorname{Linear Comb}\{\xi_i^{(j,k+1)}\}_{j \leq k+1} + b_{k+1},
\end{equation}
where the linear combination is chosen to match the Laurent part of $w_{k+1}$.
Taking constant terms yields
\begin{equation}\label{e:equ for bolddelta_k}
\boldsymbol{\delta}_{k+1} = f_k(\boldsymbol v;b_i,i \leq k+1),
\end{equation}
where \(f_k(\boldsymbol v)\) is a (vector-valued) polynomial in \(\boldsymbol v\) of degree $k$, with coefficients determined by
\(c^0_m(Q^{\mathrm{reg}}_i)_{i \leq k+1},\{c^l_m(\xi_i^{(j,k)})\}_{l+j \leq k+1}, c^l_m(b_i)_{l+i\leq k+1}\).

\medskip
\noindent
\textbf{Step $n$.} 
From $\mathbf D$ we have extracted $\boldsymbol v\in \mathbb{C}^d$, $b\in B$, and $(w_2,\dots,w_n)$. These data satisfy one additional constraint, namely the top equation in \eqref{e:recursion formula}. Eliminating $c_k^{(i)}$ and using \eqref{e:equ for bolddelta_k}, we obtain an algebraic constraint for $\boldsymbol v$:
$$
(-1)^{n-1} \boldsymbol v^n + \left(\sum_{k=2}^{n} (-1)^{n-k}  \boldsymbol v^{n-k} \circ \boldsymbol{\delta}_{k}  \right) = (-1)^{n-1} \boldsymbol v^n + \left(\sum_{k=2}^{n} (-1)^{n-k}  \boldsymbol v^{n-k} \circ f_{k-1}(\boldsymbol v;b)  \right) = 0.
$$
which is exactly \eqref{nonlinear equ for v}. This proves the first statement.

\noindent\textbf{Conversely.}
Given $\boldsymbol v\in \mathbb{C}^d$ and $b\in B$, one constructs $w_k$ inductively via \eqref{e:construct w_k}. If moreover \eqref{nonlinear equ for v} holds, then the resulting meromorphic operator
\(
\mathbf D_m=\mathbf D_{\mathrm{reg}}+(w_2,\dots,w_n)
\)
satisfies the apparent singularity condition \eqref{e:recursion formula} on each chart $U_i$. Equivalently, $\mathbf D_m$ lies in the image of some
\(
\mathbf D\in \operatorname{Diff}_C^{n} \bigl(L^{-1},\,L^{-1}\otimes K_C^{n}(D)\bigr)^{\circ}
\)
under the inclusion \eqref{e:inclusion in diff operator}. This completes the proof.
\end{proof}

The following lemma concerns solutions to \eqref{nonlinear equ for v} for any $b\in B$.
\begin{lemma}\label{l:solve nonlinear equ}
    Adopt the notation above. Fix a holomorphic $\mathrm{PSL}_n$--oper
\(
\mathbf D_{\mathrm{reg}} \in \operatorname{Diff}_C^{n}\bigl(K_C^{-\frac{n-1}{2}},\,K_C^{\frac{n+1}{2}}\bigr)
\)
and meromorphic higher differentials $\{ \xi_i^{(j,k)} \}_{i,j,k}$. For any $b\in B$, the equation \eqref{nonlinear equ for v}
$$
\boldsymbol v^n+\sum_{i=1}^{n-1}\boldsymbol v^{\,n-i-1} f_i(\boldsymbol v;b)=0,
$$
admits a solution, and the number of solutions is finite.
\end{lemma}

\begin{proof}
By way of contradiction, we assume that \eqref{nonlinear equ for v} has no solution. 
Fix $r>0$, consider 
\(
\boldsymbol{v} = r\boldsymbol{\tilde v} = (r\tilde v_1,...,r\tilde v_d),
\)
with $\|\tilde v \| =1$.
Consider a one-parameter family of algebraic equations
\[
P_t(\boldsymbol{ v}) := \boldsymbol{v}^n + t\sum_{i=1}^{n-1} \boldsymbol{v}^{n-i-1} f_{i}(\boldsymbol{v};b). 
\]
For any $r>0$,  define the maps 
\begin{equation}\label{homotopy_1}
F_1(\boldsymbol{\tilde v};r) := \frac{P_1(r\boldsymbol{ \tilde v})}{|P_1(r\boldsymbol{\tilde v})|} : S^{2d-1} \to S^{2d-1}, \quad |\tilde v|=1.
\end{equation}
Note that $F_1(\cdot;r)$ are well defined continuous maps depending continuously on the parameter $r \in \mathbb R^+$. Taking $r\to 0$, we see that $F_1(\cdot;r)$ are homotopic to the constant map $F_1(\cdot;0)$ sending $S^{2d-1}$ to the point $\frac{P(0)}{|P(0)|}$.

On the other hand, since $f_i(\boldsymbol{v})$ are polynomial in $\boldsymbol{v}$ of degree $i$, it holds that for $r>0$ suitably large, 
\(
\|\boldsymbol{v}^n\| > \sum_{i=1}^{n-1} \|\boldsymbol{v}^{n-i-1} f_{i}(\boldsymbol{v})\|.
\)
As a result, the polynomials $P_t(\boldsymbol{ v})$ have no root on the sphere of radius $r$ for any $t\in[0,1]$. Therefore, the maps 
\begin{equation}\label{homotopy_2}
F_{t}(\boldsymbol{\tilde v};r) := \frac{P_t(r\boldsymbol{ \tilde v})}{|P_t(r\boldsymbol{ \tilde v})|} : S^{2d-1} \to S^{2d-1}
\end{equation}
are well defined continuous maps depending continuously on $t\in[0,1]$. Keeping $r$ fixed and letting $t\to 0$, these maps are therefore homotopic to the map 
\begin{equation}\label{diagonal}
F_{0}(\boldsymbol{\tilde v};r)=\frac{(\tilde v_1^n,...,\tilde v_d^n)}{\|(\tilde v_1^n,...,\tilde v_d^n)\|} : S^{2d-1} \to S^{2d-1}.
\end{equation}
It is easy to calculate that the degree of this map is equal to $n^d$. We have reached a contradiction since we have shown that the constant map $F_1(\cdot;0)$ which has degree zero has to be homotopic to the map $F_{0}(\cdot;r)$ in \eqref{diagonal}. 

The same degree reason implies that $P_1(\boldsymbol v)\neq 0$ for $\|\boldsymbol v\|$ sufficiently large, hence the zero set $P_1^{-1}(0)$ is bounded.
Since $P_1^{-1}(0)$ is an affine complex algebraic set, boundedness forces it to be finite: otherwise it would contain a positive-dimensional algebraic component,
which is necessarily unbounded in $\mathbb{C}^d$.
\end{proof}

\begin{proposition}\label{dim count for dR}
In $G=\mathrm{SL}_n$, let \( D=\sum_{i=1}^d p_i \) be a reduced effective divisor on $C$ with $n\mid d$. Then
\[
\dim_{\mathbb{C}} \mathbb L_{dR}(L,D) = \dim_{\mathbb{C}} \mathbb L_{dR}(D)
= \frac{1}{2}\dim_{\mathbb{C}}\,\mathcal{M}_{dR}(n, \mathcal{O}_C)
= (g-1)\,(n^2-1).
\]
\end{proposition}

\begin{proof}
By \Cref{p:AP-to-algebraic-v} and \Cref{l:solve nonlinear equ}, for each $b\in B$ the equation \eqref{nonlinear equ for v} admits at least one solution and only finitely many.
Thus the parameter $b$ labels a nonempty finite fiber in the construction of \(\operatorname{Diff}_C^{n}\!\bigl(L^{-1},\,L^{-1}\otimes K_C^{n}(D)\bigr)^{\circ}\), and we obtain
\[
\dim_{\mathbb{C}} \operatorname{Diff}_C^{n}\!\bigl(L^{-1},\,L^{-1}\otimes K_C^{n}(D)\bigr)^{\circ}
=\dim_{\mathbb{C}} B.
\]
On the other hand, $\dim_{\mathbb{C}} B=\frac12\dim_{\mathbb{C}}\mathcal{M}_{dR}(n, \mathcal{O}_C)=(g-1)(n^2-1)$.
Finally, the dimension statement for $\mathbb L_{dR}(L,D)$ and $\mathbb L_{dR}(D)$ follows from the generic finiteness of the forgetful map in \Cref{p: generic finite dR}. 
\end{proof}

In the following, we prove the Lagrangian property of $\mathbb{L}_{dR}(D)$.
\begin{theorem}\label{thm: lag de rham}
The subvariety $\mathbb{L}_{dR}(D)\subset \mathcal{M}_{dR}(n, \mathcal{O}_C)$ is holomorphic Lagrangian with respect to the Atiyah--Bott holomorphic symplectic form. The same statement holds for $\mathrm{GL}_n$.
\end{theorem}

\begin{proof}
By \Cref{dim count for dR}, it suffices to show that $\mathbb{L}_{dR}(D)$ is isotropic for the Atiyah--Bott form.

Fix a point $(E,\nabla) \in \mathbb{L}_{dR}(D)$. Choose a gauge adapted to the filtration given by \Cref{l:lower hecke}. In this gauge, the holomorphic structure on $E$ is represented by the Dolbeault operator,
\[
\bar\partial_E=
\begin{pmatrix}
\bar\partial_{L_1} & & \\
& \ddots & \\
& & \bar\partial_{L_n}
\end{pmatrix}
+ B,
\quad
B \in \Omega^{0,1}\!\left(C,\operatorname{End}(E)\right)\ \text{strictly upper triangular},
\]
where $L_i=\mathcal F_i/\mathcal F_{i-1}$ for $i<n$ and $ L_n=E/\mathcal F_{n-1}$. In particular, the holomorphic structures on the graded pieces $L_i$ are fixed, and an infinitesimal deformation of $\bar\partial_E$ in $T_{(E,\nabla)}\mathbb{L}_{dR}(D)$ is represented by
\(
\mathfrak b\in \Omega^{0,1}\!\left(C,\operatorname{End}(E)\right),
\)
which is strictly upper triangular.

Similarly, \Cref{l:lower hecke} provides a normal form for the holomorphic connection:
\[
\nabla=
\begin{pmatrix}
\partial_{L_1} & & & \\
& \ddots & & \\
& & \ddots & \\
& & & \partial_{ L_n}
\end{pmatrix}
+
\begin{pmatrix}
0 & & & \\
1 & 0 & & \\
\vdots & \ddots & \ddots & \\
0 & \cdots & s & 0
\end{pmatrix}
+ N,
\quad
N\ \text{weakly upper triangular}.
\]
Here $s$ denotes the unique section in $H^0\!\left(K_C^{\binom{n}{2}}L^{-n}\right)$ vanishing precisely along $D$, which is always fixed. Therefore an infinitesimal deformation of $\nabla$ tangent to $\mathbb{L}_{dR}(D)$ is represented by
\(
\mathfrak n\in \Omega^{1,0}\!\left(C,\operatorname{End}(E)\right),
\)
which is weakly upper triangular.

Therefore, take two tangent vectors $\delta_1(E,\nabla), \delta_2(E,\nabla) \in T_{(E,\nabla)}\mathbb{L}_{dR}(D)$, represented in the chosen gauge by pairs
\(
(\mathfrak b_1,\mathfrak n_1)
\)
and
\(
(\mathfrak b_2,\mathfrak n_2)
\).
Then
\begin{align*}
\Omega_{\mathrm{AB}}(\delta_1,\delta_2)
=\int_C \operatorname{tr}\!\left((\mathfrak b_1+\mathfrak n_1)\wedge(\mathfrak b_2+\mathfrak n_2)\right)
=\int_C \operatorname{tr}(\mathfrak b_1\wedge \mathfrak n_2)+\operatorname{tr}(\mathfrak b_2\wedge \mathfrak n_1)
=0.
\end{align*}
Thus $\mathbb{L}_{dR}(D)$ is isotropic, and hence (being half-dimensional) it is a holomorphic Lagrangian subvariety.
\end{proof}

\subsection{Lagrangian correspondence for de Rham moduli spaces}
In this section, for \(\mathrm{GL}_n\) and \(\mathrm{SL}_n\) cases, we define the subvarieties
\(\mathbb{L}_{dR}(L,d)\) and \(\mathbb{L}_{dR}(d)\) in
\(
\mathrm{Hilb}^d(\mathcal{S}) \times \mathcal{M}_{dR}(n)
\)
and
\(
\mathrm{Hilb}^d(\mathcal{S}) \times \mathcal{M}_{dR}(n, \mathcal{O}_C),
\)
respectively. Here $\mathcal{S}$ is the $T^\ast C$--torsor introduced in \Cref{p: def twist cotanbun}.
These subvarieties provide a Lagrangian correspondence between the two factors.

Recall that given a triple $(i\colon L\hookrightarrow E,\nabla)$ with reduced apparent singular divisor
$D_i(\nabla)=\sum p_i$, one associates residue parameters $v_i$ at each $p_i$ via the operator
\(
\mathbf D \in \operatorname{Diff}_C^{n} (L^{-1},\,L^{-1}\otimes K_C^{n}(D))^\circ
\)
and \Cref{def: res para}. By \Cref{p: def twist cotanbun}, each pair $(p_i,v_i)$ defines a point of $\mathcal{S}$, hence determines a reduced length-$d$ subscheme
\[
\widetilde{D}_i(\nabla):=\{(p_i,v_i)\}_{i=1}^d \in \operatorname{Hilb}^d(\mathcal{S}),
\qquad d=\deg D.
\]
Moreover, its support projects to $D\in \operatorname{Sym}^d(C)\setminus \Delta$ under $\pi:\mathcal{S}\to C$, so
\(
\widetilde{D}(i\colon L\hookrightarrow E,\nabla)\in \pi^\ast\bigl(\operatorname{Sym}^d(C)\setminus \Delta\bigr)\subset \operatorname{Hilb}^d(\mathcal{S}).
\)
We now define the desired correspondences.

\begin{definition}\label{def:lag-on-extended-space-dR-alt}
Let $d \in \mathbb{Z}$ and let $L$ be a line bundle on $C$ satisfying
\[
d = n(n-1)(g-1) - n\,\deg(L).
\]
We define $\mathbb{L}_{dR}(L,d)$ to be the subvariety of
\[
\pi^\ast\bigl(\operatorname{Sym}^d(C)\setminus \Delta\bigr)\times \mathcal{M}_{dR}(n)
\ \subset\
\operatorname{Hilb}^d(\mathcal{S})\times \mathcal{M}_{dR}(n)
\]
consisting of triples $(\widetilde{D},E,\nabla)$ such that
\begin{align}
\mathbb{L}_{dR}(L,d)
:=
\left\{
    (\widetilde{D}, E, \nabla)\ \middle|\ 
    \begin{array}{l}
        \exists\, i: L \hookrightarrow E \text{ such that } \\[2pt]
        \widetilde{D}=\widetilde{D}_i(\nabla)
    \end{array}
\right\}.
\end{align}

Assume $n\mid d$. We define $\mathbb{L}_{dR}(d)$ to be the subvariety of
\[
\pi^\ast\bigl(\operatorname{Sym}^d(C)\setminus \Delta\bigr)\times \mathcal{M}_{dR}(n, \mathcal{O}_C)
\ \subset\
\operatorname{Hilb}^d(\mathcal{S})\times \mathcal{M}_{dR}(n, \mathcal{O}_C)
\]
consisting of triples $(\widetilde{D},E,\nabla)$ such that
\begin{align}
\mathbb{L}_{dR}(d)
:=
\left\{
    (\widetilde{D}, E, \nabla)\ \middle|\ 
    \begin{array}{l}
        \exists\, i:L \hookrightarrow E \text{ such that } \\[2pt]
        L^{\otimes n} \cong K_C^{\binom{n}{2}}\otimes \mathcal{O}_C\!\bigl(-\pi_\ast(\widetilde{D})\bigr), \\[2pt]
        \widetilde{D}=\widetilde{D}_i(\nabla)
    \end{array}
\right\}.
\end{align}
Here $\pi_\ast:\operatorname{Hilb}^d(\mathcal{S})\to \operatorname{Sym}^d(C)$ is induced by the projection $\pi:\mathcal{S}\to C$.
\medskip
By construction, if $(\widetilde{D},E,\nabla)\in \mathbb{L}_{dR}(L,d)$ (resp.\ $\mathbb{L}_{dR}(d)$), then
$(E,\nabla)\in \mathbb{L}_{dR}(L,\pi_\ast(\widetilde{D}))$ (resp.\ $\mathbb{L}_{dR}(\pi_\ast(\widetilde{D}))$).
\end{definition}

\begin{theorem}\label{thm: Lag corres dR; SLn}
The variety $\mathbb{L}_{dR}(d)$ defines a \emph{Lagrangian correspondence} between
\(
\mathcal{M}_{dR}(n, \mathcal{O}_C)
\)
and
\(
\operatorname{Hilb}^d(\mathcal{S}).
\)
In other words, $\mathbb{L}_{dR}(d)$ is a Lagrangian subvariety of
\(
\operatorname{Hilb}^d(\mathcal{S})\times \mathcal{M}_{dR}(n, \mathcal{O}_C),
\)
equipped with the holomorphic symplectic form
\(
\omega_{\mathcal{S}}^{[d]}\oplus \Omega_{\mathrm{AB}}.
\)
The same statement holds for $\mathrm{GL}_n$.
\end{theorem}

Before proving the theorem, we need the following lemma.
\begin{lemma}\label{l: standard form near div dr}
Let $(\widetilde D,E,\nabla)\in\mathbb{L}_{dR}(d)$. By definition, there exists 
$i: L \hookrightarrow E$ such that 
\[
\widetilde D=\widetilde{D}_i(\nabla)=\{(p_i,v_i)\}_{i=1}^d,
\qquad d=\deg D.
\]
For each $p_i$, choose a local chart $(U_i,z)$ with $z(p_i)=0$.
Then, over $U_i$, with respect to a local frame adapted to the filtration in \eqref{eqn-nabla-filtration}, the dual connection $\nabla^{\vee}|_{U_i}$ is gauge equivalent to
\begin{equation}\label{e: gauge form for nabla Ui}
    \nabla^{\vee} \mid_{U_i} = \mathrm{d} +
    \begin{pmatrix}
        &v^{(z)}_i  &\ast &\ast &\dots &\ast &\ast &\ast  \\
        &-z &-v^{(z)}_i &\ast &\dots &\ast &\ast &\ast \\
        &0 &-1  &0 &\dots &\ast &\ast &\ast \\
        &\vdots &\vdots &\vdots &\dots &\vdots  &\vdots &\vdots \\
        &0  &0 &0 &\dots &-1 &0 &\ast \\
        &0  &0 &0 &\dots &0 &-1 & 0 
    \end{pmatrix} \mathrm{d}z,
\end{equation}
where $v^{(z)}_i$ is the residue parameter under the trivialization over $U_i$ in \Cref{p: def twist cotanbun}, and all unspecified entries are holomorphic in $z$.
\end{lemma}

\begin{proof}
By \eqref{eqn-nabla-filtration}, over $U_i$, the connection $\nabla^{\vee}\big|_{U_i}$ is gauge equivalent to
\[
\nabla^{\vee}\big|_{U_i}
=\mathrm{d}+
\begin{pmatrix}
  -f_i(z) & * & * &\cdots & * & * & * \\
  -z & f_i(z) & * &\cdots & * & * & * \\
  0 & -1 & 0 &\cdots & * & * & * \\
  \vdots & \vdots &\vdots & \ddots & \vdots & \vdots & \vdots \\
  0 & 0 & 0 &\cdots & -1 & 0 & * \\
  0 & 0 & 0 & \cdots & 0 & -1 & 0
\end{pmatrix}\mathrm{d}z,
\]
where all entries are holomorphic in $z$, since $\nabla^{\vee}$ is holomorphic.
Thus it suffices to show that
\[
f_i(z)= -v_i^{(z)}+\mathrm{O}(z),
\]
since, once this holds, a further gauge transformation by a matrix valued in the unipotent Lie subgroup $N$ eliminates the $\mathrm{O}(z)$ terms in the first two diagonal entries, yielding the form \eqref{e: gauge form for nabla Ui}.

Recall that solving for a flat section
\begin{equation}\label{e:conn flat equ}
\nabla^{\vee}\big|_{U_i}\,\boldsymbol{s}(z)=0,
\qquad 
\boldsymbol{s}(z)=
\begin{pmatrix}
s_1(z) & s_2(z) & \cdots & s_n(z)
\end{pmatrix}^{\mathsf T},
\end{equation}
is equivalent to solving the scalar $n$th-order equation
\(
\mathbf D\big|_{U_i}s_n(z)=0
\),
where via \eqref{e:local form of D}
\begin{equation}\label{e: nth order pde}
\mathbf D\big|_{U_i}
:=z\Bigl(
\partial^n+a_1(z)\partial^{n-1}
+a_2(z)\partial^{n-2}+\cdots+a_n(z)
\Bigr),
\qquad a_1(z)=-\frac{1}{z}.
\end{equation}
By the definition of the residue parameter \eqref{def: res para}, we have
\[
a_2(z)=\frac{v_i^{(z)}}{z}+\delta_2^{(i)}+\mathrm{O}(z).
\]

To derive the scalar equation from \eqref{e:conn flat equ}, we eliminate variables recursively:
\[
s_k(z)
=\Bigl(\partial^{n-k}
+\mathrm{LinComb}_{\mathcal O_{U_i}}\{\partial^{\le n-k-2}\}\Bigr)s_n(z),
\quad k\ge2,
\]
and
\[
s_1(z)
=\frac{1}{z}\Bigl(
\partial^{n-1}
+f_i(z)\partial^{n-2}
+\mathrm{LinComb}_{\mathcal O_{U_i}}\{\partial^{\le n-3}\}
\Bigr)s_n(z).
\]
Substituting these expressions into the first row of
\eqref{e:conn flat equ}, we obtain
\begin{align}\label{e:scalar-from-first-row}
\Bigl(
\partial^{n}
-\frac{1}{z}\partial^{n-1}
+\Bigl(
-\frac{f_i(z)}{z}
+\mathrm{O}(z)
\Bigr)\partial^{n-2}
+\frac{1}{z}\,
\mathrm{LinComb}_{\mathcal O_{U_i}}\{\partial^{\le n-3}\}
\Bigr)s_n(z)=0.
\end{align}
Comparing the coefficient of $\partial^{n-2}$ in
\eqref{e:scalar-from-first-row} with \eqref{e: nth order pde}, we conclude that
\[
a_2(z)
=-\frac{f_i(z)}{z}+\mathrm{O}(z)
=\frac{v_i^{(z)}}{z}+\delta_2^{(i)}+\mathrm{O}(z),
\]
which implies
\(
f_i(z)=-v_i^{(z)}+\mathrm{O}(z)
\),
as claimed.
\end{proof}

\begin{proof}[Proof of \Cref{thm: Lag corres dR; SLn}]
By \Cref{p: generic finite dR}, the space $\mathbb{L}_{dR}(d)$ is a half-dimensional subvariety of
\(
\operatorname{Hilb}^d(\mathcal A)\times \mathcal{M}_{dR}(n, \mathcal{O}_C).
\)
It therefore suffices to show that $\mathbb{L}_{dR}(d)$ is isotropic.

Let
\[
(\widetilde D(t_1,t_2),E(t_1,t_2),\nabla(t_1,t_2))
\in \mathbb{L}_{dR}(d)
\]
be a two-parameter deformation family, with $(t_1,t_2)$ sufficiently small.
By \Cref{p: generic finite dR}, this family is in one-to-one correspondence with a two-parameter deformation of triples
\begin{equation}\label{e: two para class dR}
(i(t_1,t_2)\colon L(t_1,t_2)\hookrightarrow E(t_1,t_2),\nabla(t_1,t_2)).
\end{equation}
We first study the deformation of the bundles and connections $(E(t_1,t_2),\nabla(t_1,t_2))$.
Since there is a finite cover from $(E(t_1,t_2),\nabla(t_1,t_2))$ to the associated
$\mathrm{PSL}_n$–flat connection
\(
(\mathbb P(E(t_1,t_2)),\nabla(t_1,t_2)),
\)
it is sufficient to work on the $\mathrm{PSL}_n$ level.

By \Cref{p:trans of diff opers}, once the theta characteristic is fixed, each triple determines a meromorphic operator
\(
\mathbf D_m(t_1,t_2)\in \operatorname{Diff}^n_{C\setminus D} (K_{C\setminus D}^{-\frac{n-1}{2}}, K_{C\setminus D}^{\frac{n+1}{2}}).
\)
Combining this with \Cref{p:oper bundle out D} and the fact that $L(t_1,t_2)^{-1} \cong K_C^{-\frac{n-1}{2}}\otimes \mathcal{O}_C(\frac{D}{n})$, the projective bundle
$\mathbb P(E(t_1,t_2))$ can be obtained as a modification of $\mathbb P(J^{n-1}(K_C^{-\frac{n-1}{2}})^\vee)$ at the divisor
$D(t_1,t_2)=\pi_*(\widetilde D(t_1,t_2))$.

Away from the divisor $D(t_1,t_2)$, we gauge
$\nabla(t_1,t_2)$ (and hence its dual)
into the standard oper form.
More precisely, if $(V,z) \subset C\setminus D(t_1,t_2)$ is a local chart, then
\[
\mathbf D_m(t_1,t_2)\big|_V
=\partial^n+Q_2(z;t_1,t_2)\partial^{n-2}+\cdots+Q_n(z;t_1,t_2),
\]
and the dual connection takes the form
\begin{equation}\label{e: mero oper form dr lag}
\nabla(t_1,t_2)^{\vee}\big|_V
=\mathrm{d}+
\begin{pmatrix}
 0 & Q_2 & \cdots & Q_{n-1} & Q_{n} \\
 -1 & 0 & \cdots & 0 & 0\\
 0 & -1 & \ddots & \vdots & \vdots\\
 \vdots & \vdots & \ddots & 0 & 0\\
 0 & 0 & \cdots & -1 & 0
\end{pmatrix}\mathrm{d}z.
\end{equation}

Now write
\(
D(t_1,t_2)=\sum_i p_i(t_1,t_2)
\).
For each $p_i(0,0)\in C$, fix a neighborhood $(U_i,z)$ with $z(p_i(t_1,t_2))=u_i(t_1,t_2)$, $u_i(0,0)=0$, $\mathbb P(E(t_1,t_2)^{\vee})$ is obtained from $\mathbb P(J^{n-1}(K_C^{-\frac{n-1}{2}}))$ by a transition
function
\[
G_i(z;t_1,t_2)\in
\Gamma(U_i\setminus p_i(t_1,t_2),\mathrm{PSL}_n\otimes\mathcal O),\qquad 1 \leq i \leq d,
\]
satisfying
\begin{equation}\label{e:AS condition ess}
G_i(z;t_1,t_2)\cdot
\nabla(t_1,t_2)^{\vee}\big|_{U_i}
=
\nabla(t_1,t_2)^{\vee}\big|_{U_i\setminus p_i(t_1,t_2)},
\end{equation}
where the RHS is the oper form \eqref{e: mero oper form dr lag}, while LHS is in the form given by \Cref{l: standard form near div dr},
\begin{equation}\label{e:form of nabla check Ui}
\nabla^{\vee}(t_1,t_2)\big|_{U_i}
= \mathrm{d} +
\scalebox{0.85}{$
\begin{pmatrix}
        &v^{(z)}_i(t_1,t_2)  &\ast &\ast &\dots &\ast &\ast &\ast  \\
        &-(z-u_i(t_1,t_2)) &-v^{(z)}_i(t_1,t_2) &\ast &\dots &\ast &\ast &\ast \\
        &0 &-1  &0 &\dots &\ast &\ast &\ast \\
        &\vdots &\vdots &\vdots &\dots &\vdots  &\vdots &\vdots \\
        &0  &0 &0 &\dots &-1 &0 &\ast \\
        &0  &0 &0 &\dots &0 &-1 & 0 
    \end{pmatrix}
$}\,\mathrm{d}z .
\end{equation}
For simplicity, we write $\nabla^{\vee}(t_1,t_2)\big|_{U_i} = \mathrm{d} + A_i(z;t_1,t_2)$ for this gauge form.

Moreover, note that $G_i$ admits a factorization
\[
G_i(z;t_1,t_2)
=(z-u_i(t_1,t_2))^{\omega_1^{\vee}}\cdot N_i(z;t_1,t_2),
\]
with $N_i$ valued in the unipotent lie subgroup $N$.
Here
\[
(z-u_i(t_1,t_2))^{\omega_1^{\vee}}
=
\mathrm{diag}\bigl(
(z-u_i)^{(n-1)/n},
(z-u_i)^{-1/n},\dots,(z-u_i)^{-1/n}
\bigr).
\]
From \eqref{e:AS condition ess},
one finds that $N_i(z;t_1,t_2)$ must be of the form
\begin{equation}\label{e:form of Ni}
N_i(z;t_1,t_2)=
\begin{pmatrix}
 1 & -\frac{n-1}{n(z-u_i(t_1,t_2))^2}+\frac{v^{(z)}_i(t_1,t_2)}{z-u_i(t_1,t_2)} & * & \cdots & * \\
 0 & 1 & -\frac{n-2}{n(z-u_i(t_1,t_2))} & \cdots & * \\
 0 & 0 & 1 & \cdots & * \\
 \vdots & \vdots & \vdots & \ddots & \vdots \\
 0 & 0 & 0 & \cdots & -\frac{1}{n(z-u_i(t_1,t_2))} \\
 0 & 0 & 0 & \cdots & 1
\end{pmatrix}.
\end{equation}

Therefore, evaluating the Atiyah--Bott form $\Omega_{AB}$ on the two
tangent vectors yields,
\begin{align*}
&\Omega_{AB} \left( (E,\nabla)_{t_1}, ( E,\nabla)_{t_2}\right)
= \Omega_{AB} \left( (\mathbb{P}(E),\nabla)_{t_1}, ( \mathbb{P}(E),\nabla)_{t_2}\right) \\
=& 
\sum_{i=1}^d \operatorname{Res}_{z=0} \operatorname{Tr}\left(
    \partial_{t_1} A_i \cdot \partial_{t_2} G_i
    \cdot G_i^{-1}
    - \partial_{t_2} A_i \cdot \partial_{t_1} G_i
    \cdot G_i^{-1}
\right) \Big|_{t_1=t_2=0} \mathrm{d}z \\
=& \sum_{i=1}^d \operatorname{Res}_{\lambda = 0} 
\left[
    -\frac{2(n-1)}{n z^2} \left( \partial_{t_1}u_i + \partial_{t_2}u_i \right)
    + 
    \frac{1}{z} \left( \partial_{t_2} v^{(z)}_i \partial_{t_1} u_i - \partial_{t_1} v^{(z)}_i \partial_{t_2} u_i \right)
    + 
    \mathrm{O}(1)
\right] d\lambda \\
=& \sum_{i=1}^d -\left(
    \partial_{t_1} v^{(z)}_i \cdot \partial_{t_2} u_i
    - \partial_{t_2} v^{(z)}_i \cdot \partial_{t_1} u_i
\right) \Big|_{t_1=t_2=0}\\
=& -\omega^{[d]} \left( \widetilde{D}_{t_1}, \widetilde{D}_{t_2} \right).
\end{align*}
Here $(E,\nabla)_{t_j}$ denotes the tangent vector
\(
(E,\nabla)_{t_j}
:=\partial_{t_j}(E(t_1,t_2),\nabla(t_1,t_2))\Big|_{t_1=t_2=0}.
\)
In the second equality we use \eqref{e:form of nabla check Ui} and
\eqref{e:form of Ni}, from which it follows that only the $2\times2$ minors
contribute to the trace.
In the last line, we write
\(
\widetilde D_{t_j}
:=\partial_{t_j}\widetilde D(t_1,t_2)\Big|_{t_1=t_2=0}.
\)

Therefore, this shows that $\mathbb L_{dR}(d)$ is isotropic, and hence gives a Lagrangian correspondence.
\end{proof}

\section{Examples and applications}
The aim of this last section is to provide explicit examples and applications of the theory developed and relate it to previous work. We hope that this section nicely complements the abstract and general nature of the rest of the paper. In the first subsection, by choosing appropriate degrees for the divisors, we construct Darboux coordinates on the Hitchin and de Rham moduli spaces. In the second subsection we relate our holomorphic Lagrangian subvarieties to the $\C^{\star}$-action on the Hodge moduli space and the conformal limit. In the final subsection, we give a detailed description of our Lagrangian subvarieties and Lagrangian correspondences in the rank $2$ cases. 

\subsection{Darboux coordinates on moduli spaces}
Recall that the dimensions of $\mathcal{M}_H(n, \Lambda)$ and $\mathcal{M}_{dR}(n, \mathcal{O}_C)$ are equal to $2 (n^2-1)(g-1)$. 

\begin{proposition}
(a) Let $\Lambda$ be a line bundle on $C$ and $d \equiv \ell = \deg(\Lambda) \pmod{n}$. 
Then $\mathbb{L}_H(d) \rightarrow \mathcal{M}_H(n, \Lambda)$ is dominant if and only if $d \geq (n^2-1)(g-1)$. 
\\
(b) Let $d \equiv 0 \pmod{n}$. 
Then $\mathbb{L}_{dR}(d) \rightarrow \mathcal{M}_{dR}(n, \mathcal{O}_C)$ is dominant if and only if $d \geq (n^2-1)(g-1)$. 
\end{proposition}

\begin{proof}[Proof idea]
This follows from \cite{yoshi97}. One can alternatively prove the statement via a direct Riemann-Roch calculation to show that a generic bundle $E \in \mathrm{Bun}_{n, \Lambda}$ admits a line subbundle $L$ such that $\deg(C, K_C^{n \choose 2} L^{-n} \Lambda) \geq (n^2-1)(g-1)$.
\end{proof}

\begin{corollary}
(a) For $\ell + g - 1 \equiv 0 \pmod{n}$, $\deg(\Lambda) = \ell$, the Lagrangian correspondence $\mathbb{L}_H((n^2-1)(g-1))$ provides Darboux coordinates on an open dense subset of $\mathcal{M}_H(n, \Lambda)$.
\\
(b) For $g \equiv 1 \pmod{n}$, apparent singularities and their residue parameters provide Darboux coordinates on an open dense subset of $\mathcal{M}_H(n, \mathcal{O}_C)$.
\end{corollary}

\subsection{Conformal limit and Bia\l ynicki-Birula stratification}
\label{sect-CL}

When the degree of the divisor is small enough the Lagrangians $\mathbb L_H(L,D)$ and $\mathbb L_{dR}(L,D)$ interact nicely with the conformal limit construction for Higgs bundles \cite{CW18}. Below, we briefly explain this interaction. 

Given a stable Higgs bundle $[(\bar\pa_E,\phi)]$ denote its $\C^{\star}$-limit by $[(\bar\pa_0,\phi_0)] := \lim\limits_{\xi\to 0}[(\bar\pa_E,\xi\phi)]$. Given such a fixed point we define 
\begin{equation}
W(\bar\pa_0,\phi_0):= \{[(\lambda,\bar\pa_E,\nabla_{\lambda})] \in \mathcal M_{Hod}~|~ \lim\limits_{\xi\to 0} [(\xi\lambda,\bar\pa_E,\xi\nabla_{\lambda})] = [(\bar\pa_0,\phi_0)]\}.
\end{equation}
Then it holds that $ W(\bar\pa_0,\phi_0)\cap \pi^{-1}(0) =: W^0(\bar\pa_0,\phi_0)$ and $ W(\bar\pa_0,\phi_0)\cap \pi^{-1}(1) =: W^1(\bar\pa_0,\phi_0)$ are holomorphic Lagrangian submanifolds of the Dolbeaut and De Rham moduli space respectively. Assume that the fixed point $[(\bar\pa_0,\phi_0)]$ is of type $(1,...,1)$ so that 
\begin{equation}\label{eqn-Hodge-bundle-form}
\begin{split}
E\cong& L_1\oplus\dots\oplus L_n\\
\phi_0 =& \begin{pmatrix} 0 & \dots & \dots & 0 \\ \Phi_1 & 0 & \dots & 0\\ \vdots & \ddots & \ddots & 0 \\ 0 & \dots & \Phi_{n-1} & 0 \end{pmatrix},
\end{split}
\end{equation}
with $\Phi_i \in H^0(C; L_i^{-1}L_{i+1} K)$. Denote $d_i = \deg L_i$ and $D_i$ the vanishing set of $\Phi_i$. The divisor $D$ we have studied throughout the paper is given by the weighted sum $\sum\limits_{i=1}^{n-1} (n-i)D_i$. We can also associate to the fixed point the divisor $\mathfrak D := \sum\limits_{i=1}^{n-1} D_i$ which is used in \cite{HH22}. Under the above assumption it can be inferred from \cite{CW18} that 
\begin{equation}
\mathbb L_H(L_1,D) = W^0(\bar\pa_0,\phi_0)~\text{ and }~\mathbb L_{dR}(L_1,D) = W^1(\bar\pa_0,\phi_0). 
\end{equation}
In the next proposition, given a $\C^{\star}$-fixed point $[(\bar\pa_0,\phi_0)]$ we give bounds on the possible degree of the divisors $D$ and $\mathfrak D$. 
\begin{proposition}\label{Dbounds}
(a) Consider a stable Hodge bundle $[(\bar\pa_0,\Phi_0)]$, and let $D$ and $\mathfrak D$ be divisors on $C$ defined by $[(\bar\pa_0,\Phi_0)]$ as above. If $\deg E \equiv \ell \pmod{n}$ then necessarily 
\begin{equation*}
\begin{split}
\deg(D) &\le n(n-1)(g-1)-(n-\ell),\\
\deg(\mathfrak D)&\le 2(n-1)(g-1) -1- \delta_{0,\ell}.
\end{split}
\end{equation*}
There exist divisors which realize equality in both inequalities. 
\\
(b) Let $D$ be a reduced effective divisor of degree $d \equiv \ell \pmod{n}$.
Then there exists a stable Hodge bundle $[(\bar\pa_0,\Phi_0)]$ of the form \eqref{eqn-Hodge-bundle-form} with  
\[\mathbb L_H(L_1,D) = W^0(\bar\pa_0,\phi_0) \]
if and only if 
\begin{equation*}
\deg(D) \le n(g-1)-(n-\ell). 
\end{equation*}
\\
(c) Let $D$ be a reduced effective divisor of degree $d \equiv 0 \pmod{n}$.
Then there exists a stable Hodge bundle $[(\bar\pa_0,\Phi_0)]$ of the form \eqref{eqn-Hodge-bundle-form} with 
\[\mathbb L_{dR}(L_1,D) = W^1(\bar\pa_0,\phi_0) \]
if and only if 
\begin{equation*}
\deg(D) \le  n(g-2).
\end{equation*}
\end{proposition}
\begin{proof}
The total degree of the divisor $D$ is given by 
\begin{equation}
\begin{split}
\deg(D) = \sum\limits_{i=1}^{n-1} (n-i)\deg(D_i) =& \sum\limits_{i=1}^{n-1}(n-i)(\deg K + \deg L_{i+1}-\deg L_i)\\ =& n(n-1)(g-1) + d_1+...+d_n-nd_1,
\end{split}
\end{equation}
and the total degree of $\mathfrak D$ is given by 
\begin{equation}
\begin{split}
\deg(\mathfrak D) = \sum\limits_{i=1}^{n-1}\deg(D_i) =& \sum\limits_{i=1}^{n-1}(\deg K + \deg L_{i+1}-\deg L_i)\\ =& 2(n-1)(g-1) + d_n-d_1. 
\end{split}
\end{equation}

The $\phi$-invariant subbundles of $E$ are precisely $E_k = \bigoplus\limits_{i=k}^n L_i$ and by stability we get the following inequalities: 
\begin{equation}\label{SS}
\frac{d_k+...+d_n}{n-k+1} < \frac{d_1+...+d_n}{n}.
\end{equation}
In particular, the inequality \eqref{SS} for $k =2$ gives the restriction $d_1+...+d_n < nd_1$. Since all the $d_i$ are integers and $d_1+...+d_n \equiv \ell \pmod{n}$, it holds that $d_1+...+d_n - nd_1\le -(n-\ell)$ and therefore $\deg(D) \le n(n-1)(g-1)-(n-\ell)$. Equality can be obtained by choosing for example $d_1=d_2=...=d_\ell = d$ and $d_{\ell+1}=... = d_n = d-1$. 

Regarding the divisor $\mathfrak D$, the inequality \eqref{SS} for $k=n$ gives $nd_n < d_1+...+d_n$. Combining with $d_1+...+d_n < nd_1$ we get that $d_n -d_1 \le -1$. If $d_1+...+d_n \neq 0 \mod n$ then equality can be obtained by choosing $d_1=d_2=...=d_k = d$ and $d_{k+1}=... = d_n = d-1$. If $d_1+...+d_n = 0 \mod n$ then $d_n -d_1 \le -2$ with equality obtained by choosing $d_1 = d+1$, $d_2=...=d_{n-1} = d$ and $d_n = d_1$.

Finally, suppose that $D$ is reduced. This forces all $\Phi_i$ to be constant sections of $\mathcal{O}_C$ for $i\neq n-1$ so that 
\begin{equation*}
L_{i+1} = L_iK^{-1}~\text{for}~i\le n-2~\text{and}~L_n = L_{n-1}K^{-1}(D).
\end{equation*}
This implies that $d_{i+1} = d_i -(2g-2)$ for $i\le n-2$ and $d_n = d_{n-1} -(2g-2) + d$ where $d = \deg(D)$. The only constraint on $D$ comes from the stability constraints \eqref{SS}. The most restrictive constraint is the one for $k=n$ which gives 
\begin{equation*}
d_n < d_n + (n-1)(g-1) - \frac{(n-1)d}{n}
\end{equation*}
which implies that $d \le n(g-1)-1$. Finally, since $d_1+...+d_n \equiv \ell \pmod{n}$ and $d_1+...+d_n = nd_n + n(n-1)(g-1) -(n-1)d = d \mod n$ we must have $d \le n(g-1) - (n-\ell)$. 
This proves both parts (b) and (c) of the proposition.
\end{proof}
\begin{remark}
When the divisor $\mathfrak D$ is reduced, one can recover $D$ from $\mathfrak D$ and vice versa. Neither inequality in \ref{Dbounds} is a sufficient condition in general. 
\end{remark}

A surprising fact is that there exists a natural map from the Dolbeaut to the De Rham space which identifies these Lagrangians biholomorphically. In order to define it, let $R>0$ and let $h_R$ denote the harmonic metric associated to $[(\bar\pa_E,R\Phi)]$. Then consider the family of complex flat connections 
\begin{equation*}
\nabla_{R,\hbar} := \bar\pa_E + \pa_E^{\dagger_{h_R}} + \hbar R^2 \Phi^{\dagger_{h_R}} + \hbar^{-1}\Phi.
\end{equation*}
The limit of this family as $R\to 0$ exists and is given by the explicit formula 
\begin{equation*}
\begin{split}
\mathcal {CL}_{\hbar}(\bar\pa_E,\Phi):&= \lim\limits_{R\to 0} \nabla_{R,\hbar}\\
&= \bar\pa_E + \pa_0^{\dagger_{h_0}} + \hbar \Phi_0^{\dagger_{h_0}} + \hbar^{-1}\Phi.
\end{split}
\end{equation*}
\begin{theorem}
The $\hbar$-conformal limit map $\mathcal {CL}_{\hbar}$ maps $W^0(\bar\pa_0,\Phi_0)$ biholomorphically to $W^1(\bar\pa_0,\Phi_0)$.
\end{theorem}
This theorem enables among other things the study of the geometry of $\mathbb L_{dR}(D)$ through the use of coordinates on the Dolbeaut moduli space and was crucial in the proof of the following conjecture 
\begin{conjecture}[Simpson's conjecture] 
The Lagrangians $W^1(\bar\pa_0,\Phi_0)$ are complete subvarieties of $M_{dR}$
\end{conjecture}

\vspace{5pt}
\noindent
in rank $2$ by the first author \cite{DS24}. It would be interesting to know if there exists a map similar to the conformal limit which relates $\mathbb L_H(D)$ to $\mathbb L_{dR}(D)$ for arbitrary divisor $D$. It is also tempting to ask whether the Lagrangians $\mathbb L_{dR}(D)$ are closed for arbitrary divisor $D$ and arbitrary rank $n$.

\subsection{Properties of Lagrangian subvarieties and Lagrangian correspondences in rank-$2$ cases}
\label{sect-Lag-rk2}

We now discuss properties of the Lagrangians $\mathbb{L}_H(D)$ and $\mathbb{L}_{dR}(D)$ in the case of $\mathrm{SL}_2$ Higgs bundles and holomorphic connections.

\paragraph{Projections to the spaces of Hodge bundles}
Consider the space of Hodge bundles
\[ \mathrm{Hod}_2 = \mathcal{M}_H(2, \mathcal{O}_C)^{\mathbb{C}^\ast} \cong \mathrm{Bun}_{2, \mathcal{O}_C} \sqcup \left( \sqcup_{i=0}^{g-1} X_i \right), \]
where $X_i$, for $0 \leq i \leq g-1$, is the connected component parametrizing Hodge bundles of the form 
\[\mathcal{E}_L = \left( L \oplus L^{-1}, \left( \begin{smallmatrix} 0 & 0 \\ s & 0\end{smallmatrix}\right)\right), \qquad \deg(L) = g-1-i. \]  
Given an effective divisor $D$ of even degree $d$, denote by 
\begin{equation}\label{eqn-graded-pieces}
\mathrm{gr}_H: \mathbb{L}_H(D) \rightarrow \mathrm{Hod}_2, \qquad \qquad \mathrm{gr}_{dR}: \mathbb{L}_{dR}(D) \rightarrow \mathrm{Hod}_2, 
\end{equation}
the map defined by taking graded pieces, or equivalently by taking the limit $\underset{t \rightarrow 0}{\lim} t. [E, \phi] \in \mathrm{Hod}_2$.
Note that the composition of \eqref{eqn-graded-pieces} with the projection to $\mathrm{Bun}_{2, \mathcal{O}_C}$ is the rational forgetful map $(E, \phi) \mapsto E$ for $E$ stable.

Consider the moduli spaces $\mathcal{M}_H(2, \mathcal{O}_C)$ and $\mathcal{M}_{dR}(2, \mathcal{O}_C)$, and the respective Lagrangians $\mathbb{L}_H(L, D)$ and $\mathbb{L}_{dR}(L, D)$ for some $L$ satisfying $L^2 \cong K_C(-D)$. 
Our following results concern properties of these Lagrangian subvarieties as $d = \deg(D)$ varies.

\begin{proposition} \label{prop-rank-2-all-range}
Let $D$ be an effective divisor on $C$ of even degree $d$, and $L$ a line bundle with $L^2 \cong K_C(-D)$. \\
(a) Let $d \in [1, 2g-2)$. Then there exists a unique Hodge bundle 
\[\mathcal{E}_L = \left( L \oplus L^{-1}, \left( \begin{smallmatrix} 0 & 0 \\ s_D & 0\end{smallmatrix}\right)\right) \in X_{d/2}, \]    
such that 
\[ \mathbb{L}_H(L, D) = W^0(\mathcal{E}), \qquad \qquad \mathbb{L}_{dR}(L, D) = W^1(\mathcal{E}). \]
For each $\hbar\in \mathbb{C}^\ast$, the $\hbar$-conformal limit defines a biholomorphism $\mathbb{L}_H(L, D) \overset{\simeq}{\rightarrow} \mathbb{L}_{dR}(L, D)$. Furthermore, $\mathbb{L}_H(L, D)$ is a closed subvariety if and only if $D$ is reduced (i.e. $\mathcal{E}_L$ is a very stable Higgs bundle), while $\mathbb{L}_{dR}(L, D)$ is closed for any $D$. \\
(b) Let $d \in (2g-2, 3g-3]$ and $D$ be generic. 
Let $\mathrm{Bun}_{H}(L, D)$ (respectively, $\mathrm{Bun}_{dR}(L, D)$) be respectively the image along the rational forgetful map $H^1(C, L^2) \dashrightarrow \mathrm{Bun}_{2, \mathcal{O}_C}$ of     
\[ \ker(s_D) = \{[E] \in H^1(C, L^2) \mid \langle [E], s_D \rangle = 0 \},  \]
(respectively, the complement of $\ker(s_D)$ in $H^1(C, L^2)$).
Then there exists an open dense subset of $\mathbb{L}_H(L, D)$ (respectively, $\mathbb{L}_{dR}(L, D)$) that is dominant over $\mathrm{Bun}_{H}(L, D)$ (respectively, $\mathrm{Bun}_{dR}(L, D)$).  \\
(c) Let $d > 8g-8$, or $D$ be generic with $d \in [4g-3, 8g-8]$. Then $\mathbb{L}_H(L, D)$ projects to a divisor on $\mathrm{Bun}_{2, \mathcal{O}_C}$; $\mathbb{L}_{dR}(L, D)$ dominates $\mathrm{Bun}_{2, \mathcal{O}_C}$ and in particular is a multi-valued rational section of the rational forgetful map $\mathcal{M}_{dR}(2, \mathcal{O}_C) \dashrightarrow \mathrm{Bun}_{2, \mathcal{O}_C}$.       
\end{proposition}

\paragraph{Lagrangians from wobbly bundles}
\begin{proposition} \label{prop-wobbly-lag}
Let $D$ be an effective divisor of degree $2g-2 < d = \deg(D) \leq 4g-4$, and suppose that there exists a quadratic differential $q \in H^0(C, K_C^2)$ such that $\mathrm{div}(q) - D$ is effective. Then there exists an open dense subvariety $U_D \subset \mathbb{L}_H(D)$ such that 
\[ U_D = \{ (E,\phi) \in \mathbb{L}_H(D) \mid E \text{ is wobbly}\} \]
Such wobbly bundles $E$ admit nonzero nilpotent Higgs fields that vanish at $\mathrm{div}(q) - D$.
\end{proposition}

\paragraph{Fibers of Lagrangian correspondences}
We now concern ourselves with the properties of the fibers of the forgetful maps 
\[ \mathbb{L}_H(d) \overset{\mathrm{pr}_2}{\longrightarrow} \mathcal{M}_H^{s}(2, \Lambda), \qquad \qquad 
\mathbb{L}_{dR}(d) \overset{\mathrm{pr}_2'}{\longrightarrow} \mathcal{M}_H^{s}(2, \mathcal{O}_C). \]
The following results hold for both $\mathrm{pr}_2$ and $\mathrm{pr}_2'$; for simplicity, we only state the results concerning $\mathrm{pr}_2$. 
Let us denote by $\mathcal{M}_H^{vs}(2, \Lambda) \subset \mathcal{M}_H(2, \Lambda)$ the open dense subvariety parametrizing Higgs bundles $(E, \phi)$ with $E$ very stable. 
Let $\mathcal{M}_H^{s}(2, \Lambda) \equiv T^\ast\mathrm{Bun}_{2, \Lambda}$ denote the subvariety parametrizing $(E, \phi)$ with $E$ stable.

\begin{proposition} \label{prop-fiber-Lag-cor}
Consider the moduli space $\mathcal{M}_H(2, \Lambda)$ and let $d \equiv \ell = \deg(\Lambda) \pmod{2}$. 
\\
(a) Let $0 < d < 2g-2$. Then the fiber of $\mathrm{pr}_2$ over any $(E, \phi) \in \mathrm{im}(\mathrm{pr}_2)$ consists of exactly $1$ point. Consequently, $\mathrm{im}(\mathrm{pr}_2)$ fibers over $\mathrm{Hilb}^d(C)$ with fibers over $D$ being $\mathbb{L}_H(D)$.
\\
(b) Let $2g-2 < d \leq 3g-4$. Then the fiber of $\mathrm{pr}_2$ over a generic $(E, \phi) \in \mathrm{im}(\mathrm{pr}_2)$ consists of exactly $1$ point. Consequently, there exists an open dense subset of $\mathrm{im}(\mathrm{pr}_2)$ that fibers over $\mathrm{Hilb}^d(C)$ with fibers over $D$ being $\mathbb{L}_H(D)$.
\\
(c) Let $g \equiv \ell-1 \pmod{2}$ and $d = 3g-3$. Then the fiber of $\mathrm{pr}_2$ over any $(E,\phi) \in \mathcal{M}_H^{vs}(2, \Lambda)$ consists of exactly $2^g$ points.
\\ 
(d) Let $g \equiv \ell \pmod{2}$ and $d = 3g-2$. Then the fiber of $\mathrm{pr}_2$ over a generic $(E,\phi) \in \mathcal{M}_H^{vs}(2, \Lambda)$ is $1$-dimensional.
\\
(e) Let $D$ be an effective divisor of degree $d = 4g-3$. Then the fiber of $\mathrm{pr}_2$ over any $(E,\phi) \in \mathcal{M}_H^{s}(2, \Lambda)$ is smooth, irreducible and has dimension $g$. 
\end{proposition}

\paragraph{Proofs of Propositions \ref{prop-rank-2-all-range},  \ref{prop-wobbly-lag} and \ref{prop-fiber-Lag-cor}}
We will need the following results to prove Proposition \ref{prop-rank-2-all-range}.
\begin{lemma}\label{lem-Serre-pairing}
(a) For a rank-2 Higgs bundle connection $(E, \phi)$, $\det(E) = \phi$, and a line subbundle $L \overset{i}{\hookrightarrow} E$, the Serre duality pairing of the induced extension class $[E]_i \in H^1(C, L^2 \Lambda^{-1})$ and $s_i(\nabla_\lambda) \in H^0(C, K_C L^{-2} \Lambda)$ vanishes,
\begin{equation}
    \big\langle [E]_i, s_i(\nabla_\lambda) \big\rangle = 0.
\end{equation}
\\
(b) For a rank-2 $\lambda$-connection $(E, \nabla_\lambda)$, $\det(E) = \mathcal{O}_C$, and a line subbundle $L \overset{i}{\hookrightarrow} E$, the Serre duality pairing of the induced extension class $[E]_i \in H^1(C, L^2)$ and $s_i(\nabla_\lambda) \in H^0(C, K_C L^{-2})$ evaluates to 
\begin{equation}
    \big\langle [E]_i, s_i(\nabla_\lambda) \big\rangle = \lambda \deg(L).
\end{equation}
\end{lemma}
\begin{proof}
    See \cite{DT23, D24-lambda}. 
\end{proof}

\begin{lemma}\label{lemma-im-s_i}
Let $E$ be a stable rank-2 bundle of determinant $\det(E) \cong \Lambda$, $\deg(\Lambda) = \ell$, with line subbundle $L \overset{i}{\hookrightarrow} E$, and consider the map 
\begin{equation}\label{s_i}
s_i: H^0(C, \mathrm{End}_0(E) \otimes K_C) \longrightarrow H^0(C, K_C L^{-2}\Lambda), 
\qquad \qquad \phi \longmapsto s_i(\phi).     
\end{equation}
Let $d = \ell - 2\deg(L) + 2g-2$. Then: \\
(a) The codimension of $\mathrm{im}(s_i)$ is $h^0(C, L^{-1}E)$. \\
(b) For generic $E \in \mathrm{Bun}_{2, \Lambda}$ and generic $L$ with $d\geq 4g-4$, $s_i$ is injective. \\
(c) For $d > 8g-8$, $s_i$ is injective.
\end{lemma}
\begin{proof}
The map $s_i$ fits in the long exact sequence induced from 
\begin{equation}\label{eqn-s.e.s-1}
0 \longrightarrow E^\vee L K_C 
\longrightarrow \mathrm{End}_0(E) \otimes K_C \longrightarrow K_C L^{-2} \Lambda  
\longrightarrow 0,    
\end{equation}
where $E^\ast L K_C$ is the bundle of Higgs fields $\phi$ that keep $L$ invariant, i.e. $\phi(L) \subset L\otimes K_C$. A Riemann-Roch calculation shows that $h^0(C, E^\vee L K_C) = 4g-4 - d + h^0(C, L^{-1}E)$. It follows that, for $E$ stable, the dimension of $\mathrm{im}(s_i)$ is
\begin{align*}
    h^0(C, \mathrm{End}_0(E) \otimes K_C) - h^0(C, E^\vee L K_C) &= d - g +1 - h^0(C, L^{-1}E) \\ &= h^0(C, K_C L^{-2}\Lambda) - h^0(C, L^{-1}E).
\end{align*}
This proves part (a) of the Lemma. 
For parts (b) and (c), observe that for $(E, \phi)$ a Higgs bundle with smooth spectral curve $\widetilde{C}$ and $\tilde{D} = \tilde{D}_i(\phi)$,
\[ h^0(C, L^{-1} E) = h^0(\widetilde{C}, \mathcal{O}_{\widetilde{C}}(\tilde{D})) = d - 4g + 4 + h^1(\widetilde{C}, \mathcal{O}_{\widetilde{C}}(\tilde{D})). \]
One then can check that
\begin{align}\label{eqn-h1-Dtilde}
    h^0(C, E^\vee L K_C) &= 2g - 2 + 2 \deg(L) + h^0(L^{-1} E) = h^1(\widetilde{C}, \mathcal{O}_{\widetilde{C}}(\tilde{D})) \nonumber \\
    &= \begin{cases}
    4g - 4 - d &\text{for generic } \tilde{D}, 0 \leq d \leq 4g-4 \\
    0 &\text{for generic } \tilde{D}, 4g-4 \leq d \leq 8g-8  \\
    0 &\text{for } d > 8g-8
    \end{cases} \quad .
\end{align}
This proves part $(c)$.
Note that the spectral line bundle $\mathcal{L}_{(E, \phi)} \cong \pi^\ast(L)(\tilde{D})$, where $L^2 \cong K_C \Lambda(-\pi_\ast(\tilde{D}))$, 
defines a point in the subvariety $\mathrm{Prym}^{2g-2+\ell}_\Lambda(\widetilde{C}) \subset \mathrm{Pic}^{2g-2+\ell}(\widetilde{C})$. This subvariety parametrizes line bundles $\mathcal{L}$ on $\widetilde{C}$ satisfying $\mathcal{L} \otimes \sigma^\ast \mathcal{L} \cong \pi^\ast(K_C \Lambda)$, which is equivalent to the condition $\det(\pi_\ast \mathcal{L}) \cong \Lambda$. 
Via the assignment 
\[ \tilde{D} \longmapsto \pi^\ast(L')(\tilde{D}'), \qquad \qquad (L')^2 \cong K_C \Lambda(-\pi_\ast(\tilde{D})), \]
divisors $\tilde{D} \in \widetilde{C}^{2g-2+\ell}$ determine spectral line bundles up to $2^{2g}$ square-roots of $\mathcal{O}_C$. Since the direct image map $\pi^\ast: \mathrm{Prym}^{2g-2+\ell}_\Lambda(\widetilde{C}) \rightarrow \mathrm{Bun}_{2, \Lambda}$ is generically finite (and in particular finite over the locus of very stable bundles), we can use the estimate \eqref{eqn-h1-Dtilde} for generic bundles $E \in \mathrm{Bun}_{2, \Lambda}$. This proves part (b).
\end{proof}

\begin{proof}[Proof of \Cref{prop-rank-2-all-range}]
Part (a) follows from our discussion on the conformal limit in \Cref{sect-CL} \cite{CW18} and the proof of Simpson conjecture for $\mathrm{SL}_2$ cases \cite{DS24}. 

For part (b), \Cref{lem-Serre-pairing}(a) implies that the projection of $\mathbb{L}_H(L, D)$ to $\mathrm{Bun}_{2, \mathcal{O}_C}$ is contained in $\mathrm{Bun}_{H}(L, D)$.
For the other direction, let $E \in \mathrm{Bun}_{2, \mathcal{O}_C}$ defined by a generic extension class $[E]_i \in H^1(C, L^2)$ via an embedding $L \overset{i}{\hookrightarrow} E$. 
For $2g-2 < d \leq 3g-3$ and $D$ generic, it follows from the work of Lange-Narasimhan \cite{LN83} on secant varieties of $H^1(C, L^2)$ that $E$ admits $L$ as a maximal line subbundle $\deg(L) = \underset{L' \hookrightarrow E}{\max} \{\deg(L') \}$. 
In this case, $h^0(C, L^{-1}E) = 1$. It now follows from \Cref{lem-Serre-pairing}(a) and \Cref{lemma-im-s_i}(a) that the image of the map $s_i$ is precisely the hyperplane $\ker(i) = \{ s\in H^0(C, K_C L^{-2}) \mid \langle [E]_i, s \rangle = 0 \}$. 
We then see that if $[E]_i \in \ker(s_D)$ then $E$ admits Higgs fields $\phi$ with $s_i(\phi) = s_D$, i.e. $(E, \phi) \in \mathbb{L}_H(L, D)$.

As for $\mathbb{L}_{dR}(L, D)$, it follows from \Cref{lem-Serre-pairing}(b) that the projection of $\mathbb{L}_{dR}(L, D)$ to $\mathrm{Bun}_{2, \mathcal{O}_C}$ is contained in $\mathrm{Bun}_{dR}(L, D)$. For the other direction, let $E \in \mathrm{Bun}_{2, \mathcal{O}_C}$ be realized by an extension class $[E]_i$ in the complement of $\ker(s_D)$. 
Let $\nabla_{\mathrm{NS}}$ be the holomorphic connection on $E$ corresponding to the Narasimhan-Seshadri connection, i.e. the non-abelian Hodge correspondence of $(E, 0)$. 
By \Cref{lem-Serre-pairing}(b), $s_i(\nabla_{\mathrm{NS}})$ is not contained in the hyperplane $\ker(s_D)$, and so we can write $s_D = s_i(\nabla_{\mathrm{NS}}) + s'$ for some $s' \in \ker(s_D)$. 
Then $E$ admits a holomorphic connection $\nabla = \nabla_{\mathrm{NS}} + \phi$ where $s_i(\phi) = s$ such that $s_i(\nabla) = s_D$, i.e. $(E, \nabla) \in \mathbb{L}_{dR}(L, D)$.
This completes the proof of part (b).

For part (c), it follows from \Cref{p: generic finite dR} that the kernel of the differential at $(E, \nabla) \in \mathbb{L}_{dR}(L, D)$ of the projection to $\mathrm{Bun}_{2, \mathcal{O}_C}$ is generated by $H^0(C, E^\vee L K_C)$. 
The statement for $\mathbb{L}_{dR}(L, D)$ now follows from \eqref{eqn-h1-Dtilde}.
On the other hand, since $\mathbb{L}_{H}(L, D)$ is invariant with respect to the $\mathbb{C}^\ast$ action, the closure of its projection to $\mathrm{Bun}_{2, \mathcal{O}_C}$ is a divisor.
\end{proof}

\begin{proof}[Proof idea of \Cref{prop-wobbly-lag}]
This follows from \cite{D24-wobbly}. The idea is that a nonzero nilpotent Higgs fields $\psi$ on $E$ with $\ker(\psi) = (L \overset{i}{\hookrightarrow} E)$ is equivalent to a nonzero section of $K_C L^2 \Lambda^{-1}$. Suppose $(E, \phi) \in \mathbb{L}_H(D)$ and there exist effective divisors $D'$ and a quadratic differential $q \in H^0(C, K_C^2)$ such that $D_i(\phi) + D' = \mathrm{div}(q)$. 
Then $E$ admits a nonzero nilpotent Higgs field $\psi$ that corresponds to a section of $\mathcal{O}_C(D') \cong K_C L^2 \Lambda^{-1}$ and vanishes precisely at $D'$.
\end{proof}

\begin{proof}[Proof of \Cref{prop-fiber-Lag-cor}]
The fiber of $\mathrm{pr}_2$ over a Higgs bundle $(E, \phi) \in \mathrm{im}(\mathrm{pr}_2)$ is precisely the scheme 
\[ S_E^d = \left\{ (L, [i]) \mid L \in \mathrm{Pic}^{g-1+(\ell -d)/2}(C), [i] \in \mathbb{P}H^0(C, L^{-1} E) \right\}. \] 
Noting that a maximal line subbundle $L$ of $E$ has a unique embedding up to scaling, i.e. $h^0(C, L^{-1} E) = 1$, we obtain parts (a) -- (d) of the proposition using known results counting maximal line subbundles of rank-2 bundles.
Specifically:
\begin{itemize}
\item If $d < 2g-2$ then $E$ is unstable, and hence $E$ has a unique maximal, destabilizing line subbundle $L$ with unique embedding up to scaling. This proves part (a). 
\item If $0 \leq d < 3g-3$, then the underlying bundle $E$ of a \textit{generic} $(E, \phi) \in \mathrm{im}(\mathrm{pr}_2)$ is stable but not maximally stable, i.e. it is contained in the complement in $\mathrm{Bun}_{2, \Lambda}$ of the top open Segre stratum. In this case $E$ admits a maximal subbundle $L$ of degree $g-1+(\ell -d)/2$. Lange-Narasimhan showed that a generic stable but not maximally stable bundle $E$ admits a unique maximal line subbundle \cite{LN83}. This proves part (b).
\item If $g \equiv \ell -1$ and $E$ is a very stable bundle, then its maximal line subbundles are of degree $g-1+(\ell -d)/2$. It is well-known that a very stable bundle in this case has exactly $2^g$ maximal subbundles; see Gronow's PhD thesis \cite{Gronow97}. This proves part (c).
\item If $g \equiv \ell$ and a generic $E \in \mathrm{Bun}_{2, \Lambda}$ admits a $1$-dimensional family of maximal subbundles of degree $g-1+(\ell -d)/2$ \cite{LN83}. This proves part (d).
\end{itemize}
Part (e), on the other hand, follows directly from part of a conjecture of Kazhdan-Polishchuk \cite{BravermanKazhdan2023, KazhdanPolishchuk2024} that was proved by Debarre \cite{Debarre2024}.
\end{proof}

\printbibliography

\vspace{7pt}
\noindent
Panagiotis Dimakis, University of Maryland, College Park, MD, USA.
\textit{pdimakis12345@gmail.com}

\vspace{7pt}
\noindent 
\dqd, University of Pennsylvania, Philadelphia, PA, USA. \\
University of Khanh Hoa, Nha Trang, Khanh Hoa, Vietnam.
\textit{duongdinh.mp@gmail.com}

\vspace{7pt}
\noindent
Shengjing Xu, University of Pennsylvania, Philadelphia, PA, USA. 
\textit{xsj@upenn.edu}

\end{document}